\documentclass[10pt]{amsart}
\usepackage{amssymb}
\usepackage{amscd}
\usepackage[all]{xy}
\usepackage[colorlinks=true, linkcolor=red,
urlcolor=blue]{hyperref}

\newcounter{TmpEnumi}
\numberwithin{equation}{section}

\def\today{\number\day\space\ifcase\month\or
 January\or February\or
   March\or April\or May\or June\or
    July\or August\or September\or
   October\or November\or December\fi\
     \number\year}

\theoremstyle{definition}
\newtheorem{thm}{Theorem}[section]
\newtheorem{lem}[thm]{Lemma}
\newtheorem{prp}[thm]{Proposition}
\newtheorem{dfn}[thm]{Definition}
\newtheorem{cor}[thm]{Corollary}

\newtheorem{ctn}[thm]{Construction}
\newtheorem{rmk}[thm]{Remark}
\newtheorem{ntn}[thm]{Notation}
\newtheorem{exa}[thm]{Example}

\newtheorem{qst}[thm]{Question}

\newcommand{\beq}{\begin{equation}}
\newcommand{\eeq}{\end{equation}}
\newcommand{\beqr}{\begin{eqnarray*}}
\newcommand{\eeqr}{\end{eqnarray*}}
\newcommand{\bal}{\begin{align*}}
\newcommand{\eal}{\end{align*}}
\newcommand{\bei}{\begin{itemize}}
\newcommand{\eei}{\end{itemize}}
\newcommand{\limi}[1]{\lim_{{#1} \to \infty}}

\newcommand{\af}{\alpha}
\newcommand{\bt}{\beta}
\newcommand{\gm}{\gamma}
\newcommand{\dt}{\delta}
\newcommand{\ep}{\varepsilon}

\newcommand{\et}{\eta}

\newcommand{\io}{\iota}

\newcommand{\ld}{\lambda}
\newcommand{\sm}{\sigma}
\newcommand{\kp}{\kappa}
\newcommand{\ph}{\varphi}
\newcommand{\ps}{\psi}
\newcommand{\rh}{\rho}

\newcommand{\ta}{\tau}

\newcommand{\Z}{{\mathbb{Z}}}
\newcommand{\R}{{\mathbb{R}}}
\newcommand{\C}{{\mathbb{C}}}
\newcommand{\N}{{\mathbb{Z}}_{> 0}}
\newcommand{\Nz}{{\mathbb{Z}}_{\geq 0}}

\newcommand{\Cu}{{\operatorname{Cu}}}

\newcommand{\id}{{\operatorname{id}}}
\newcommand{\ev}{{\operatorname{ev}}}

\newcommand{\spec}{{\operatorname{sp}}}

\newcommand{\diag}{{\operatorname{diag}}}
\newcommand{\supp}{{\operatorname{supp}}}
\newcommand{\rank}{{\operatorname{rank}}}

\newcommand{\card}{{\operatorname{card}}}
\newcommand{\Aut}{{\operatorname{Aut}}}

\newcommand{\QT}{{\operatorname{QT}}}
\newcommand{\T}{{\operatorname{T}}}
\newcommand{\W}{{\operatorname{W}}}
\newcommand{\tr}{{\operatorname{tr}}}

\newcommand{\rc}{{\operatorname{rc}}}

\newcommand{\dirlim}{\varinjlim}

\newcommand{\Mi}{M_{\infty}}

\newcommand{\andeqn}{\qquad {\mbox{and}} \qquad}

\newcommand{\wolog}{without loss of generality}
\newcommand{\Wolog}{Without loss of generality}

\newcommand{\tfae}{the following are equivalent}
\newcommand{\ifo}{if and only if}

\newcommand{\ca}{C*-algebra}
\newcommand{\uca}{unital C*-algebra}

\newcommand{\ssuc}{stably finite simple unital C*-algebra}

\newcommand{\hm}{homomorphism}

\newcommand{\fd}{finite-dimensional}

\newcommand{\pj}{projection}

\newcommand{\CGAa}{C^* (G, A, \af)}

\newcommand{\ct}{continuous}
\newcommand{\cfn}{continuous function}

\newcommand{\cms}{compact metric space}

\newcommand{\mh}{minimal homeomorphism}

\newcommand{\I}{\infty}
\newcommand{\E}{\varnothing}
\newcommand{\Lem}[1]{Lemma~\ref{#1}}
\newcommand{\Def}[1]{Definition~\ref{#1}}
\newcommand{\Thm}[1]{Theorem~\ref{#1}}
\newcommand{\Prp}[1]{Proposition~\ref{#1}}

\newcommand{\Cor}[1]{Corollary~\ref{#1}}
\newcommand{\Rmk}[1]{Remark~\ref{#1}}

\newcommand{\Ntn}[1]{Notation~\ref{#1}}

\newcommand{\Ctn}[1]{Construction~\ref{#1}}

\title[Radius of comparison of the crossed product]{The
  Cuntz semigroup and the radius of comparison of the crossed product by
  a finite group}

\author{M.~Ali Asadi-Vasfi}

\address{Department of Mathematics, University of Oregon,
     Eugene OR 97403-1222, USA.}

\email[]{Aliasadi@uoregon.edu}

\curraddr{School of Mathematics, Statistics and Computer Science,
College of Science, University of Tehran, Tehran, Iran.}

\email[]{Asadi.ali@ut.ac.ir}

\author{Nasser Golestani}

\address{Department of Pure Mathematics,
 Faculty of Mathematical Sciences,
 Tarbiat Modares University, P.O.\  Box 14115--134, Tehran, Iran}

\email[]{n.golestani@modares.ac.ir}

\author{N.~Christopher Phillips}

\address{Department of Mathematics, University  of Oregon,
       Eugene OR 97403-1222, USA.}

\date{17~August 2019}


\begin{document}

\begin{abstract}
Let $G$ be a finite group,
let $A$ be
an infinite-dimensional stably finite simple unital C*-algebra,
and let $\alpha \colon G \to \Aut (A)$
be an action of $G$ on $A$ which has the
weak tracial Rokhlin property.
Let $A^{\af}$ be the fixed point algebra.
Then the radius of comparison satisfies
$\rc (A^{\alpha}) \leq \rc (A)$ and
$\rc \big( \CGAa \big) \leq \frac{1}{\card (G)} \cdot \rc (A)$.
The inclusion of $A^{\alpha}$ in~$A$
induces an isomorphism from the purely positive part of the
Cuntz semigroup $\Cu (A^{\alpha})$ to the fixed points
of the purely positive part of $\Cu (A)$,
and the purely positive part of
$\Cu \big( \CGAa \big)$ is isomorphic to this semigroup.
We construct an example
in which $G = \Z / 2 \Z$,
$A$ is a simple unital AH~algebra,
$\alpha$ has the Rokhlin property,
$\rc (A) > 0$, $\rc (A^{\af}) = \rc (A)$,
and $\rc (\CGAa) = \frac{1}{2} \rc (A)$.
\end{abstract}

\maketitle
\tableofcontents


\section{Introduction}\label{Sec_Intro}

We prove that if $G$ is a finite group,
$A$ is an infinite-dimensional stably finite simple unital C*-algebra,
and $\alpha \colon G \to \Aut (A)$
is an action of $G$ on $A$ which
has the weak tracial Rokhlin property,
then the radii of comparison
(see below for further discussion) of $A$,
the crossed product,
and the fixed point algebra
are related by
\[
\rc (A^{\alpha}) \leq \rc (A)
\andeqn
\rc \bigl( \CGAa \bigr) \leq \frac{1}{\card (G)} \cdot \rc (A).
\]
See Theorem~\ref{Main.Thm1}
and Theorem~\ref{T_9412_RcCrPrd}.
These inequalities fail
for general pointwise outer actions;
see Example~\ref{R_9412_TrRPNeeded}.

In fact,
we prove a much stronger result,
relating the Cuntz semigroups
(see below and Section~\ref{Sec_Cu} for further discussion):
the inclusion of $A^{\alpha}$ in~$A$
induces an isomorphism from the subsemigroup
$\Cu_{+} (A^{\alpha}) \subseteq \Cu (A^{\alpha})$
consisting of zero and the
purely positive elements
(recalled in Definition~\ref{D_9421_Pure})
to the fixed points of $\Cu_{+} (A)$
under the action induced by~$\af$.
By Example~\ref{E_9421_NotOnK0},
the restriction to the purely positive part is necessary.
If $A$ has stable rank one,
then one can use $\W (A)$ in place of $\Cu (A)$.
We consider this to be a striking result,
since the Cuntz semigroup is often considered to be too complicated
to compute for \ca{s} without strict comparison.

We further give an example of
an infinite-dimensional stably finite simple unital C*-algebra~$A$
and an action $\af \colon \Z / 2 \Z \to \Aut (A)$
which even has the Rokhlin property,
and for which
the radii of comparison of $A$,
the crossed product,
and the fixed point algebra are all strictly positive.
It is initially not obvious that such an example should exist.
The algebra~$A$ even has stable rank one.

Along the way, we estimate
(Theorem~\ref{Ourcornertheorem})
the radius of comparison of a corner
of a simple unital C*-algebra.
This result is surely known,
but we have not found it in the literature.
We also prove (Lemma~\ref{proj.quasitrace})
that if $A$ is simple and unital,
$G$ has order~$n$,
and $\af \colon G \to \Aut (A)$
has the weak tracial Rokhlin property,
then every quasitrace on $\CGAa$
takes the value $\frac{1}{n}$
on the average of the unitaries in the crossed product
which correspond to the  elements of~$G$.

The importance of the Cuntz semigroup
has become apparent
in work related to the Elliott classification program.
See~\cite{APT11}
for a survey of many aspects of the Cuntz semigroup.
It is generally large and complicated;
roughly speaking,
among simple nuclear \ca{s},
the classifiable ones are those whose Cuntz semigroups
are easily accessible.
With the near completion of the Elliott program,
attention is turning to nonclassifiable \ca{s},
and the Cuntz semigroup is the main additional available invariant.
Given its complexity,
it is somewhat surprising that there is such a strong connection
between the Cuntz semigroup of a simple \ca{}
and the Cuntz semigroup of its crossed product
by a weak tracial Rokhlin action.
It seems, also by comparison with~\cite{Ph14},
that the purely positive part of the Cuntz semigroup
does not see differences which are ``small in trace''.

The radius of comparison
is a numerical invariant, based on the Cuntz semigroup,
which was introduced
in Section~6 of~\cite{Tom06}
to distinguish examples of nonisomorphic simple separable
unital AH~algebras
with the same Elliott invariant.
Its importance goes well beyond this application.
For example,
it is now conjectured that if $h$ is a \mh{}
of a \cms~$X$,
then $\rc \bigl( C^* (\Z, X, h) \bigr)$
is equal to half the mean dimension of~$h$;
mean dimension is an invariant introduced in dynamics
which at the time had no apparent connection with C*-algebras.
The radius of comparison also plays a key role
in a recent example
of a simple separable unital AH~algebra
whose Elliott invariant has an automorphism
not implemented by any automorphism of the algebra~\cite{HP19}.

The weak tracial Rokhlin property
(Definition 2.2 of~\cite{GHS17};
see Definition~\ref{W_T_R_P_def} below)
is a generalization of the tracial Rokhlin property
(Definition~1.2 of~\cite{PhT1})
which uses positive elements instead of projections.
It is a slight modification
of the generalized tracial Rokhlin property
of Definition 5.2 of~\cite{HO13}.
It is much more common than the Rokhlin property,
any of the
higher dimensional Rokhlin properties with commuting towers,
or even the tracial Rokhlin property.
See Example 3.12 of~\cite{Ph_FrSrv}
for a collection of examples of actions of finite groups
which have the tracial Rokhlin property
(and hence the weak tracial Rokhlin property)
but not the Rokhlin property.
It is shown in~\cite{AGJP17}
that if $A$ is a simple \ca{}
which is tracially ${\mathcal{Z}}$-absorbing,
and if the minimal tensor product $A^{\otimes n}$
of $n$ copies of~$A$ is finite,
then the permutation action of $S_n$ on $A^{\otimes n}$
has the weak tracial Rokhlin property.
Using ${\mathcal{Z}}^{\otimes n} \cong {\mathcal{Z}}$,
one gets in particular an action of $S_n$ on ${\mathcal{Z}}$
which has the weak tracial Rokhlin property.
(This was proved for the generalized tracial Rokhlin property
in Example 5.10 of~\cite{HO13}.)
However, by Corollary 4.8(1) of~\cite{HrsPh1},
there is no action on ${\mathcal{Z}}$
which has
any higher dimensional Rokhlin property with commuting towers.

The Rokhlin property case of our Cuntz semigroup
is already known, even for nonunital C*-algebras
(Theorem~4.1 of~\cite{GdlStg}).
That paper
does not consider the radius of comparison,
does not consider $\W (A)$,
and gives no example like that in our Section~\ref{Sec_Ex},
in which the algebra is simple and has nonzero radius of comparison.
Moreover, as pointed out above,
the weak tracial Rokhlin property is much more common than
the Rokhlin property.

Something close to the case $\rc (A) = 0$
of our radius of comparison result is also already known.
Let $A$ be a simple separable nuclear unital C*-algebra.
If $\rc (A) = 0$,
and if one assumes that the set of extreme points
of $\T (A)$ is compact and finite-dimensional
then $A$ is ${\mathcal{Z}}$-stable
(Corollary~7.9 of~\cite{KbgRdm4};
Corollary~1.2 of~\cite{Sato3};
Corollary~4.7 of~\cite{TWW})
and, in particular, tracially ${\mathcal{Z}}$-absorbing
in the sense of Definition 2.1 of~\cite{HO13}.
If $G$ is finite and
$\af \colon G \to \Aut (A)$
has the generalized tracial Rokhlin property
(Definition 5.2 of~\cite{HO13}),
then $C^* (G, A, \af)$
is tracially ${\mathcal{Z}}$-absorbing
by Theorem 5.6 of~\cite{HO13}.
Therefore $A$ has strict comparison by Theorem 3.3 of~\cite{HO13}.
(The group need not be finite;
see Definition 6.1 of~\cite{HO13}
and Theorem 6.7 of~\cite{HO13}
for results for~$\Z$,
and~\cite{OvPhWn} for some results
for actions of countable amenable groups.)

Despite the relative
abundance of actions with the weak tracial Rokhlin property,
it is not obvious that there are actions
with this property
on \ssuc{s} with strictly positive radius of comparison.
Getting the Rokhlin property seems even harder.
For example,
according to Theorems 3.4 and~3.5 of~\cite{Iz2},
if in addition $A$ is nuclear,
satisfies the Universal Coefficient Theorem,
and has tracial rank zero
(or is a unital Kirchberg algebra
satisfying the Universal Coefficient Theorem),
and $\alpha \colon G \to \Aut (A)$
is an action of $G$ on $A$ which has the Rokhlin property
and is trivial on K-theory,
then $A$ is stable
under tensoring with the $\card (G)^{\infty}$~UHF algebra.
The same conclusion,
under somewhat different hypotheses,
is obtained in Theorem 5.10 of~\cite{GdlStg}.
One might naively expect something like this to be true more generally.
In fact, though, we exhibit an
action $\alpha \colon G \to \Aut (A)$
with the Rokhlin property
(not just the weak tracial Rokhlin property),
in which $A$ is a simple unital AH~algebra,
$G = \Z / 2 \Z$
(although a similar construction will work for any finite group),
and $A$, $A^{\af}$, and $\CGAa$
all have finite but nonzero radius of comparison.

When $G$ has order~$n$
and $\alpha \colon G \to \Aut (A)$
is an action of $G$ on $A$ which has the Rokhlin property,
the usual method of proving properties
of $\CGAa$
is local approximation by algebras of the form
$M_n (e A e)$ for suitable projections $e \in A$.
See Theorem 3.2 of~\cite{OskPhl3};
there are a number of applications of this method in that paper.
Weaker versions of this are true for
versions of the weak tracial Rokhlin property,
and are implicit in Sections 5 and~6 of~\cite{HO13}
and in~\cite{OvPhWn}.
This method does not seem to work
even for the Rokhlin property case of our radius of comparison result;
the best we could get this way
is $\rc \bigl( \CGAa \bigr) \leq \rc (A)$.
The difficulty is with $\rc (e A e)$.
Instead,
we first prove that $\rc (A^{\af}) \leq \rc (A)$.
Using the notation of Definition~\ref{D_9422_dtau_dfn}
and Definition~\ref{rc_dfn} below,
suppose we have $a, b \in (A^{\af})_{+}$
with $d_{\ta} (a) + \rc (A) < d_{\ta} (b)$
for every normalized quasitrace~$\ta$ on~$A^{\af}$.
This applies in particular for
normalized quasitraces $\ta$ on~$A$,
so $a \precsim_A b$.
Thus,
there is $v \in A$ such that
$\| v b v^* - a \|$ is small.
Now use the Rokhlin property to ``average'' $v$ over~$G$,
as described in Remark 10.3.9 and Exercise 10.3.10 of~\cite{GKPT18},
getting $w \in A^{\af}$ such that $\| w b w^* - a \|$ is small.
The generalization to the weak tracial Rokhlin property
uses the same idea,
but is considerably more technical,
and requires generalizations of some of
the Cuntz comparison results in~\cite{Ph14}.

The construction of an action with the Rokhlin property
is a modification of the idea used in~\cite{HP19}
to find an automorphism of the Elliott invariant
of a simple AH~algebra which does not lift to an automorphism
of the algebra.
The construction there ``merged'' two direct systems whose
direct limits had different radii of comparison but the same
Elliott invariant.
This ``merging'' was done by adding a very small number
of maps which go from one of the original systems to the other.
Here,
we ``merge'' two copies of the same direct system,
with the system having been chosen
so that its direct limit has large radius of comparison.
The action exchanges the two copies of this system.
The Rokhlin projections are, roughly speaking,
the identities of the algebras in the two original systems.

This paper is organized as follows.
Section~\ref{Sec_Cu} contains information
on the Cuntz semigroup,
Cuntz comparison,
quasitraces,
and the radius of comparison.
It also contains several approximation results
which are used repeatedly.
Some of this material is new or at least not in the literature,
and some definitions and results are stated here
for the convenience of the reader and for easy reference.
In Section~\ref{Sec_rcFix}
we prove injectivity on the purely positive part for the map
$\Cu (A^{\af}) \to \Cu (A)^{\af}$.
This is enough to prove the bound on the radius of comparison
of a crossed product
by an action with the weak tracial Rokhlin property,
and our surjectivity result does not seem to help
with the reverse inequality,
so we prove the bound in Section~\ref{Sec_rcCP}.
Section~\ref{Sec_Surj} contains
our surjectivity result on the purely positive part for the map
$\Cu (A^{\af}) \to \Cu (A)^{\af}$,
as well as results on $\W (A^{\af}) \to \W (A)^{\af}$
when $A$ has stable rank one.
Since $\W (A)$ is not considered in~\cite{GdlStg},
we also prove the corresponding result for Rokhlin actions
on unital but not necessarily simple \ca{s}.
In Section~\ref{Sec_Ex},
we construct the example referred to above.
In Section~\ref{Sec_Q},
we state a few open problems.


\section{Preliminaries}\label{Sec_Cu}

In this section,
we collect for easy reference
some information on the Cuntz semigroup, quasitraces,
and the radius of comparison.
A fair amount is already in the literature,
but there are several facts we did not find,
among them,
the estimate in Theorem~\ref{Ourcornertheorem}
for the radius of comparison of a corner.
Lemma~\ref{Lem.ANP.Dec.18}
is definitely new.

\subsection{Cuntz subequivalence}\label{Sec_CuSub}

\begin{ntn}\label{N_9408_StdNotation}
We use the following standard notation.
If $A$ is a \ca, or if $A = M_{\infty} (B)$
for a C*-algebra~$B$, we write $A_{+}$ for the
set of positive elements of $A$.
\end{ntn}

Parts (\ref{Cuntz_def_property_a}) and~(\ref{Cuntz_def_property_b})
of the following definition are originally from~\cite{Cun78}.
The usual notation for Cuntz subequivalence is $a \precsim b$.
We include $A$ in the notation
because we need to use Cuntz subequivalence
with respect to subalgebras.

\begin{dfn}\label{Cuntz_def_property}
Let $A$ be a \ca.
\begin{enumerate}
\item\label{Cuntz_def_property_a}
For $a, b \in M_{\infty} (A)_{+}$,
we say that $a$ is {\emph{Cuntz subequivalent to~$b$ in~$A$}},
written $a \precsim_{A} b$,
if there is a sequence $(v_n)_{n = 1}^{\infty}$ in $M_{\infty} (A)$
such that
\[
\limi{n} v_n b v_n^* = a.
\]
\item\label{Cuntz_def_property_b}
We say that $a$ and $b$ are {\emph{Cuntz equivalent in~$A$}},
written $a \sim_{A} b$,
if $a \precsim_{A} b$ and $b \precsim_{A} a$.
This relation is an equivalence relation,
and we write $\langle a \rangle_A$ for the equivalence class of~$a$.
We define $\W (A) = M_{\infty} (A)_{+} / \sim_A$,
together with the commutative semigroup operation
$\langle a \rangle_A + \langle b \rangle_A
 = \langle a \oplus b \rangle_A$
and the partial order
$\langle a \rangle_A \leq \langle b \rangle_A$
if $a \precsim_{A} b$.
We write $0$ for~$\langle 0 \rangle_A$.
\item\label{Cuntz_def_property_d}
We take $\Cu (A) = \W (K \otimes A)$.
We write the classes as $\langle a \rangle_A$
for $a \in (K \otimes A)_{+}$.
\item\label{Cuntz_def_property_e}
Let $A$ and $B$ be C*-algebras,
and let $\ph \colon A \to B$ be a \hm.
We use the same letter for the induced maps
$M_n (A) \to M_n (B)$
for $n \in \N$ and
$\Mi (A) \to \Mi (B)$.
We define
$\W (\ph) \colon \W (A) \to \W (B)$
and $\Cu (\ph) \colon \Cu (A) \to \Cu (B)$
by $\langle a \rangle_A \mapsto \langle \ph (a) \rangle_B$
for $a \in M_{\infty} (A)_{+}$
or $a \in (K \otimes A)_{+}$ as appropriate.
\end{enumerate}
\end{dfn}

\begin{dfn}
Let $A$ be a \ca,
let $a \in A_{+}$,
and let $\ep \geq  0$.
Let $f \colon [0, \infty) \to
[0, \infty)$  be the function
$f (t) = \max (0, \, t - \ep) = (t - \ep)_{+}$.
Then, by functional calculus, define $(a - \ep)_{+} = f (a)$.
\end{dfn}

Part (\ref{PhiB.Lem_18_4_11}) of the following
is taken from Proposition 2.4  of~\cite{Ror92}.
Parts (\ref{PhiB.Lem_18_4_8}) and~(\ref{PhiB.Lem_18_4_10.a})
are Lemma 2.5(i) and Lemma 2.5(ii) of ~\cite{KR00}.
Part (\ref{A.P.T}) is Lemma~2.2 of~\cite{KR02}.
Part (\ref{Item_9420_LgSb_1_6}) is Corollary 1.6 of~\cite{Ph14}.
Part (\ref{PhiB.Lem_18_4_4}) is taken from the discussion after
Definition~2.3 of~\cite{KR00} and Proposition~2.3(ii) of~\cite{ERS11}.

\begin{lem}\label{PhiB.Lem_18_4}
Let $A$ be a \ca.
\begin{enumerate}
\item\label{PhiB.Lem_18_4_11}
Let $a, b \in A_{+}$.
Then \tfae:
\begin{enumerate}
\item\label{PhiB.Lem_18_4_11.a}
$a \precsim_A b$.
\item\label{PhiB.Lem_18_4_11.b}
$(a - \ep)_{+} \precsim_A b$ for all $\ep > 0$.
\item\label{PhiB.Lem_18_4_11.c}
For every $\ep > 0$ there is $\dt > 0$ such that
$(a - \ep)_{+} \precsim_A (b - \dt)_{+}$.
\end{enumerate}
\item\label{PhiB.Lem_18_4_8}
Let $a \in A_{+}$ and let $\ep_1, \ep_2 > 0$.
Then
\[
\big( ( a - \ep_1)_{+} - \ep_2 \big)_{+}
 = \big( a - ( \ep_1 + \ep_2 ) \big)_{+}.
\]
\item\label{PhiB.Lem_18_4_10}
Let $a, b \in A_{+}$ and let $\ep > 0$.
If $\| a - b \| < \ep$, then:
\begin{enumerate}
\item\label{PhiB.Lem_18_4_10.a}
$(a - \ep)_{+} \precsim_A b$.
\item\label{A.P.T}
There is a contraction $d$ in $A$ such that
$d b d^*= ( a - \ep )_{+}$.
\item\label{Item_9420_LgSb_1_6}
For any $\ld > 0$,
we have $(a - \ld - \ep)_{+} \precsim_A (b - \ld)_{+}$.
\end{enumerate}
\item\label{PhiB.Lem_18_4_4}
Let $c \in A$ and let $\lambda \geq 0$.
Then $(c^* c - \lambda)_{+} \sim_A (c c^* - \lambda)_{+}$.
\end{enumerate}
\end{lem}

\begin{lem}\label{elem.lem1}
Suppose $t \in [0, 1]$ and $s \in [0, 1)$.
Then:
\begin{enumerate}
\item\label{12.31.18.a}
$2t - t^{2} - s > 0$ \ifo{}
$t - 1 + \sqrt{1 - s} > 0$.
\item\label{12.31.18.b}
$1 - \sqrt{1 - s} \geq \frac{s}{2}$.
\end{enumerate}
\end{lem}

\begin{proof}
These statements are easy to check.
\end{proof}

The following lemma is a generalization
of Lemma~1.8 of \cite{Ph14} or Lemma~12.1.5 of \cite{GKPT18}.

\begin{lem}\label{Lem.ANP.Dec.18}
Let $A$ be a unital \ca, let $a, g \in A$
satisfy $0 \leq a, g \leq 1$,
and let $\ep_1, \ep_2 \geq 0$.
Then
\[
\big( a - (\ep_1 + \ep_2) \big)_{+}
\precsim_{A} \big( (1 - g ) a (1 - g ) - \ep_1 \big)_{+}
         \oplus \Big( g - \frac{\ep_2 }{2} \Big)_{+}.
\]
\end{lem}

\begin{proof}
We may clearly assume $\ep_2 < 1$.
Set $h = 2 g - g^2$.
Functional calculus and Lemma~\ref{elem.lem1}(\ref{12.31.18.a})
imply that
\begin{equation}\label{Eq_9422_hep}
(h - \ep_2)_{+} \sim_{A} \big( g - [1 - (1 - \ep_2)^{1/2}] \big)_{+}.
\end{equation}

Since $\ep_2 \in [0, \, 1)$, it follows from
\Lem{elem.lem1}(\ref{12.31.18.b}) that
$1 - \sqrt{1 - \ep_2} \geq \frac{\ep_2}{2}$.
So
\begin{equation}\label{1.4.19.22}
\big( g - 1 + (1 - \ep_2)^{1/2} \big)_{+}
 \precsim_A \Big( g - \frac{\ep_2}{2} \Big)_{+}.
\end{equation}

Set $b = \big( (1 - g ) a (1 - g ) - \ep_1 \big)_{+}$.
Using Lemma~1.5 of~\cite{Ph14}
at the first step,
\Lem{PhiB.Lem_18_4}(\ref{PhiB.Lem_18_4_4})
at the second step,
$\| a \| \leq 1$ and Lemma~1.7 of~\cite{Ph14}
on the second summand at the third step,
(\ref{Eq_9422_hep}) at the fourth step,
and (\ref{1.4.19.22}) at the last step,
we get
\begin{align*}
\big( a - (\ep_1 + \ep_2) \big)_{+}
& \precsim_A \bigl( a^{1/2} (1 - h) a^{1/2} - \ep_1 \bigr)_{+}
         \oplus \bigl( a^{1/2} h a^{1/2} - \ep_{2} \bigr)_{+}
\\
& \sim_A \bigl( (1 - g) a (1 - g) - \ep_1 \bigr)_{+}
         \oplus \bigl( h^{1/2} a h^{1/2} - \ep_2 \bigr)_{+}
\\
& \precsim_A b \oplus ( h - \ep_2 )_{+}
\\
& \sim b \oplus \bigl( g - 1 + (1 - \ep_2)^{1/2} \bigr)_{+}
  \precsim_A b \oplus \Big( g - \frac{\ep_2}{2} \Big)_{+}.
\end{align*}
This completes the proof.
\end{proof}

Let $a, b \in A_{+}$.
If $a \precsim_A b$ then by definition
there is a sequence $(v_n)_{n = 1}^{\infty}$ in $A$
such that
$\limi{n} v_n b v_n^* = a$.
But there need not be a bounded sequence with this property.
As a substitute,
we have the following result,
originally from \cite{AGJP17}.
We give a proof for the sake of completeness.
(There is a similar result in Lemma~2.4(ii) of~\cite{KR02}, but there
is a gap in the proof.)

\begin{lem}\label{A.G.J.P}
Let $A$ be a \ca, let $a, b \in A_{+}$, and let $\dt> 0$.
If $a \precsim_{A} ( b - \dt )_{+}$,
then there exists a sequence $(w_n)_{n \in \N}$ in $A$ such that
$\| a - w_n b w_n^* \| \to 0$
and $\| w_n \| \leq \|a \|^{1/2} \dt^{- 1/2}$
for every $n \in \N$.
\end{lem}

\begin{proof}
Let $n \in \N$.
Since $a \precsim_A (b - \dt)_{+}$, there exists $v_n \in A$ such that
\[ \|a -v_n (b - \dt)_{+} v_n^* \|< \frac{1}{n}.
\]
Using \Lem{PhiB.Lem_18_4}(\ref{A.P.T}), we find
a contraction $d_n \in A$ such that
\[
\Big( a - \frac{1}{n} \Big)_{+} = d_n v_n (b - \dt)_{+} v_n^* d_n^*.
\]
Now, applying Lemma~2.4(i) of~\cite{KR02},
we get $w_n \in A$ such that
\[
\Big( a - \frac{1}{n} \Big)_{+} = w_n b w_n^*
\qquad
\mbox{and}
\qquad
\| w_n \|
 \leq \Big\| \Big( a - \frac{1}{n} \Big)_{+} \Big\|^{1/2} \dt^{-1/2}.
\]
Therefore $w_n b w_n^* \to a$ and
$\| w_n\| \leq \| a \|^{1/2} \dt^{-1/2}$.
\end{proof}


\subsection{Quasitraces on \ca{s}}
The following definition is from \cite{Hag14}.
Parts (\ref{quasitrace.a}), (\ref{quasitrace.b}), and
(\ref{quasitrace.c})
correspond to the definition of a quasitrace in \cite{BH82}.
What we and \cite{Hag14} call a quasitrace
is called a ``2-quasitrace'' in \cite{BH82}.

\begin{dfn}\label{quasitrace}
Let $A$ be a \ca.
A function $\tau \colon A \to \mathbb{C}$
is a {\emph{quasitrace}} if the following hold:
\begin{enumerate}
\item\label{quasitrace.a}
$\tau (x^* x) = \tau (x x^*) \geq 0$ for all $x \in A$.
\item\label{quasitrace.b}
$\tau (a + i b ) = \tau (a) + i \tau ( b ) $
for $a, b \in A_{\mathrm{sa}}$.
\item\label{quasitrace.c}
$\tau |_B$ is linear
for every commutative C*-subalgebra $B \subseteq A$.
\item\label{quasitrace.d}
There is a function $\tau_2 \colon M_2 (A) \to \mathbb{C}$
satisfying
(\ref{quasitrace.a}), (\ref{quasitrace.b}), and (\ref{quasitrace.c})
with $M_2 (A)$ in place of $A$,
and such that,
with $(e_{j, k})_{j, k = 1}^{2}$
denoting the standard system of matrix units in $M_2 (\mathbb{C})$,
for all $x \in A$ we have
\[
\tau (x) = \tau_2 (x \otimes e_{1, 1}).
\]
\end{enumerate}
A quasitrace $\tau$ on a unital \ca{}
is {\emph{normalized}} if $\tau (1) = 1$.
The set of normalized quasitraces on $A$ is denoted by $\QT (A)$.
\end{dfn}

All quasitraces on a unital exact \ca{} are
traces, by Theorem~5.11 of~\cite{Hag14}.

\begin{prp}[\cite{BH82}]\label{P_9312_ExistQT}
Let $A$ be a stably finite unital \ca.
Then $\QT (A) \neq \varnothing$.
\end{prp}

\begin{proof}
This is in the discussion after Proposition II.4.6 of~\cite{BH82}.
\end{proof}

Part~(\ref{BH82_Cor_2_2_3}) of the following proposition is
Corollary II.2.3 of~\cite{BH82} and Parts
(\ref{BH82_Cor_2_2_5_a}) through~(\ref{BH82_Cor_2_2_5_f})
are taken from Corollary II.2.5 of~\cite{BH82}.
That paper uses $\| \tau \|$ instead of $N (\tau)$.
We want to avoid conflict
with the definition of the norm of a linear functional.

\begin{prp}[\cite{BH82}]\label{BH82_Cor_2_2_5}
Let $\tau \colon A \to \mathbb{C}$ be a quasitrace on a \ca~$A$,
and define
\[
N (\tau)
 = \sup \bigl( \bigl\{ \tau (a) \colon
   {\mbox{$a \in A_{+}$ and $\|a \| \leq 1$}} \bigr\} \bigr).
\]
Then:
\begin{enumerate}
\item\label{BH82_Cor_2_2_3}
$N (\tau) < \infty$.
\item\label{BH82_Cor_2_2_5_a}
If $A$ is unital and $\tau \in \QT (A)$, then $N (\tau) = 1$.
\item\label{BH82_Cor_2_2_5_b}
$\tau$ is order-preserving.
\item\label{BH82_Cor_2_2_5_c}
If $a, b \in A_{\mathrm{sa}}$, then
$|\tau (a) - \tau (b)| \leq N (\tau) \| a - b \|$.
\item\label{BH82_Cor_2_2_5_e}
$\tau$ is norm-continuous.
\item\label{BH82_Cor_2_2_5_f}
If $a, b \in A_{+}$, then $\tau (a + b)
\leq 2 \big( \tau (a) + \tau (b) \big)$.
\end{enumerate}
\end{prp}

\begin{prp}[\cite{BH82}]\label{P_9305_BH82_Cor2225}
Let $A$ be a \ca{}
and let $\tau$ be a quasitrace on~$A$.
Then $\tau$ extends uniquely to a quasitrace~$\ta_{\I}$
on $M_{\infty} (A)$
such that,
with $(e_{j, k})_{j, k = 1}^{\I}$
denoting the standard system of matrix units
in $M_{\infty} (\mathbb{C})$,
we have $\tau_{\I} (a \otimes e_{j, j}) = \tau (a)$
for all $a \in A$ and $j \in \N$.
\end{prp}

We denote the restriction of $\ta_{\I}$ to $M_n (A)$ by $\ta_n$.
When no confusion is likely,
we abbreviate $\ta_{\I}$ and $\ta_n$ to~$\ta$.

\begin{proof}[Proof of Proposition~\ref{P_9305_BH82_Cor2225}]
For $M_n (A)$ in place of $M_{\infty} (A)$,
this is Proposition~II.4.1 of~\cite{BH82}.
By uniqueness there, for all $n \in \N$,
the restriction to $M_n (A)$ of the extension to $M_{n + 1} (A)$
is the extension to $M_n (A)$.
This implies existence of the extension to $M_{\infty} (A)$,
and uniqueness is now immediate.
\end{proof}

The following lemma is part of Proposition~3.2 of~\cite{Hag14}.
(There is a misprint there:
it cites Theorem I.1.1 of~\cite{BH82},
but apparently Theorem I.1.17 is intended.)
Given Proposition \ref{BH82_Cor_2_2_5},
we can give a simple direct proof,
which is the same as for traces except for an extra factor
of~$2$ in the proof of closure under addition.

\begin{lem}\label{Ntau.ideal}
Let $\tau$  be a quasitrace on a \ca~$A$.
Then the set
\[
J_{\tau} = \{ x \in A \colon \tau (x^*x) = 0  \}
\]
is a closed two-sided ideal in~$A$.
\end{lem}

\begin{proof}
It is obvious that $J_{\ta}$ is closed under scalar
multiplication and $x \mapsto x^*$.

Let $x, y \in J_{\ta}$.
Then
\[
(x + y)^* (x + y)
 \leq (x + y)^* (x + y) + (x - y)^* (x - y)
 = 2 x^* x + 2 y^* y,
\]
so,
by Proposition \ref{BH82_Cor_2_2_5}(\ref{BH82_Cor_2_2_5_b})
and Proposition \ref{BH82_Cor_2_2_5}(\ref{BH82_Cor_2_2_5_f}),
\[
0 \leq \ta \bigl( (x + y)^* (x + y) \bigr)
  \leq \ta \bigl( 2 x^* x + 2 y^* y \bigr)
  \leq 4 \bigl( \ta ( x^* x ) + \ta ( y^* y ) \bigr)
  = 0.
\]
Hence $x + y \in J_{\ta}$.

Let $x \in J_{\ta}$ and let $a \in A$.
Then,
using Proposition \ref{BH82_Cor_2_2_5}(\ref{BH82_Cor_2_2_5_b}),
\[
0 \leq \ta ( (a x)^* (a x) )
  \leq \| a^* a \| \ta (x^* x)
  = 0,
\]
so $a x \in J_{\ta}$.
Now
$x a = (a^* x^*)^* \in J_{\ta}$.
\end{proof}

We will need Murray-von Neumann equivalence.
We use notation which distinguishes it from Cuntz equivalence.

\begin{dfn}\label{N_9422_MvN}
Let $A$ be a \ca, and let $p, q \in K \otimes A$ be projections.
We say {\emph{$p$ is Murray-von Neumann subequivalent to $q$}},
denoted $p \lessapprox q$, if there exists $v \in K \otimes A$ such that
$p = v v^*$ and $v^* v \leq q$.
We say that $p$ and $q$ are
{\emph{Murray-von Neumann equivalent}}, denoted $p \approx q$,
if there exists $v \in K \otimes A$
such that $p = v v^*$ and $v^* v = q$.
\end{dfn}

It is well known
that $p \lessapprox q$
if and only if $p \precsim_{A} q$.
However,
it is in general not true that
$p \sim_A q$ implies $p \approx q$.
For example, this fails in a purely infinite simple \ca{}
with nonzero $K_0$-group.
However, if $A$ is stably finite then
$p \sim_A q$ and $p \approx q$ are equivalent.

\subsection{Radius of comparison}
The following definition is Definition~12.1.7 of~\cite{GKPT18}.

\begin{dfn}\label{D_9422_dtau_dfn}
Let $A$ be a unital \ca,
and let $\tau \in \QT (A)$.
Recalling the notation of and
after Proposition~\ref{P_9305_BH82_Cor2225},
define $d_{\tau} \colon \Mi (A)_{+} \to [0, \infty)$
by
\[
d_{\tau} (a) = \lim_{n \to \infty} \tau (a^{1/n})
\]
for $a \in \Mi (A)_{+}$.
We also use the same notation for the corresponding functions
on $\Cu (A)$ and $\W (A)$.
\end{dfn}

The following is Definition~6.1 of~\cite{Tom06},
except that we allow $r = 0$ in~(\ref{rc_dfn.a}).
This change makes no difference.

\begin{dfn}\label{rc_dfn}
Let $A$ be a stably finite unital C*-algebra.

\begin{enumerate}
\item\label{rc_dfn.a}
Let $r \in [0, \I)$.
We say that $A$ has {\emph{$r$-comparison}} if whenever
$a, b \in M_{\infty} (A)_{+}$ satisfy
$d_{\ta} (a) + r < d_{\ta} (b)$
for all $\ta \in \QT (A)$,
then $a \precsim_A b$.
\item\label{rc_dfn.b}
The {\emph{radius of comparison}} of~$A$,
denoted ${\operatorname{rc}} (A)$, is
\[
\rc (A)
 = \inf \big( \big\{ r \in [0, \I) \colon
    {\mbox{$A$ has $r$-comparison}} \big\} \big)
\]
if it exists, and $\infty$ otherwise.
\end{enumerate}
\end{dfn}

If $A$ is simple,
then the infimum in \Def{rc_dfn}(\ref{rc_dfn.b})
is attained,
that is, $A$ has ${\operatorname{rc}} (A)$-comparison;
see Proposition~6.3 of~\cite{Tom06}.
For exact \ca{s}, one only needs to consider extreme tracial states;
see Lemma~2.3 of \cite{EN13}.

By Proposition 6.12 of~\cite{Ph14}, the radius of comparison
of a simple \uca{}
is the same
whether computed using $\W (A)$ or ${\operatorname{Cu}} (A)$.


\subsection{The radius of comparison of a corner}\label{Sec_Corner}
We give bounds on the radius of comparison of a full corner
in a matrix algebra over a \ca.
The result is surely known,
but we have not seen a proof in the literature.
It will be needed in Section~\ref{Sec_rcCP} below,
to relate $\rc (A^{\af})$ to $\rc (C^* (G, A, \af) )$.

\begin{lem}\label{Ourlemma.1112}
Let $A$ be a stably finite \uca,
let $n \in \N$, and
let $p$ be a full projection in $M_n (A)$.
Then,
recalling the notation in and
after Proposition~\ref{P_9305_BH82_Cor2225}:
\begin{enumerate}
\item\label{Ourlemma.1112_a}
$\inf_{\tau \in \QT (A)} \tau_n (p) > 0$.
\item\label{Ourlemma.1112_b}
The map
$\theta \colon \QT (A) \to \QT (p M_n (A) p)$,
given by
$\tau \mapsto \frac{1}{\tau_n (p)} \tau_n |_{p M_n (A) p}$,
is bijective.
\end{enumerate}
\end{lem}

\begin{proof}
Since $\ta \mapsto \frac{1}{n} \ta_n$
is a bijection from $\QT (A)$ to $\QT \bigl( M_n (A) \bigr)$,
it is easily seen that it suffices to prove the result
when $n = 1$.

Since $A$ is unital and $p$ is full in $A$,
it follows that $A p A = A$.
Therefore, using \Lem{Ntau.ideal},
$\tau (p) > 0$ for all $\tau \in \QT (A)$.
Since $\QT (A)$ is nonempty (by Proposition~\ref{P_9312_ExistQT})
and compact,
and since $\tau \mapsto \tau (p)$ is \ct,
(\ref{Ourlemma.1112_a}) follows.

To prove (\ref{Ourlemma.1112_b}), clearly
$\frac{1}{\tau (p)} \tau |_{p A p} \in \QT (p A p)$.
Bijectivity of $\theta$ now follows
from Proposition II.4.2 of~\cite{BH82}
and Proposition~\ref{P_9305_BH82_Cor2225}.
\end{proof}

\begin{lem} \label{Ourlemma2221}
Let $A$ be a stably finite unital \ca,
let $n \in \N$, and let $p$ be
a full projection
in $M_n (A)$.
Recalling the notation in and
after Proposition~\ref{P_9305_BH82_Cor2225},
if
$\lambda
 = \inf \bigl( \bigl\{ \tau_n (p) \colon
     \tau \in \QT (A) \bigr\} \bigr)$,
then $0 < \ld \leq n$ and
\[
\rc \bigl(p M_{n} (A) p \bigr)
 \leq \frac{1}{\lambda} \cdot \rc (A).
\]
\end{lem}

\begin{proof}
By \Lem{Ourlemma.1112}(\ref{Ourlemma.1112_a})
we have
$\lambda > 0$.
Since $p \leq 1_{M_n (A)}$ and $\tau_n (1_{M_n (A)}) = n$
for all $\ta \in \QT (A)$,
and since $\QT (A) \neq \E$ by Proposition~\ref{P_9312_ExistQT},
it follows from
Proposition \ref{BH82_Cor_2_2_5}{(\ref{BH82_Cor_2_2_5_b})} that
$\lambda \leq n$.

Now let $m \in \N$,
let $a, b \in M_{m} (p M_n (A) p)_{+} \subseteq M_{m n} (A)_{+}$,
and suppose that
$d_{\rho} (a) + \frac{1}{\lambda} \cdot \rc (A) < d_{\rho} (b)$
for all $\rho \in \QT (p M_n (A) p)$.
By \Lem{Ourlemma.1112}{(\ref{Ourlemma.1112_b})},
this is the same as
\[
d_{\ta} (a) + \frac{\ta_n (p)}{\lambda} \cdot \rc (A) < d_{\ta} (b)
\]
for all $\ta \in \QT (A)$.
Since $\ld \leq \ta_n (p)$,
it follows that $a \precsim_{A} b$,
so $a \precsim_{p M_n (A) p} b$.
\end{proof}

\begin{thm}\label{Ourcornertheorem}
Let $A$ be a stably finite unital \ca, let $n \in \N$,
and let $p$ be a full
projection in $M_{n} (A)$.
Recalling the notation in and
after Proposition~\ref{P_9305_BH82_Cor2225},
define
\[
\lambda
 = \inf \bigl( \bigl\{ \tau_{n} (p) \colon
     \tau \in \QT (A) \bigr\} \bigr)
\andeqn
\eta
 = \sup \bigl( \bigl\{ \tau_{n} (p) \colon
     \tau \in \QT (A) \bigr\} \bigr).
\]
Then $0 < \ld \leq \et \leq n$ and
\[
\frac{1}{\eta} \cdot \rc (A)
 \leq \rc \big( p M_{n} (A) p \big)
 \leq \frac{1}{\lambda} \cdot \rc (A).
\]
\end{thm}

\begin{proof}
The parts involving $\ld$ are \Lem{Ourlemma2221}.
Since $\QT (A) \neq \E$
(by Proposition~\ref{P_9312_ExistQT}),
the relations $\ld \leq \et \leq n$
are clear.

Since $p$ is full,
there are $m \in \N$ and a projection
$q \in M_m (p M_n (A) p)$ such that $1_A \approx q$.
Then
$A \cong q M_m (p M_n (A) p) q$.
Apply \Lem{Ourlemma2221}
with $p M_n (A) p$ in place of~$A$,
with $m$ in place of~$n$,
and with $q$ in place of~$p$.
We get
\begin{equation}\label{Eq_9306_rcA}
\rc (A)
 \leq \left( \frac{1}{\inf \bigl( \bigl\{ \sm_m (q) \colon
           \sm \in \QT (p M_n (A) p) \bigr\} \bigr)} \right)
      \rc (p M_n (A) p).
\end{equation}
By \Lem{Ourlemma.1112}{(\ref{Ourlemma.1112_b})},
\begin{equation}\label{Eq_9306_CorrQT}
\QT (p M_n (A) p)
 = \left\{ \tau_n (p)^{-1} \tau_n |_{p M_n (A) p} \colon
         \ta \in \QT (A) \right\}.
\end{equation}
If $\sm \in \QT (p M_n (A) p)$ and $\ta \in \QT (A)$
is the corresponding quasitrace from~(\ref{Eq_9306_CorrQT}),
then, using $q \approx 1_A$,
we get
$\sm_m (q) = \frac{\ta_{m n} (q)}{\ta_{n} (p)} = \frac{1}{\ta_n (p)}$.
Therefore
\[
\inf \bigl( \bigl\{ \sm_m (q) \colon
     \sm \in \QT (p M_n (A) p) \bigr\} \bigr)
  = \frac{1}{\sup \bigl( \bigl\{ \ta_n (p) \colon
     \ta \in \QT (A) \bigr\} \bigr)}
  = \frac{1}{\et}.
\]
So~(\ref{Eq_9306_rcA})
implies that $\frac{1}{\eta} \cdot \rc (A) \leq \rc (p M_n (A) p)$.
\end{proof}


\subsection{Approximation lemmas}\label{Sec_App}
This subsection contains several approximation lemmas
which will be needed frequently.

\begin{lem}\label{9625_FuncCalcCom}
Let $M \in (0, \I)$,
let $f \colon [0, M] \to \C$ be continuous,
and let $\ep > 0$.
Then there is $\delta > 0$
such that whenever $A$ is a \ca{}
and $a, x \in A$
satisfy
\begin{equation*}
a \in A_{+},
\qquad
\| a \| \leq M,
\qquad
\| x \| \leq M,
\andeqn
\| a x - x a \| < \delta,
\end{equation*}
then $\| f (a) x - x f (a) \| < \ep$.
\end{lem}

\begin{proof}
The case $M = 1$ is Lemma~2.5 of~\cite{ArPh}.
The proof of this version is the same.
\end{proof}

The statement can also be gotten from Lemma 2.5 of~\cite{ArPh}
by scaling.

\begin{lem}\label{L_9420_NearZero}
Let $f, g \colon [0, \I) \to [0, \I)$
be \cfn{s} such that $f (0) = g (0) = 0$,
let $\ep > 0$,
and let $M \in (0, \I)$.
Then there is $\dt > 0$ such that whenever $A$ is a \ca,
and $a, b \in A_{+}$
satisfy $\| a b \| < \dt$ and $\| a \|, \, \| b \| \leq M$,
then $\| f (a) g (b) \| < \ep$.
\end{lem}

This lemma can be proved by approximating $a$ and $b$ by
positive elements whose product is zero,
but a direct proof seems just as easy.

\begin{proof}[Proof of Lemma~\ref{L_9420_NearZero}]
\Wolog{} $M \geq 1$ and $\ep < 1$.
Set
$C = \max \big( \| f |_{[0, M]} \|_{\I},
                 \, \| g |_{[0, M]} \|_{\I} \big)$.
Choose $m, n \in \Nz$ and
\[
\af_1, \af_2, \ldots, \af_m, \bt_1, \bt_2, \ldots, \bt_n \in \R
\]
such that
the polynomial functions
with no constant term,
given by
$f_0 (\ld) = \sum_{k = 1}^m \af_k \ld^k$
and
$g_0 (\ld) = \sum_{l = 1}^n \bt_l \ld^l$
for $\ld \in [0, M]$,
satisfy
\[
| f_0 (\ld) - f (\ld) | < \frac{\ep}{3 (C + 1)}
\andeqn
| g_0 (\ld) - g (\ld) | < \frac{\ep}{3 (C + 1)}
\]
for all $\ld \in [0, M]$.
\Wolog{} $m, n \geq 1$.
Define
\[
R = \max \big( | \af_1 |, | \af_2 |, \ldots, | \af_m |,
   | \bt_1 |, | \bt_2 |, \ldots, | \bt_n | \big)
  + 1
\quad {\mbox{and}} \quad
\dt = \frac{\ep}{3 m n R^2 M^{m + n}}.
\]

Now let $A$, $a$, and~$b$ be as in the hypotheses.
Using $M \geq 1$ at the second step, we have
\begin{align*}
\| f_0 (a) g_0 (b) \|
& \leq \Bigg( \sum_{k = 1}^m | \af_k | \cdot \| a \|^{k - 1} \Bigg)
          \| a b \|
          \Bigg( \sum_{l = 1}^n | \bt_l | \cdot \| b \|^{l - 1} \Bigg)
\\
& < (m R M^{m - 1}) \dt (n R M^{n - 1})
  \leq \frac{\ep}{3}.
\end{align*}
Therefore,
since $\ep < 1$ implies $\| f_0 (a) \| \leq \| f (a) \| + 1$,
\begin{align*}
\| f (a) g (b) \|
& \leq \| f (a) - f_0 (a) \| \cdot \| g (b) \|
         + \| f_0 (a) \| \cdot \| g (b) - g_0 (b) \|
         + \| f_0 (a) g_0 (b) \|
\\
& \leq \left( \frac{\ep}{3 (C + 1)} \right) C
       + (C + 1) \left( \frac{\ep}{3 (C + 1)} \right)
       + \frac{\ep}{3}
  < \ep.
\end{align*}
This completes the proof.
\end{proof}


\section{Injectivity of
  $\Cu_{+} (A^{\af}) \to \Cu_{+} (A)^{\af}$}\label{Sec_rcFix}

In this section,
we prove that if $G$ is finite,
$A$ is unital, stably finite, and simple,
and $\alpha \colon G \to \Aut (A)$
has the weak tracial Rokhlin property,
then the inclusion $A^{\af} \to A$
induces an isomorphism from the ordered semigroup
of purely positive elements
$\Cu_{+} (A^{\af}) \cup \{ 0 \}$
(see Definition~\ref{D_9421_Pure} below)
to a subsemigroup of $\Cu (A)$.
Example~\ref{E_9421_NotOnK0}
shows that this result fails if we do not
discard the classes of projections.

\begin{ntn}\label{N_9408_StdNotation_CP}
Let $\af \colon G \to \Aut (A)$
be an action of a finite group $G$ on a \uca~$A$.
For $g \in G$,
we let $u_g$ be the element of $C_{\mathrm{c}} (G, A, \af)$
which takes the value $1$ at $g$
and $0$ at the other elements of~$G$.
We use the same notation for its image in $C^* (G, A, \af)$.
We denote by $A^{\alpha}$ the fixed
point algebra, given by
\[
A^{\alpha} = \big\{ a \in A \colon
\alpha_g (a) = a \mbox{ for all } g \in G \big\}.
\]
We extend this notation to the elements of various objects associated
with~$A$
under the actions induced by~$\af$,
getting,
for example, $(K \otimes A)^{\alpha}$,
$K_0 (A)^{\alpha}$,
$\W (A)^{\alpha}$,
$\Cu (A)^{\alpha}$,
etc.
\end{ntn}

The following definition,
without Condition (\ref{Def.w.t.r.p.d})
but also requiring $\| f_g \| = 1$,
appears in Definition 5.2 of~\cite{HO13}
under the name \emph{generalized tracial Rokhlin property}.
Definition 2.2 of \cite{GHS17}
includes Condition~(\ref{Def.w.t.r.p.d})
but only has approximate orthogonality
of the contractions.
By Proposition~3.10 of~\cite{FG17},
Definition 2.2 of \cite{GHS17} is equivalent to our definition.
Condition~(\ref{Def.w.t.r.p.d}) is needed to ensure that
the trivial action on $\mathbb{C}$ or a purely infinite simple
unital \ca{} does not have the weak tracial Rokhlin property.

\begin{dfn}\label{W_T_R_P_def}
Let $G$ be a finite group,
let $A$ be a simple unital \ca,
and let $\alpha \colon G \to \Aut (A)$  be an action of
$G$ on $A$.
We say that $\alpha$ has the
\emph{weak tracial Rokhlin property} if for every $\ep > 0$,
every finite set $F \subseteq A$, and every positive
element $x \in A$ with $\| x \| = 1$,
there exist orthogonal positive contractions
$f_g \in A$ for $g \in G$ such that,
with $f = \sum_{g \in G} f_g$, the following hold:
\begin{enumerate}
\item\label{Def.w.t.r.p.a}
$\| a f_g - f_g a \| < \ep$ for all $g \in G$ and all $a \in F$.
\item\label{Def.w.t.r.p.b}
$\| \alpha_{g} ( f_h ) - f_{g h} \| < \ep$ for all $g, h \in G$.
\item\label{Def.w.t.r.p.c}
$1 - f \precsim_A x$.
\item\label{Def.w.t.r.p.d}
$\|  f x f \| > 1 - \ep$.
\end{enumerate}
\end{dfn}

In \Def{W_T_R_P_def},
if $G \neq \{ 1 \}$ the algebra $A$ can't be type~I,
since $\alpha$ must be pointwise outer.
(See Proposition 3.2 of \cite{FG17}.)
Therefore $A$ is infinite-dimensional.
(For clarity, we often explicitly include
infinite-dimensionality in hypotheses anyway.)

\begin{lem}\label{invariant_contractions}
Let $G$ be a finite group,
let $A$ be an infinite-dimensional simple unital \ca,
and let $\alpha \colon G \to \Aut (A)$  be an action of
$G$ on $A$ which has the
weak tracial Rokhlin property.
Then for every $\ep > 0$,
every finite set $F \subseteq A$, and every positive
element $x \in A$ with $\| x \| = 1$,
there exist positive contractions
$e_g, f_g \in A$ for $g \in G$ such that,
with $e = \sum_{g \in G} e_g$ and $f = \sum_{g \in G} f_g$,
the following hold:
\begin{enumerate}
\item\label{W.T.R.P.66}
$\| e_g e_h \| < \ep$ and $\| f_g f_h \| < \ep$
for all $g, h \in G$.
\item\label{W.T.R.P_11}
$\| a e_g - e_g a \| < \ep$ and $\| a f_g - f_g a \| < \ep$
for all $g \in G$ and all $a \in F$.
\item\label{W.T.R.P_22}
$\alpha_{g} ( e_h ) = e_{g h}$
and $\alpha_{g} ( f_h ) = f_{g h}$ for all $g, h \in G$.
\item\label{W.T.R.P.33}
$(1 - f - \ep)_{+} \precsim_A x$.
\item\label{W.T.R.P.44}
$\| f x f \| > 1 - \ep$.
\item\label{W.T.R.P.55}
$e \in A^{\alpha}$, $f \in A^{\alpha}$, and $\| f \| = 1$.
\item\label{Item_9525_New_ef}
$e_g f_g = f_g$ for all $g \in G$.
\setcounter{TmpEnumi}{\value{enumi}}
\end{enumerate}
\end{lem}

For most applications,
we do not need the elements~$e_g$.
Also,
presumably one can arrange to have
$\| e \| \leq 1$,
but we don't need this.

\begin{proof}[Proof of Lemma~\ref{invariant_contractions}]
Set $n = \card (G)$.

We may assume $\ep < \frac{1}{2}$.
Let $F \subseteq A$ be a finite set, and
let $x \in A_{+}$ satisfy $\| x \| = 1$.
Define
\begin{equation}\label{Eq_9526_DefRh}
M = \max \Bigl( 1, \, \max_{a \in F} \| a \| \Bigr)
\andeqn
\rh = \frac{\ep}{4 ( 1 + 3 n ) n + 1}.
\end{equation}
Define \cfn{s} $s, t \colon [0, 1] \to [0, 1]$
by
\[
s (\ld) = \left\{ \begin{array}{ll}
     \rh^{-1} \ld   & \hspace{1em}  0 \leq \ld \leq \rh  \\
     1         & \hspace{1em}  \rh \leq \ld \leq 1
                     \rule{0em}{2.5ex}
    \end{array} \right.
\andeqn
t (\ld) = \left\{ \begin{array}{ll}
     0         & \hspace{1em}  0 \leq \ld \leq \rh  \\
     \frac{\ld - \rh}{1 - \rh}   & \hspace{1em}  \rh \leq \ld \leq 1.
                     \rule{0em}{2.5ex}  \\
    \end{array} \right.
\]
Thus,
if $c \in A_{+}$ satisfies $\| c \| \leq 1$,
then
\begin{equation}\label{Eq_9528_StSt}
s (c) t (c) = t (c)
\andeqn
\| t (c) - c \| \leq \rh.
\end{equation}

Use Lemma~\ref{L_9420_NearZero}
to choose $\ep_0 > 0$ such that whenever
$B$ is a \ca{}
and $a, b \in B_{+}$ satisfy $\| a \| \leq 1$,
$\| b \| \leq 1$,
and $\| a b \| < \ep_0$,
then
\[
\| s (a) s (b) \| < \frac{\ep}{4}
\andeqn
\| t (a) t (b) \| < \frac{\ep}{4}.
\]
Use Lemma~\ref{9625_FuncCalcCom}
to choose $\ep_1 > 0$ such that whenever
$B$ is a \ca{}
and $a \in B_{+}$ and $z \in B$ satisfy $\| a \| \leq M$,
$\| z \| \leq M$,
and $\| a z - z a \| < \ep_1$,
then
\[
\| s (a) z - z s (a) \| < \frac{\ep}{2}
\andeqn
\| t (a) z - z t (a) \| < \frac{\ep}{2}.
\]
Define
\begin{equation}\label{Eq_9526_DefEp}
\ep' = \min \left( \rh, \, \frac{\ep_0}{2},
            \, \frac{\ep_1}{2 M + 1}
 \right).
\end{equation}
Applying \Def{W_T_R_P_def} with with $F$ and $x$ as given
and with $\ep'$ in place of $\ep$,
we get orthogonal positive contractions
$d_g \in A$ for $g \in G$ such that,
with $d = \sum_{g \in G} d_g$, the following hold:
\begin{enumerate}
\setcounter{enumi}{\value{TmpEnumi}}
\item\label{W.T.R.P_111}
$\| a d_g - d_g a \| < \ep'$ for all $g \in G$ and all $a \in F$.
\item\label{W.T.R.P_222}
$\| \alpha_{g} ( d_h ) - d_{g h} \|  < \ep'$ for all $g, h \in G$.
\item\label{W.T.R.P.333}
$1 - d \precsim_A x$.
\item\label{W.T.R.P.444}
$\| d x d  \| > 1 - \ep'$.
\setcounter{TmpEnumi}{\value{enumi}}
\end{enumerate}

Define
$\nu = \big\| \sum_{g \in G} \alpha_{g} (t (d_1)) \big\|$.
We claim that
\begin{equation}\label{W.T.R.P_111_prime}
| \nu - 1 | < 3 n \rh
\andeqn
\nu > \frac{1}{2}.
\end{equation}
To prove the claim, first
use $\| x \| = 1$ and $\| d \| \leq 1$ to get
\begin{equation}\label{Eq_9528_Star}
1 - \rh \leq 1 - \ep' < \| d x d \| \leq \| d \|^2 \leq \| d \|.
\end{equation}
Second, using the second part of~(\ref{Eq_9528_StSt})
at the second step,
\begin{align*}
\Bigg\| \sum_{g \in G} \alpha_{g} (t (d_1)) - d \Bigg\|
& \leq \sum_{g \in G} \| \alpha_{g} (t (d_1) - d_1) \|
       + \sum_{g \in G} \| \alpha_{g} (d_1) - d_g \|
\\
& < n \rh + n \ep'
  \leq 2 n \rh.
\end{align*}
This relation implies that
$\bigl| \nu - \| d \| \bigr| < 2 n \rh$,
so
\begin{equation}\label{1.3.19.11}
\nu < \| d \| + 2 n \rh
   \leq 1 + 2 n \rh
   \leq 1 + 3 n \rh,
\end{equation}
and, using~(\ref{Eq_9528_Star}) and $\ep < \frac{1}{2}$,
\begin{equation}\label{1.3.19.22}
\nu > \| d \| - 2 n \rh
  > 1 - \rh - 2 n \rh
  \geq 1 - 3 n \rh
  > 1 - \ep
  > \frac{1}{2}.
\end{equation}
Using (\ref{1.3.19.11}) and  (\ref{1.3.19.22}), we get
$| \nu - 1 | < 3 n \rh$.
The claim is proved.

Define
$e_g = \alpha_{g} (s (d_1) )$
and $f_{g} = \frac{1}{\nu} \alpha_{g} (t (d_1) )$
for $g \in G$,
and define
$e = \sum_{g \in G} e_g$
and $f = \sum_{g \in G} f_g$.
Clearly $\| e_g \| \leq 1$ and $\| f_g \| \leq 1$ for all $g \in G$,
and $\| f \| = 1$.
Part~(\ref{W.T.R.P_22}) of the conclusion is immediate.
Clearly $e, f \in A^{\alpha}$,
so we have~(\ref{W.T.R.P.55}),
and (\ref{Item_9525_New_ef})
follows from~(\ref{Eq_9528_StSt}).

Now we claim that:
\begin{enumerate}
\setcounter{enumi}{\value{TmpEnumi}}
\item\label{W.T.R.P_222_prime}
$\| f_g - d_g \|
 < 2 ( 1 + 3 n ) \rh$ for all $g \in G$.
\end{enumerate}
To prove the claim,
use the second part of~(\ref{Eq_9528_StSt})
and (\ref{W.T.R.P_222}) at the second step and
(\ref{W.T.R.P_111_prime}) at the third step to get
\begin{align*}
\| f_g - d_g \|
& \leq \Big\| \frac{1}{\nu} \alpha_{g} (t (d_1))
              - \alpha_{g} (t (d_1)) \Big\|
        + \| \alpha_{g} (t (d_1) - d_1) \|
                 + \| \alpha_{g} (d_1) - d_g \|
\\
& \leq \frac{1}{\nu} | \nu - 1 | \cdot \| t (d_1) \| + \rh + \ep'
  < 2 ( 1 + 3 n ) \rh,
\end{align*}
as desired.

We prove Part~(\ref{W.T.R.P.66}) of the conclusion.
For $g \neq h$,
using $d_g d_h = 0$ at the first step,
(\ref{W.T.R.P_222}) at the second step,
and (\ref{Eq_9526_DefEp}) at the third step,
we have
\[
\| \af_g (d_1) \af_h (d_1) \|
 \leq \| \af_g (d_1) \| \cdot \| \af_h (d_1) - d_h \|
       + \| \af_g (d_1) - d_g \| \cdot \| d_h \|
 < 2 \ep'
 \leq \ep_0.
\]
The choice of $\ep_0$
and the relations $e_g = s ( \af_g (d_1) )$
and $f_g = \nu^{-1} t ( \af_g (d_1) )$
now imply
\[
\| e_g e_h \| < \frac{\ep}{4} \leq \ep
\andeqn
\| f_g f_h \|
 < \nu^{-2} \left( \frac{\ep}{4} \right)
 < 4 \left( \frac{\ep}{4} \right)
 = \ep,
\]
which is~(\ref{W.T.R.P.66}).

We prove~(\ref{W.T.R.P_11}).
So let $g \in G$ and let $a \in F$.
Using~(\ref{W.T.R.P_111}) and~(\ref{W.T.R.P_222}) at the second step,
and using~(\ref{Eq_9526_DefEp}) at the third step,
we get
\[
\| a \af_g (d_1) - \af_g (d_1) a \|
 \leq 2 \| a \| \cdot \| \af_g (d_1) - d_g \| + \| a d_g - d_g a \|
 < (2 M + 1) \ep'
 \leq \ep_1.
\]
The choice of $\ep_1$
implies
\[
\| a e_g - e_g a \|
 = \| a s ( \af_g (d_1) ) - s ( \af_g (d_1) ) a \|
 < \frac{\ep}{2}
 < \ep
\]
and
\[
\| a f_g - f_g a \|
 = \frac{1}{\nu} \| a t ( \af_g (d_1) ) - t ( \af_g (d_1) ) a \|
 < 2 \left( \frac{\ep}{2} \right)
 = \ep,
\]
as desired.

To prove~(\ref{W.T.R.P.33}), we estimate,
using (\ref{W.T.R.P_222_prime}) at the third step
and (\ref{Eq_9526_DefRh}) at the last step
\begin{equation}\label{1.3.19.10}
\| (1 - f) - (1 - d) \|
 = \| d - f \|
  \leq \sum_{g \in G } \|  f_g - d_g \|
 < n \cdot 2 ( 1 + 3 n ) \rh
  < \ep.
\end{equation}
Therefore, using
\Lem{PhiB.Lem_18_4}(\ref{PhiB.Lem_18_4_10.a}) at the first step and
(\ref{W.T.R.P.333}) at the second step,
\[
(1 - f - \ep)_{+} \precsim_{A} 1 - d \precsim_{A} x.
\]

For~(\ref{W.T.R.P.44}), we estimate,
using part of~(\ref{1.3.19.10}) at the second step,
\[
\| d x d - f x f \|
 \leq \| d x \| \cdot \| d - f \| + \| d - f \| \cdot \| x f \|
 < 4 n ( 1 + 3 n ) \rh.
\]
Therefore, using (\ref{W.T.R.P.444}) at the second step,
\[
\| f x f \|
 > \| d x d \|
     - 4 n ( 1 + 3 n ) \rh
 > 1 - \ep'
     - 4 n ( 1 + 3 n ) \rh
  \geq 1 - \ep.
\]
This completes the proof.
\end{proof}

\begin{lem}\label{Invariant.Cuntz1}
Let $G$ be a finite group,
let $A$ be a \ca,
let $\alpha \colon G \to \Aut (A)$  be an action of
$G$ on $A$, and let $a, b \in \left(A^{\alpha}\right)_{+}$.
If $a \in \overline{b A b}$, then $a \in \overline{b A^{\alpha} b}$ and
$a \precsim_{A^{\alpha}} b$.
\end{lem}

\begin{proof}
Let $\ep > 0$.
Choose $c \in A$ such that
$\| a - b c b \| < \ep$.
Then
$\| a - b  \alpha_{g} ( c ) b \| < \ep$ for all $g \in G$.
So $d = \frac{1}{\card (G)} \sum_{g \in G} \alpha_{g} (c)$
satisfies
\[
d \in A^{\alpha}
\andeqn
\| a - b d b \|
\leq
 \frac{1}{\card (G)} \sum_{g \in G} \|  a - b \, \alpha_{g} (c) \, b \|
< \ep.
\]
Since $\ep > 0$ is arbitrary,
$a \in \overline{b A^{\alpha} b}$,
whence also $a \precsim_{A^{\alpha}} b$.
\end{proof}

\begin{lem}\label{Invariant.Cuntz2}
Let $\alpha \colon G \to \Aut (A)$
be an action of a finite group $G$ on a \ca{} $A$,
let $\dt \in \left( 0, \, [2 \, \card (G)]^{-2}\right)$, and let
$a, b \in (A^{\alpha})_{+}$ with $\| b \| = 1$ and  $\| a \| \leq 1$.
If $x$ is a positive element in $\overline{b A b}$ with
$a \precsim_{A} (x^2 - \tfrac{1}{2})_{+}$ and
$\| x \alpha_{g} (x) \| < \dt$ for all $g \in G \setminus \{ 1 \}$,
then there exists $t \in A^{\alpha}$ such that:
\begin{enumerate}
\item\label{1.2.19.a}
$ ( a - \dt^{1/2})_{+} \precsim_{A^{\alpha}} t t^*$.
\item\label{1.2.19.b}
$t^* t \in \overline{b A b}$.
\item\label{1.2.19.c}
$( a - \dt^{1/2})_{+} \precsim_{A^{\alpha}} b$.
\end{enumerate}
\end{lem}

\begin{proof}
To prove~(\ref{1.2.19.a}),
set $\eta = \sqrt{\dt} - 2 \, \card (G) \, \dt$.
Since $\dt \in \left( 0, \, [2 \, \card (G)]^{-2} \right)$,
it follows that $\eta > 0$.
Since $a \precsim_{A} (x^2 - \tfrac{1}{2})_{+}$,
by \Lem{A.G.J.P} there exists
$w \in A$ such that $\| a - w x^{2} w^* \| < \eta$
and $\| w \| \leq \sqrt{2}$.
Using \Lem{PhiB.Lem_18_4}(\ref{A.P.T}),
we find  $d \in A$ with $\| d \| \leq 1$ such that
$(a - \eta )_{+} = d w x^{2} w^* d^*$.
Set $v= d w$.
Then
\begin{equation}\label{1.2.19.1}
(a - \eta )_{+} = v x^{2} v^*
\andeqn
\| v \| \leq \sqrt{2}.
\end{equation}
Define
\[
t = \frac{1}{\card (G)^{1/2}} \sum_{g \in G} \alpha_{g} (v x )
\qquad
\mbox{and}
\qquad
s = \frac{1}{\card (G)} \sum_{g \neq h}
         \alpha_{g} (v x ) \, \alpha_{h} (x v^* ).
\]
Clearly $t \in A^{\alpha}$.

Now we claim that
$\| s \| < 2 \, \card (G) \, \dt$.
To prove the claim, for every $g, h \in G$ with $g \neq h$
we have $\| x \, \alpha_{g^{-1} h} (x) \| < \dt$.
So
\begin{equation}\label{1.2.19.3}
\| \alpha_{g} (x) \, \alpha_{h} (x) \| < \dt.
\end{equation}
Thus, using (\ref{1.2.19.1}) and (\ref{1.2.19.3}) at the second step,
\begin{align*}
\| s \|
& \leq
\frac{\| v \|^{2}}{\card (G)}
  \sum_{g \neq h} \| \alpha_{g} (x ) \, \alpha_{h} (x ) \|
< \left( \frac{2}{\card (G)} \right) \card (G)^{2} \dt
= 2 \, \card (G) \, \dt.
\end{align*}
The claim follows.

Using (\ref{1.2.19.1}) at the second step
and $a \in A^{\alpha}$ (so that $(a - \eta)_{+} \in A^{\alpha} $)
at the last step, we get
\begin{align*}
t t^{*}
& = \frac{1}{\card (G)}
  \Bigg( \sum_{g \in G} \alpha_{g} (v x^2 v^*) +
     \sum_{g \neq h} \alpha_{g} (v x) \, \alpha_{h} (x v^*) \Bigg)
\\
& = \frac{1}{\card (G)}
     \sum_{g \in G} \alpha_{g} \big( (a - \eta)_{+} \big)
    + \frac{1}{\card (G)}
      \sum_{g \neq h} \alpha_{g} (v x) \, \alpha_{h} (x v^*)
  = (a - \eta)_{+} + s.
\end{align*}
It follows that
$t t^* - (a - \eta)_{+} = s$.
Using the claim, we get
\[
\| t t^* - (a - \eta)_{+} \| < 2 \, \card (G) \, \dt.
\]
Therefore,
using \Lem{PhiB.Lem_18_4}(\ref{PhiB.Lem_18_4_8}) at the first step
and \Lem{PhiB.Lem_18_4}(\ref{PhiB.Lem_18_4_10.a}) at the second step,
\[
(a - \dt^{1/2}) = \big( (a - \eta)_{+} - 2 \, \card (G) \, \dt \big)_{+}
\precsim_{A^{\alpha}} t t^{*}.
\]
This is~(\ref{1.2.19.a}).

To prove~(\ref{1.2.19.b}), we claim that
$\alpha_{g} (x v^*) \, \alpha_{h} (v x) \in \overline{b A b}$
for all $g, h \in G$.
Since $x \in \overline{b A b}$,
there exists a sequence $(r_n)_{n \in \N}$ in $A$ such that
\begin{equation}\label{S.x.1.2.19}
x = \lim_{n \to \infty} b r_n b.
\end{equation}
So for all $g, h \in G$ we get,
using $b \in A^{\alpha}$ at the first step,
\begin{equation}\label{1.2.19.77}
\alpha_{g} (x v^*) \, \alpha_{h} (v x)
= \lim_{n \to \infty}
      b \, \alpha_{g} (r_n) \, b \, \alpha_{g} (v^*) \,
        \alpha_{h} (v) \, b \, \alpha_{h} (r_n) \, b.
\end{equation}
The claim is proved.

Use the definition of $t$ to compute
\begin{equation}\label{last.1.2.19.11}
t^* t
 = \frac{1}{\card (G)}
   \sum_{g,h \in G} \alpha_{g} (x v^*) \, \alpha_{h} (v x).
\end{equation}
By (\ref{1.2.19.77}) and (\ref{last.1.2.19.11}), we have
$t^* t \in \overline{b A b}$,
which is~(\ref{1.2.19.b}).

Finally, we prove (\ref{1.2.19.c}).
Since $b \in A^{\alpha}$ and
$t^* t \in \overline{b A b}$,
it follows from \Lem{Invariant.Cuntz1} that
$t^* t \precsim_{A^{\alpha}} b$.
Therefore, using (\ref{1.2.19.a}) at the first step and
\Lem{PhiB.Lem_18_4}(\ref{PhiB.Lem_18_4_4}) at the second step,
\[
(a - \dt^{1/2})_{+}
 \precsim_{A^{\alpha}} t t^* \sim_{A^{\alpha}} t^* t
 \precsim_{A^{\alpha}} b.
\]
This completes the proof of the lemma.
\end{proof}

\begin{lem}\label{approx.orthogonal_action}
Let $A$ be an infinite-dimensional simple unital \ca,
and let $\alpha \colon G \to \Aut (A)$ be an action
of a finite group $G$ on $A$
which has the weak tracial Rokhlin property.
Then for every $\ep > 0$ and
$b \in (A^{\alpha})_{+}$ with $\| b \| = 1$,
there is a positive element $x \in \overline{b A b}$
with $\| x \| = 1$ such that
$\| x \alpha_{g} (x) \| < \ep$
for all $g \in G \setminus \{ 1 \}$.
\end{lem}

\begin{proof}
We may assume $\ep < \frac{1}{2}$.
Set
\[
F = \{ b^2 \}
\qquad
\mbox{and}
\qquad
\ep' = \frac{\ep}{8 (1 + \card (G)^{2})}.
\]
Applying \Def{W_T_R_P_def} using $F$, $\ep'$, and $b^2$,
we get positive contractions
$f_g \in A$ for $g \in G$ such that,
with $f = \sum_{g \in G} f_g$, the following hold:
\begin{enumerate}
\item\label{Weak.T1}
$\| z f_g - f_g z \| < \ep'$ for all $g \in G$ and  $z\in F$.
\item\label{Weak.T2}
$\| \alpha_{g} ( f_h ) - f_{g h} \| < \ep'$ for all $g, h \in G$.
\item\label{Weak.T3}
$\|  f b^{2} f \| > 1 - \ep'$.
\end{enumerate}
So we have, using at the first step $g \neq h$
(so that $f_{g} f_{ h} = 0$), and
using (\ref{Weak.T1}) at the last step,
\[
\sum_{g \neq h} \| f_g b^{2} f_{h} \|
\leq \sum_{g \neq h} \| f_g\| \cdot \| b^{2} f_{h} - f_h b^2 \|
< \card (G)^{2} \ep'.
\]
Using this
and orthogonality of the elements $f_{g}$ for $g \in G$
at the last step,
we estimate
\[
\| f b^{2} f \|
\leq \Bigg\| \sum_{g \in G} f_g b^{2} f_g \Bigg\|
 + \Bigg\| \sum_{g \neq h} f_{g} b^{2} f_h  \Bigg\|
< \max_{g \in G} \| f_g b^{2} f_g \| + \card (G)^{2} \ep'.
\]
Therefore, by~(\ref{Weak.T3}),
\[
\max_{g \in G} \| f_g b^{2} f_g \|
> 1 - \ep' (1 + \card (G)^{2})
> 1 - \ep > \frac{1}{2}.
\]
It follows that
there exists $s \in G$ such that
$\| f_{s} b^{2} f_{s} \| > \frac{1}{2}$.
Set $y = b f_{s}^{2} b$.
Then
\[
\| y \| = \| b f_s f_s b \| = \| f_{s} b^{2} f_{s} \| > \frac{1}{2}.
\]

Now define $x = \| y \|^{-1} \cdot y$.
We claim that
$\| x \alpha_{g} (x) \| < \ep$
for all $g \in G \setminus \{ 1 \}$.
To prove the claim, using $b \in A^{\alpha}$ at the first step,
$g \neq 1$
(so that $f_{s} f_{ gs} = 0$) at the second step,
$\| b \| = 1$ at the third step, and
(\ref{Weak.T1}) and (\ref{Weak.T2}) at the last step, we estimate
\begin{align}
\| y \alpha_{g} (y) \|
& = \| b f_{s}^{2} b^{2} \alpha_{g} (f_{s}^{2}) \, b \|
\\ \notag
& \leq
   \| b f_{s} \| \cdot \| f_{s} b^{2} - b^{2} f_s \|
       \cdot \| \alpha_{g} (f_{s}^{2}) \, b \|
\\ \notag
& \qquad
  {\mbox{}} +
   \| b f_{s} b^{2} f_s \| \cdot\| \alpha_{g} (f_{s}) - f_{g s} \|
 \cdot\| \alpha_{g} (f_{s}) \, b \|
\\ \notag
& \leq \| f_{s} b^{2} - b^{2} f_{s} \|
     + \| \alpha_{g} (f_s) - f_{gs} \|
  < 2 \ep'.
\end{align}
Since $\| y \|^{-1} < 2$, we have
\[
\| x \alpha_{g} (x) \| =
\frac{1}{\| y \|^{2}} \cdot \| y  \alpha_{g} (y) \|
 < 8 \ep'
 < \ep.
\]
This completes the proof.
\end{proof}

\begin{lem}\label{W_T_R_P_inject}
Let $A$ be an infinite-dimensional simple unital \ca{}
and let $\alpha \colon G \to \Aut (A)$
be an action of a finite group $G$ on $A$
with the weak tracial Rokhlin property.
Let
$a, b \in (A^{\alpha})_{+}$,
and suppose that $0$ is a limit point of $\spec (b)$.
Then
$a \precsim_A b$ if and only if $a \precsim_{A^{\alpha}} b$.
\end{lem}

This result holds when $\af$ has the Rokhlin property,
without the requirement that $0$ be a limit point of $\spec (b)$;
in this case, $A$ can be any unital \ca.
See Theorem 4.1(ii) of~\cite{GdlStg}.

\begin{proof}[Proof of Lemma~\ref{W_T_R_P_inject}]
We need only prove the forwards implication.
So assume that $a \precsim_A b$.

We may assume
$\| a \| \leq 1$ and $\| b\| = 1$.
Let $\ep > 0$.
We may assume $\ep < [2 \card (G)]^{-2}$.
Since $a \precsim_A b$, there is $\dt > 0$ such that
$(a- \ep)_{+} \precsim_A (b-\dt)_{+}$.
We may require $\dt < 1$.
Set $a' = (a-\ep)_{+}$ and $b' = (b-\dt)_{+}$.
Choose $w \in A$
such that
$\big\| w b' w^* - a' \big\| < [40 \, \card (G)]^{-1} \ep$.
Since $b', a' \in A^{\alpha}$, it follows that
\begin{equation}\label{Eq_12_28_18_9}
\big\| \alpha_{g} (w) b' \alpha_{g} (w^*) - a' \big\|
   < \frac{\ep}{40\, \card (G)}
\end{equation}
for all $g \in G$.
Choose $\ld \in \spec (b) \cap (0, \dt)$.
Let $h \colon [0, \I) \to [0, 1]$
be a \cfn{}
such that $h (\ld) = 1$ and $\supp (h) \subseteq (0, \dt)$.
Then
\begin{equation}\label{Eq_9420_Prop_ha}
\| h (b) \| = 1,
\qquad
h (b) \perp b',
\andeqn
h (b) + b' \precsim_{A^{\alpha}} b.
\end{equation}

Applying \Lem{approx.orthogonal_action}
with $h (b)$ in place of $b$, we find
a positive element $x \in \overline{h (b) A h (b)}$
with $\| x \| = 1$ such that
\[
\big \| x \alpha_{g} (x) \big \| < \frac{\ep^{2}}{64}
\]
for all $g \in G \setminus \{ 1 \}$.
Now set
\[
F_0 = \{a', b', w, w^*\}
\andeqn
F = \bigcup_{g \in G} \af_g (F_0).
\]
Define
\[
s = \Big\| \Big( x^2 - \frac{1}{2} \Big)_{+} \Big\|^{-1}
      \cdot \Big( x^2 - \frac{1}{2} \Big)_{+}
\andeqn
\ep' = \frac{\ep}{40 ( \| w \| + 1 )^2 \card (G)^4}.
\]
Set $M = \max \bigl( 1, \, \max_{z \in F} \| z \| \bigr)$,
and use Lemma~\ref{L_9420_NearZero}
and Lemma~\ref{9625_FuncCalcCom}
to choose $\dt > 0$ such that the following hold:
\begin{enumerate}
\item\label{Item_9417_dtlessep}
$\dt \leq \ep'$.
\item\label{Item_9417_Prod}
If
$c, d \in A_{+}$ satisfy $\| c \|, \, \| d \| \leq 1$
and $\| c d \| < \dt$,
then $\big\| c^{1/2} d^{1/2} \big\| < \ep'$.
\item\label{Item_9417_Comm}
If
$c \in A_{+}$ satisfies $\| c \| \leq M$
and $z \in A$ satisfies $\| z \| \leq M$
and $\| c z - z c \| < \dt$,
then $\big\| c^{1/2} z - z c^{1/2} \big\| < \ep'$.
\setcounter{TmpEnumi}{\value{enumi}}
\end{enumerate}

Applying \Lem{invariant_contractions} with $F$ as given,
with $\dt$ in place of $\ep$, and with $s$ in place of $x$,
we get positive contractions
$f_g \in A$ for $g \in G$ such that,
with $f = \sum_{g \in G} f_g$, the following hold:
\begin{enumerate}
\setcounter{enumi}{\value{TmpEnumi}}
\item\label{WTRP4}
$\| f_g f_h \| < \dt$ for all $g, h \in G$ with $g \neq h$.
\item\label{WTRP1}
$\| z f_g - f_g z \| < \dt$ for all $g \in G$ and $z \in F$.
\item\label{WTRP2}
$\alpha_{g} ( f_h ) = f_{g h}$ for all $g, h \in G$.
\item\label{WTRP3}
$f \in A^{\alpha}$ and $\| f \| = 1$.
\item\label{WTRP5}
$(1 - f - \dt )_{+} \precsim_{A} s$.
\setcounter{TmpEnumi}{\value{enumi}}
\end{enumerate}
Then also:
\begin{enumerate}
\setcounter{enumi}{\value{TmpEnumi}}
\item\label{Item_9417_WTRP4_13}
$\big\| f_g^{1/2} f_h^{1/2} \big\| < \ep'$
for all $g, h \in G$ with $g \neq h$.
\item\label{Item_9417_WTRP1_13}
$\big\| z f_g^{1/2} - f_g^{1/2} z \big\| < \ep'$
for all $g \in G$ and $z \in F$.
\item\label{Item_9417_WTRP2_13}
$\alpha_{g} \big( f_h^{1/2} \big) = f_{g h}^{1/2}$
for all $g, h \in G$.
\end{enumerate}

Since $\dt \leq \ep' < \frac{\ep}{8}$, (\ref{WTRP5}) implies that
$\big(   1 - f - \frac{\ep}{8} \big)_{+}
 \precsim_{A} s
 \sim \big( x^2 - \tfrac{1}{2} \big)_{+}$.
Applying \Lem{Invariant.Cuntz2}(\ref{1.2.19.c})
with $h (b)$ in place of $b$,
with $\big( 1 - f - \frac{\ep}{8} \big)_{+}$
in place of $a$, with $x$ as given,
and with
$\ep^2 / 64$ in place of $\dt$, we get
\begin{equation}\label{Eq_8855}
\Big( 1 - f - \frac{\ep}{4} \Big)_{+}
 = \Bigg( \Big( 1 - f - \frac{\ep}{8} \Big)_{+}
     - \left( \frac{\ep^{2}}{64} \right)^{1/2} \Bigg)_{+}
 \precsim_{A^{\alpha}} h (b).
\end{equation}
Now define $v = \sum_{g \in G} \alpha_{g} (f_1 w)$.
Clearly $v \in A^{\alpha}$.
We claim that
\[
\| v b' v^* - f a' f \| < \frac{\ep}{4}.
\]

To prove the claim, define
\[
a_0 = \sum_{g \in G} f_g a' f_g,
\qquad
b_0 = \sum_{g \in G} f_g^{1/2} b' f_g^{1/2},
\andeqn
v_0 = \sum_{g \in G} \alpha_{g} (w) f_g^{1/2}.
\]
It is immediate that
\begin{equation}\label{Eq_9420_TrRokhlinNormEst_2}
\| v \|, \, \| v_0 \| \leq \card (G) \| w \|.
\end{equation}
Also set
\[
{\widetilde{f}} = \sum_{g \in G}  f_g^{1/2},
\]
giving $\bigl\| {\widetilde{f}} \bigr\| \leq \card (G)$.

For $g \in G$, by (\ref{Item_9417_dtlessep})
and~(\ref{WTRP1}) we have
$\| \alpha_{g} (w) f_g - f_g \alpha_{g} (w) \|
   < \ep'$.
Therefore,
for all $g \in G$,
using~(\ref{Eq_12_28_18_9})
on the last term at the second step
and $\| b' \| \leq 1$ at the last step,
\begin{align}\label{Eq_9421_Star}
& \bigl\| \big[ \alpha_{g} (w) f_g^{1/2} \big]
           \big[ f_g^{1/2} b' f_g^{1/2} \big]
           \big[  \alpha_{g} (w) f_g^{1/2} \big]^*
   - f_g a' f_g \bigr\|
\\
\notag
& \hspace*{3em} {\mbox{}}
\leq \| \alpha_{g} (w) f_g - f_g \alpha_{g} (w) \|
        \cdot \| b' \| \cdot \| f_g \alpha_{g} (w^*) \|
\\
\notag
& \hspace*{6em} {\mbox{}}
      + \| f_g \alpha_{g} (w) \|
        \cdot \| b' \|
        \cdot \| f_g \alpha_{g} (w^*)
            - \alpha_{g} (w^*) f_g \|
\\
\notag
& \hspace*{6em} {\mbox{}}
      + \| f_g \| \cdot \| \af_g (w) b' \af_g (w^*) - a' \|
            \cdot \| f_g \|
\\
\notag
& \hspace*{3em} {\mbox{}}
< 2 \| w \| \ep' + \frac{\ep}{40 \, \card (G)}.
\end{align}
Set $S = \{ (g, g, g) \colon g \in G \} \subseteq G^3$.
Using~(\ref{Eq_9421_Star}), (\ref{Item_9417_WTRP4_13}),
and $\| b' \| \leq 1$
at the second step,
we get
\begin{align}\label{Eq_9420_Primes}
\| v_0 b_0 v_0^* - a_0 \|
& \leq \sum_{g \in G}
   \bigl\| \big[  \alpha_{g} (w) f_g^{1/2} \big]
           \big[ f_g^{1/2} b' f_g^{1/2} \big]
           \big[  \alpha_{g} (w) f_g^{1/2} \big]^*
    - f_g a' f_g \bigr\|
\\
\notag
& \hspace*{2em} {\mbox{}}
      + \sum_{(g, t, h) \in G^3 \setminus S}
            \bigl\| \big[  \alpha_{g} (w) f_g^{1/2} \big]
            \big[ f_t^{1/2} b' f_t^{1/2} \big]
            \big[ \alpha_{h} (w) f_h^{1/2} \big]^* \bigl\|
\\
\notag
& < 2 \, \card (G) \| w \| \ep'
      + \frac{\ep}{40} + \card (G)^3 \| w \|^2 \ep'
  \leq \frac{4 \ep}{40}.
\end{align}

Next, we estimate,
using $f = \sum_{g \in G} f_g$ at the first step,
and using (\ref{Item_9417_dtlessep}),
(\ref{WTRP4}),
and~(\ref{WTRP1}) at the third step,
\begin{align}\label{Eq_9420_12_28_18_5_y}
\| a_0 - f a' f \|
& \leq \sum_{g \neq h} \| f_g a' f_h \|
  \leq \sum_{g \neq h}
      \big( \| f_g \| \cdot \| a' f_h - f_h a' \|
             + \| f_g f_h \| \cdot \| a' \| \big)
\\
\notag
& < 2 \, \card (G)^2 \ep'
  \leq \frac{2 \ep}{40}.
\end{align}
A similar calculation, this time using
(\ref{Item_9417_WTRP4_13}), (\ref{Item_9417_WTRP1_13}),
and~(\ref{Item_9417_dtlessep}),
gives
\begin{equation}\label{Eq_9420_12_28_18_5_y_Forb}
\big\| b_0 - {\widetilde{f}} b' {\widetilde{f}} \big\|
  < 2 \, \card (G)^2 \ep'.
\end{equation}

The next step is to estimate $\bigl\| v - v_0 {\widetilde{f}} \bigr\|$.
For all $g\in G$,
using (\ref{WTRP1}), (\ref{WTRP2}), and~(\ref{Item_9417_dtlessep}),
we get
\[
\| \alpha_{g} (f_1 w) - \alpha_{g} ( w ) f_g \|
  = \| \alpha_{g} (f_1 w - w f_1) \|
  < \dt
  \leq \ep'.
\]
Now, by~(\ref{Item_9417_WTRP1_13}),
\begin{align}\label{Eq_9419_N1}
\bigl\| v - v_0 {\widetilde{f}} \bigr\|
& \leq \sum_{g \in G}
       \big\| \alpha_{g} (f_1 w) - \alpha_{g} ( w ) f_g \big\|
     + \sum_{g \neq h}
       \big\| \alpha_{g} ( w ) f_g^{1/2} f_h^{1/2} \big\|
\\
\notag
& < \card (G) \ep' + \card (G)^2 \| w \| \ep'
  \leq \card (G)^2 (\| w \| + 1) \ep'.
\end{align}
It follows,
using this, (\ref{Eq_9420_12_28_18_5_y_Forb}),
and~(\ref{Eq_9420_TrRokhlinNormEst_2}) at the second step,
that
\begin{align*}
\| v b' v^* - v_0 b_0 v_0^* \|
& \leq \bigl\| v - v_0 {\widetilde{f}} \bigr\|
        \cdot \| b' \| \cdot \| v^* \|
       + \| v_0 \| \cdot \big\| {\widetilde{f}} \big\| \cdot \| b' \|
         \cdot \bigl\| \bigl( v - v_0 {\widetilde{f}} \bigr)^* \bigr\|
\\
\notag
& \hspace*{3em} {\mbox{}}
      + \| v_0 \|
         \cdot \big\| {\widetilde{f}} b' {\widetilde{f}} - b_0 \big\|
         \cdot \| v_0^* \|
\\
\notag
& < \card (G)^2 (\| w \| + 1) \ep' \cdot \card (G) \| w \|
\\
\notag
& \hspace*{3em} {\mbox{}}
      + \card (G)^2 \| w \| \cdot \card (G)^2 \ep' (\| w \| + 1)
\\
\notag
& \hspace*{3em} {\mbox{}}
      + \card (G) \| w \| \cdot 2 \, \card (G)^2 \ep'
             \cdot \card (G) \| w \|
\\
\notag
& \leq \frac{4 \ep}{40}.
\end{align*}

Combining this with (\ref{Eq_9420_Primes})
and~(\ref{Eq_9420_12_28_18_5_y}),
we now have
\begin{align*}
\| v b' v^* - f a' f \|
& \leq \| v b' v^* - v_0 b_0 v_0^* \|
             + \| v_0 b_0 v_0^* - a_0 \|
             + \| a_0 - f a' f \|
\\
& < \frac{4 \ep}{40} + \frac{4 \ep}{40} + \frac{2 \ep}{40}
  = \frac{\ep}{4}.
\end{align*}
The claim is proved.

The claim implies that
\begin{equation}\label{Eq.7744}
\Big( f a' f - \frac{\ep}{4} \Big)_{+}
 \precsim_{A^{\alpha}} v b' v^*
 \precsim_{A^{\alpha}} b'.
\end{equation}
Applying \Lem{Lem.ANP.Dec.18} with
$\frac{\ep}{4}$ in place of $\ep_1$,
with $\frac{\ep}{2}$ in place of $\ep_{2}$,
with $a'$ in place of $a$,
and with $1 - f$ in place of $g$, we get
\begin{equation}\label{Eq.6633}
\Big( a' - \frac{3 \ep}{4} \Big)_{+}
 \precsim_{A^{\alpha}}
   \Big(f a' f - \frac{\ep}{4} \Big)_{+}
      \oplus \Big( 1 - f - \frac{\ep}{4} \Big)_{+}.
\end{equation}
Using (\ref{Eq.6633}) at the second step,
(\ref{Eq.7744}) and (\ref{Eq_8855}) at the third step,
and (\ref{Eq_9420_Prop_ha}) at the last step, we have
\[
(a - \ep)_{+}
  = \Big( a' - \frac{3 \ep}{4} \Big)_{+}
  \precsim_{A^{\alpha}} \Big( f a' f - \frac{\ep}{4} \Big)_{+}
       \oplus \Big( 1 - f - \frac{\ep}{4} \Big)_{+}
  \precsim_{A^{\alpha}} b' \oplus h (b)
  \precsim_{A^{\alpha}} b.
\]
Therefore
$(a - \ep)_{+} \precsim_{A^{\alpha}} b$.
\end{proof}

\begin{dfn}\label{D_9421_Pure}
Let $A$ be a \ca.
Following the discussion before Corollary 2.24 of~\cite{APT11}
and Definition~3.1 of~\cite{Ph14},
with slight changes in notation,
we define
\[
A_{++} = \big\{ a \in A_{+} \colon
  \mbox{there is no projection $p \in M_{\infty} (A)$
         such that $\langle a \rangle_A = \langle p \rangle_A$} \big\},
\]
\[
\Cu_{+} (A)
 = \big\{ \langle a \rangle_A \colon a \in (K \otimes A)_{++} \},
\andeqn
\W_{+} (A)
 = \Cu_{+} (A) \cap \W (A).
\]
The elements of $A_{++}$ are called {\emph{purely positive}}.
\end{dfn}

We recall some properties of $\W_{+} (A)$ and $\Cu_{+} (A)$.

\begin{lem}\label{L_9421_PropOfPP}
Let $A$ be a stably finite simple \uca.
Then:
\begin{enumerate}
\item\label{Item_L_9421_PropOfPP_Diff}
$\Cu (A) \setminus \W (A) \subseteq \Cu_{+} (A)$.
\item\label{Item_L_9421_PropOfPP_ZInSp}
$(K \otimes A)_{++}
  = \big\{ a \in (K \otimes A)_{+} \colon
   \mbox{$0$ is a limit point of $\spec (a)$} \big\}$.
\item\label{Item_L_9421_PropOfPP_SbSGp}
$\W_{+} (A) \cup \{ 0 \}$ and $\Cu_{+} (A) \cup \{ 0 \}$
are unital subsemigroups of $\W (A)$ and $\Cu (A)$.
\item\label{Item_L_9421_PropOfPP_ClosedUnderSups}
Let $\et_1, \et_2, \ldots \in \Cu_{+} (A) \cup \{ 0 \}$
satisfy $\et_1 \leq \et_2 \leq \cdots$.
Then $\sup \big( \{ \et_n \colon n \in \Nz \} \big)$,
evaluated in $\Cu (A)$,
is in $\Cu_{+} (A) \cup \{ 0 \}$.
\end{enumerate}
\end{lem}

\begin{proof}
Parts (\ref{Item_L_9421_PropOfPP_Diff})
and~(\ref{Item_L_9421_PropOfPP_ZInSp})
are Lemma~3.2 of~\cite{Ph14}.
Part~(\ref{Item_L_9421_PropOfPP_SbSGp})
for $\W (A)$ is Corollary~2.9(i) of~\cite{PT}.
For $\Cu (A)$ it is Corollary~3.3 of~\cite{Ph14}.
Part~(\ref{Item_L_9421_PropOfPP_ClosedUnderSups})
is Lemma~3.5 of~\cite{Ph14}
(originally Parts (i) and~(iv) of Proposition~6.4 of~\cite{ERS11}).
\end{proof}

There are further interesting properties:
still assuming $A$ is stably finite and simple,
$\Cu_{+} (A) \cup \{ 0 \}$ is absorbing
(this follows from Corollary~3.3 of~\cite{Ph14})
and, if $A$ is not of type~I, has the same functionals as $\Cu (A)$
(Lemma~3.8 of~\cite{Ph14}).

We will use the following result several times.
The main work for the last sentence of the proof
is in~\cite{AS}.

\begin{lem}\label{L_9606_Spec01}
Let $A$ be an infinite-dimensional simple unital C*-algebra,
let $G$ be a finite group,
and let $\alpha \colon G \to \Aut (A)$
be an action of $G$ on $A$
which has the weak tracial Rokhlin property.
Then, for every $x \in (A^{\alpha})_{+} \setminus \{ 0 \}$,
there exists $c \in (A^{\alpha})_{+}$
such that $c \precsim_{A^{\alpha}} x$ and $\spec (c) = [0, 1]$.
\end{lem}

\begin{proof}
The algebra~$A$ is not type~I,
so Theorem~4.1 of~\cite{Rffl}
implies that $A^{\af}$ is not type~I.
Since $C^* (G, A, \af)$ is simple,
Lemma \ref{Fixedpoint_corner}(\ref{Fixedpoint_corner_d})
below
(or \cite{Ros79})
implies that $A^{\af}$ is simple.
Apply Lemma~2.1 of~\cite{Ph14} to ${\overline{x A^{\alpha} x}}$.
\end{proof}

\begin{lem}\label{WC_plus_injectivity}
Let $A$ be a stably finite simple unital \ca{} which is not of type~I
and let $\alpha \colon G \to \Aut (A)$
be an action of a finite group $G$ on $A$
with the weak tracial Rokhlin property.
Let $\iota \colon A^{\alpha} \to A$ be the inclusion map.
Then:
\begin{enumerate}
\item\label{WC_plus_injectivity_a}
The map $\W (\iota) \colon \W (A^{\alpha}) \to \W (A)$
induces an isomorphism of ordered semigroups
from $\W_{+} (A^{\alpha}) \cup \{ 0 \}$
to its image in $\W (A)$.
\item\label{WC_plus_injectivity_b}
The map $\Cu (\iota) \colon \Cu (A^{\alpha}) \to \Cu (A)$
induces an isomorphism of ordered semigroups
from $\Cu_{+} (A^{\alpha}) \cup \{ 0 \}$
to its image in $\Cu (A)$.
\end{enumerate}
\end{lem}

\begin{proof}
In both parts,
we need only prove injectivity and order isomorphism.

By Corollary 4.6 of~\cite{FG17}, for every $n \in \N$
the action $g \mapsto \id_{M_n} \otimes \af_g$
of $G$ on $M_n (A)$
has the weak tracial Rokhlin property.
With $\W_{+} (A^{\alpha})$
in place of $\W_{+} (A^{\alpha}) \cup \{ 0 \}$,
Part~(\ref{WC_plus_injectivity_a})
now follows from
Lemma \ref{L_9421_PropOfPP}(\ref{Item_L_9421_PropOfPP_ZInSp})
and Lemma~\ref{W_T_R_P_inject}.
Part~(\ref{WC_plus_injectivity_a}) as stated is then immediate.

We prove~(\ref{WC_plus_injectivity_b}).
It suffices to prove
that if $a, b \in ( K \otimes A^{\alpha})_{++}$
satisfy $a \precsim_A b$,
then $a \precsim_{A^{\af}} b$.
Let $\ep > 0$;
we prove that
$(a - \ep)_{+} \precsim_{A^{\af}} b$.

Choose $\dt > 0$ such that
\begin{equation}\label{Eq1_2019_04_14_wTRP}
\left( a - \frac{\ep}{3} \right)_{+} \precsim_A (b - \dt)_{+}.
\end{equation}
Choose $\ld \in \spec (b) \cap \big( 0, \frac{\dt}{3} \big)$.
Let $h \colon [0, \I) \to [0, 1]$
be a \cfn{}
such that $h (\ld) = 1$
and $\supp (h) \subseteq \big( 0, \frac{\dt}{3} \big)$.
Then
\begin{equation}\label{Eq_9421_Prop_hb}
\| h (b) \| = 1,
\qquad
h (b) \perp \Big( b - \frac{\dt}{3} \Big)_{+},
\andeqn
h (b) + \Big( b - \frac{\dt}{3} \Big)_{+} \precsim_{A^{\alpha}} b.
\end{equation}

Choose $n \in \N$ and $a_0, b_0, c_0 \in M_n (A^{\alpha})_{+}$
such that
\[
\bigg\| a_0 - \Big( a - \frac{\ep}{3} \Big)_{+} \bigg\|
   < \frac{\ep}{3},
\qquad
\bigg\| b_0 - \Big( b - \frac{\dt}{3} \Big)_{+} \bigg\|
   < \frac{\dt}{3},
\andeqn
\| c_0 - h (b) \| < \frac{1}{3}.
\]
It follows from Lemma \ref{PhiB.Lem_18_4}(\ref{Item_9420_LgSb_1_6})
that
\begin{equation}\label{Eq_9417_FromIneq_wTRP}
(a - \ep)_{+}
 \precsim_{A^{\af}} \Big( a_0 - \frac{\ep}{3} \Big)_{+}
 \precsim_{A^{\af}} \Big( a - \frac{\ep}{3} \Big)_{+}
\end{equation}
and
\begin{equation}\label{Eq_9417_FromIneq_2_wTRP}
(b - \dt)_{+}
 \precsim_{A^{\af}} \Big( b_0 - \frac{\dt}{3} \Big)_{+}
 \precsim_{A^{\af}} \Big( b - \frac{\dt}{3} \Big)_{+}.
\end{equation}
Set $c_1 = \big( c_0 - \frac{1}{3} \big)_{+}$.
Then $\| c_1 \| > \frac{1}{3}$, so $c_1 \neq 0$.
Since the action induced by $\af$ on $M_n (A)$
has the weak tracial Rokhlin property,
Lemma~\ref{L_9606_Spec01}
provides $c \in M_n (A^{\af})_{+}$
such that $c \precsim_{A^{\af}} c_1$ and $\spec (c) = [0, 1]$.
In particular,
\begin{equation}\label{Eq_9606_chb}
c \precsim_{A^{\af}} h (b).
\end{equation}

At the first step
combining the second part of~(\ref{Eq_9417_FromIneq_wTRP}),
(\ref{Eq1_2019_04_14_wTRP}),
and the first part of~(\ref{Eq_9417_FromIneq_2_wTRP}),
we get
\begin{equation}\label{Eq_9417_InA_wTRP}
\Big( a_0 - \frac{\ep}{3} \Big)_{+}
  \precsim_{A} \Big( b_0 - \frac{\dt}{3} \Big)_{+}
  \precsim_{A} \Big( b_0 - \frac{\dt}{3} \Big)_{+} \oplus c.
\end{equation}
Since
\[
a_0, \, \Big( b_0 - \frac{\dt}{3} \Big)_{+}, \, c
  \in \bigcup_{k = 1}^{\I} M_k (A^{\af}),
\]
because
$0$ is a limit point of the spectrum of
$\big( b_0 - \frac{\dt}{3} \big)_{+} \oplus c$,
and using
Lemma \ref{L_9421_PropOfPP}(\ref{Item_L_9421_PropOfPP_ZInSp}),
Part~(\ref{WC_plus_injectivity_a})
and~(\ref{Eq_9417_InA_wTRP})
imply
\begin{equation}\label{Eq_9423_InAaf_wTRP}
\Big( a_0 - \frac{\ep}{3} \Big)_{+}
  \precsim_{A^{\af}} \Big( b_0 - \frac{\dt}{3} \Big)_{+} \oplus c.
\end{equation}
Using, in order, the first part of~(\ref{Eq_9417_FromIneq_wTRP}),
(\ref{Eq_9423_InAaf_wTRP}),
(\ref{Eq_9606_chb})
and the second part of~(\ref{Eq_9417_FromIneq_2_wTRP}),
and~(\ref{Eq_9421_Prop_hb}),
we get
\[
(a - \ep)_{+}
 \precsim_{A^{\af}} \Big( a_0 - \frac{\ep}{3} \Big)_{+}
 \precsim_{A^{\af}} \Big( b_0 - \frac{\dt}{3} \Big)_{+} \oplus c
 \precsim_{A^{\af}} \Big( b - \frac{\dt}{3} \Big)_{+} \oplus h (b)
 \precsim_{A^{\af}} b.
\]
This completes the proof.
\end{proof}

Lemma~\ref{WC_plus_injectivity}
fails if we don't restrict to the purely positive elements.
See Example~\ref{E_9421_NotOnK0}.
We postpone this example,
since it uses Lemma~\ref{Fixedpoint_corner}.

\section{Radius of comparison of the fixed point algebra
  and crossed product}\label{Sec_rcCP}

\indent
In the next section,
we identify the range of the map
$\Cu_{+} (A^{\af}) \to \Cu (A)$
when $\af$ has the weak tracial Rokhlin property:
it is $\Cu_{+} (A)^{\af}$.
This information is not needed for our estimate
on the radius of comparison,
and does not seem to help
with the (still open) opposite inequality to the one we prove.
So we prove the radius of comparison results now.
Then we discuss what happens
under weaker hypotheses on the action,
and give the example promised at the end of Section~\ref{Sec_rcFix}.

\begin{thm}\label{Main.Thm1}
Let $G$ be a finite group,
let $A$ be
an infinite-dimensional stably finite simple unital C*-algebra,
and let $\alpha \colon G \to \Aut (A)$ be an action of
$G$ on $A$ which has the
weak tracial Rokhlin property.
Then
$\rc (A^{\alpha}) \leq \rc (A)$.
\end{thm}

\begin{proof}
We use Theorem 12.4.4(ii) of \cite{GKPT18}.
Thus, let $m, n \in \N$ satisfy $\frac{m}{n} > \rc (A)$.
Let $l \in \N$, and let $a, b \in ( A^{\alpha} \otimes M_l )_{+}$
with $\| a \| = \| b \| = 1$ satisfy
\[
(n + 1) \langle a \rangle_{A^{\alpha}}
    + m \langle 1 \rangle_{A^{\alpha}}
 \leq n \langle b \rangle_{A^{\alpha}}
\]
in $\W ({A^{\alpha}})$.
Corollary 4.6 of~\cite{FG17},
the action
$\alpha \otimes \id_{M_{l}}\colon G \to \Aut (A\otimes M_{l})$,
defined by
\[
\left(\alpha \otimes \id_{M_{l}} \right)_{g}
        (a \otimes (\lambda_{j, k})_{j, k = 1}^{n})
= \alpha_{g} (a) \otimes (\lambda_{j, k})_{j, k = 1}^{n},
\]
also has the weak tracial property.
We may therefore assume $l = 1$.

We must prove that $a \precsim_{A^{\alpha}} b$.
Moreover, by \Lem{PhiB.Lem_18_4}(\ref{PhiB.Lem_18_4_11.b}),
it is enough to show that for every $\ep > 0$
we have $(a - \ep)_{+} \precsim_{A^{\alpha}} b$.

So let $\ep > 0$.
\Wolog{} $\ep < \frac{1}{2}$.
Choose $k \in \N$ such that
\begin{equation}\label{Eq_9422_kmkn1}
\frac{k m}{k n + 1} > {\rc} (A).
\end{equation}
Then in $\W ({A^{\alpha}})$ we have
\[
(k n + 1) \langle a \rangle_{A^{\alpha}}
 + k m \langle 1 \rangle_{A^{\alpha}}
  \leq k (n + 1) \langle a \rangle_{A^{\alpha}}
     + k m \langle 1 \rangle_{A^{\alpha}}
  \leq k n \langle b \rangle_{A^{\alpha}}.
\]
Let $u \in M_{\infty} ({A^{\alpha}})_{+}$
be the direct sum of $k n + 1$ copies of~$a$,
let $z \in M_{\infty} ({A^{\alpha}})_{+}$
be the direct sum of $k n$ copies of~$b$,
and let $q \in M_{\infty} ({A^{\alpha}})_{+}$
be the direct sum of $k m$ copies of~$1_A$.
Then, by definition,
$u \oplus q \precsim_{A^{\alpha}} z$.
Therefore \Lem{PhiB.Lem_18_4}(\ref{PhiB.Lem_18_4_11.c})
provides $\dt > 0$ such that
$\big( u \oplus q - \ep \big)_{+}
 \precsim_{A^{\alpha}} (z - \dt)_{+}$.
 Since $\ep < \frac{1}{2}$, we have
 \[
( u \oplus q - \ep )_{+}
  = ( u - \ep )_{+}
     \oplus ( q - \ep )_{+}
  \sim_{A^{\alpha}} ( u - \ep )_{+} \oplus q,
\]
so
\[
(k n + 1)
    \langle ( a - \ep )_{+} \rangle_{A^{\alpha}}
              + k m \langle 1 \rangle_{A^{\alpha}}
  \leq k n \langle (b - \dt)_{+} \rangle_{A^{\alpha}}.
\]
Set $a' = (a - \ep)_{+}$ and $b' = (b - \dt)_{+}$.
Then
\begin{equation}\label{Eq_9606_Star}
(k n + 1)
    \langle a' \rangle_{A^{\alpha}}
              + k m \langle 1 \rangle_{A^{\alpha}}
  \leq k n \langle b' \rangle_{A^{\alpha}}.
\end{equation}
Lemma~2.7 of~\cite{Ph14}
provides positive elements $c \in {A^{\alpha}}$
and $y \in {A^{\alpha}} \setminus \{ 0 \}$
such that
\begin{equation}\label{Eq_5513_cyb}
k n \langle b' \rangle_{A^{\alpha}}
   \leq (k n + 1) \langle c \rangle_{A^{\alpha}}
\andeqn
\langle c \rangle_{A^{\alpha}} + \langle y \rangle_{A^{\alpha}}
  \leq \langle b \rangle_{A^{\alpha}}
\end{equation}
in $\W (A^{\alpha})$.
By Lemma~\ref{L_9606_Spec01},
there is $y_0 \in (A^{\af})_{+}$
such that $y_0 \precsim_{A^{\af}} y$
and $\spec (y_0) = [0, 1]$.
Replacing $y$ with this element,
we may assume that $y$ is purely positive.
By (\ref{Eq_9606_Star}) and~(\ref{Eq_5513_cyb}),
\[
(k n + 1)
  \langle a' \rangle_{A^{\alpha}}
              + k m \langle 1 \rangle_{A^{\alpha}}
  \leq (k n + 1) \langle c \rangle_{A^{\alpha}}.
\]
This relation also holds in $\W (A)$.
For $\ta \in \QT (A)$,
apply $d_{\ta}$ and divide by $k n + 1$
to get
\[
d_{\ta} ( a' ) + \frac{k m}{k n + 1}
\leq d_{\ta} (c).
\]
So $a' \precsim_A c$ by~(\ref{Eq_9422_kmkn1}).
Therefore,
using \Lem{W_T_R_P_inject}
with $c \oplus y$ in place of $b$ at the second step,
and using (\ref{Eq_5513_cyb}) at the third step,
\[
(a - \ep)_{+}
  = a' \precsim_{A^{\alpha}} c \oplus y
  \precsim_{A^{\alpha}} b.
\]
This completes the proof.
\end{proof}

Using~\cite{GdlStg},
we get the same conclusion for
Rokhlin actions on stably finite unital C*-algebras.

\begin{thm}\label{T_9802_RokhlinMain}
Let $G$ be a finite group,
let $A$ be a stably finite unital C*-algebra,
and let $\alpha \colon G \to \Aut (A)$ be an action of
$G$ on $A$ which has the Rokhlin property.
Then $\rc (A^{\alpha}) \leq \rc (A)$.
\end{thm}

\begin{proof}
We may clearly assume $\rc (A) < \infty$.
Let $r \in [0, \I)$ and suppose that $A$ has $r$-comparison.
Let $a, b \in M_{\I} (A^{\alpha})_{+}$
satisfy $d_{\ta} (a) + r < d_{\ta} (b)$
for all $\tau \in \QT (A^{\alpha})$.
Since every quasitrace on~$A$ restricts to a quasitrace on~$A^{\alpha}$,
we have $d_{\ta} (a) + r < d_{\ta} (b)$
for all $\ta \in \QT (A)$.
Since $A$ has $r$-comparison, we get $a \precsim_{A} b$.
Now $a \precsim_{A^{\alpha}} b$
by Theorem 4.1(ii) of~\cite{GdlStg}.
So $\rc (A^{\alpha}) \leq r$.
Taking the infimum over $r \in [0, \I)$
such that $A$ has $r$-comparison,
we get $\rc (A^{\alpha}) \leq \rc (A)$.
\end{proof}

We now turn to the radius of comparison of the crossed product.

Parts (\ref{Fixedpoint_corner_a})--(\ref{Fixedpoint_corner_d})
of the following lemma are originally taken from \cite{Ros79}.
Since some properties of the projection $p$ are
needed in our computations,
we give a more detailed statement.

\begin{lem}\label{Fixedpoint_corner}
Let $G$ be a finite group,
let $A$ be a unital \ca, and
let $\alpha \colon G \to \Aut (A)$  be an action of $G$ on~$A$.
Recalling from Notation~\ref{N_9408_StdNotation_CP}
that $(u_g)_{g \in G}$ is the family of standard unitaries
in $\CGAa$,
define
$p = \frac{1}{\card (G)} \sum_{g \in G} u_g$.
Then:
\begin{enumerate}
\item\label{Fixedpoint_corner_a}
$p$ is a projection in $\CGAa$.
\item\label{Fixedpoint_corner_b}
$p a p
 = \Big( \frac{1}{\card (G)} \sum_{g \in G} \alpha_{g} (a) \Big) p$
for all $a \in A$.
\item\label{Fixedpoint_corner_c}
If $a \in A^{\alpha}$, then $p a p = ap$.
\item\label{Fixedpoint_corner_d}
The map $a \mapsto a p$ is an isomorphism from
$A^{\alpha}$ to the corner $p \CGAa p$.
\item\label{Fixedpoint_corner_tsr}
If $\CGAa$ has stable rank one,
then $A^{\alpha}$ has stable rank one.
\item\label{Fixedpoint_corner_Rokhlin}
If $\af$ has the Rokhlin property,
then $p$ is full in $\CGAa$.
\end{enumerate}
\end{lem}

\begin{proof}
Parts (\ref{Fixedpoint_corner_a})--(\ref{Fixedpoint_corner_d})
are computations.
(Also see~\cite{Ros79}.)

Next,
if $\CGAa$ has stable rank one,
then Theorem 3.1.8 of~\cite{LnBook}
implies that $p \CGAa p$ has stable rank one,
so (\ref{Fixedpoint_corner_tsr})
follows from~(\ref{Fixedpoint_corner_d}).

For~(\ref{Fixedpoint_corner_Rokhlin}),
let $J \subseteq \CGAa$
be the closed ideal generated by~$p$,
and set $I = J \cap A$.
Let $E \colon \CGAa \to A$ be the standard
conditional expectation.
The proof of Proposition~10.3.13 of \cite{GKPT18}
shows that $E (J) \subseteq I$.
Since $E ( \card (G) \cdot p) = 1$,
we have $1 \in I$,
so $1 \in J$.
\end{proof}

The proof of the following lemma is easier,
and well known, for tracial states.
For example, the inequality~(\ref{QT_claim_d})
is trivial for tracial states,
but it seems to require some effort for quasitraces.

\begin{lem}\label{proj.quasitrace}
Let $G$ be a finite group,
let $A$ be an infinite-dimensional stably finite simple
unital \ca,
let $\alpha \colon G \to \Aut (A)$ be an action of
$G$ on $A$ which has the
weak tracial Rokhlin property,
and let $\tau \in \QT \big( \CGAa \big)$.
Let $p = \frac{1}{\card (G)} \sum_{g \in G} u_g$,
as in Lemma~\ref{Fixedpoint_corner}.
Then
$\tau (p) = \frac{1}{\card (G)}$.
\end{lem}

\begin{proof}
Let $\ep > 0$.
We show that
$\big| \frac{1}{\card (G)} - \tau (p) \big| < \ep$.
By Corollary~2.5 of~\cite{Ph14},
there is $a \in A_{+} \setminus \{ 0 \}$
such that for all $\rho \in {\operatorname{QT}} (A)$,
\begin{equation}\label{Eq_1.10.2019.1}
d_{\rho} (a) < \frac{\ep}{4}.
\end{equation}
Applying \Def{W_T_R_P_def} with $F= \varnothing$,
with $[32 \, \card (G)]^{-1} \ep$ in place of~$\ep$,
and with $\| a \|^{-1} \cdot a$ in place of~$x$,
we get orthogonal positive contractions
$f_g \in A$ for $g \in G$ such that,
with $f = \sum_{g \in G} f_g$, we have
\begin{equation}\label{Eq_9408_StSt}
1 - f \precsim_A a
\end{equation}
and
\[
\| \alpha_{g} ( f_h ) - f_{g h} \| < \frac{\ep}{32 \, \card (G)}
\]
for all $g, h \in G$.
This inequality,
together with $\| f_g \|, \| f_{g h} \| \leq 1$,
implies
\begin{align}\label{Eq_1.10.2019.14}
\| \alpha_{g} ( f_h^{2} ) - f_{g h}^{2} \|
& \leq
\| \alpha_{g} ( f_h) \| \cdot \| \alpha_{g} ( f_h ) - f_{g h} \|
    + \| \alpha_{g} ( f_h) - f_{g h} \| \cdot \| f_{g h} \|
\\
\notag
& < \frac{\ep}{16 \, \card (G)}.
\end{align}

Now we claim that the following hold:
\begin{equation}\label{QT_claim_a}
0 \leq \tau (1) - \tau (f^2) < \frac{\ep}{4},
\end{equation}
\begin{equation}\label{QT_claim_b}
0 \leq \tau (p) - \tau (p f^2 p) < \frac{\ep}{4},
\end{equation}
\begin{equation}\label{QT_claim_c}
\Bigg| \sum_{h \in G} \tau \big( p f_h^2 p \big)
    - \tau \bigg( \bigg[ \frac{1}{\card (G)} \bigg] f^2 \bigg) \Bigg|
  < \frac{\ep}{4},
\end{equation}
and
\begin{equation}\label{QT_claim_d}
\Bigg| \tau \Bigg( \sum_{h \in G} p f_h^2 p \Bigg)
  - \sum_{h \in G} \tau ( p f_h^2 p ) \Bigg|
 < \frac{\ep}{4}.
\end{equation}

We prove (\ref{QT_claim_a}).
Since $\spec (f) \subseteq [0, 1]$,
we have $1 - f^2 \sim_A 1 - f$,
so $1 - f^2 \precsim_{A} a$
by (\ref{Eq_9408_StSt}).
Clearly $\tau |_{A} \in \QT (A)$.
Therefore, using (\ref{Eq_1.10.2019.1}) at the last step,
\[
0 \leq \tau (1 - f^2)
  \leq d_{\tau} (1 - f^2)
  \leq d_{\tau} (a)
  < \frac{\ep}{4}.
\]
The relation (\ref{QT_claim_a})
follows because $1$ and $f^2$ commute.

To prove (\ref{QT_claim_b}), we start with
$(1 - f^2 )^{1/2} p (1 - f^2 )^{1/2} \leq  (1 - f^2)$.
Then, by \Prp{BH82_Cor_2_2_5}(\ref{BH82_Cor_2_2_5_b}),
\begin{equation}\label{Eq_1.10.2019.5}
\tau \big( (1 - f^2 )^{1/2} p (1 - f^2 )^{1/2} \big)
 \leq \tau (1 - f^2).
\end{equation}
Therefore, using $[p, \, p f^2 p ] = 0$ at the second step,
the trace property
(Definition \ref{quasitrace}(\ref{quasitrace.a}))
at the third step,
(\ref{Eq_1.10.2019.5}) at the fourth step,
and (\ref{QT_claim_a}) at the last step,
\[
0 \leq \tau (p ) - \tau (p f^2 p )
  = \tau \big( p(1 - f^2) p \big)
  = \tau \big( (1 - f^2)^{1/2} p (1 - f^2)^{1/2} \big)
  \leq \tau (1 - f^2)
  < \frac{\ep}{4}.
\]

For~(\ref{QT_claim_c}), first we estimate
\begin{align}\label{Eq_1.10.2019.8}
\biggl\| f_h p f_h - \frac{1}{\card (G)} f_h^{2} \biggr\|
& = \frac{1}{\card (G)} \Bigg\| \sum_{g \in G}
    f_h \alpha_{g} (f_h) u_{g} - \sum_{g \in G} f_h f_{g h} u_g \Bigg\|
\\
\notag
& \leq \frac{1}{\card (G)} \sum_{g \in G}
     \| f_h \| \cdot \| \alpha_{g} (f_h) - f_{g h} \|
< \frac{\ep}{32 \, \card (G) }.
\end{align}
Therefore, using (\ref{Eq_1.10.2019.8}) at the last step,
\[
\Bigg\| \sum_{h \in G} f_h p f_h - \frac{1}{\card (G)} f^2 \Bigg\|
 \leq \sum_{h \in G}
     \biggl\| f_h p f_h - \frac{1}{\card (G)} f^{2}_{h} \biggr\|
 < \frac{\ep}{32}
 < \frac{\ep}{4}.
\]
Now use Proposition \ref{BH82_Cor_2_2_5}(\ref{BH82_Cor_2_2_5_c})
and $N (\ta) = 1$
to get
\[
\Bigg| \ta \Bigg( \sum_{h \in G} f_h p f_h \Bigg)
    - \ta \bigg( \frac{1}{\card (G)} f^2 \bigg) \Bigg|
 < \frac{\ep}{4}.
\]
We have
$\ta \left( \sum_{h \in G} f_h p f_h \right)
  = \sum_{h \in G} \ta (f_h p f_h)$
since the elements $f_h p f_h$,
for $h \in G$, commute with each other.
The trace property
(Definition \ref{quasitrace}(\ref{quasitrace.a}))
gives $\ta (f_h p f_h) = \ta (p f_h^2 p )$ for $h \in G$.
This completes the proof of~(\ref{QT_claim_c}).

To prove (\ref{QT_claim_d}),
set $b = \sum_{g \in G} \alpha_{g} (f^{2}_1)$.
Then, for $h \in G$, using (\ref{Eq_1.10.2019.14}) at the second step,
\begin{align}\label{Eq_1.11.2019.10}
\Bigg\| \sum_{g \in G} \alpha_{g} (f_h^{2}) - b \Bigg\|
& \leq \sum_{g \in G}
      \bigl\| \alpha_{g} (f_h^{2}) - f_{g h}^{2} \bigr\|
  + \sum_{g \in G} \bigl\| f^{2}_{g} - \alpha_{g} (f^{2}_1) \bigr\|
\\
\notag
& <  2 \, \card (G) \bigg( \frac{\ep}{16 \, \card (G)} \bigg)
  = \frac{\ep}{8}.
\end{align}
Next,
using \Lem{Fixedpoint_corner}(\ref{Fixedpoint_corner_b})
at the first step
and (\ref{Eq_1.11.2019.10}) at the last step,
\begin{align}\label{EQQ.1.11.2019.12}
\biggl\| p f_h^{2} p  - \frac{1}{\card (G)} b p \biggr\|
& = \Bigg\| \Bigg( \frac{1}{\card (G)}
                \sum_{g \in G} \alpha_{g} (f^{2}_{h}) \Bigg) p
           - \frac{1}{\card (G)} b p \Bigg\|
\\ \notag
& \leq \frac{\| p \|}{\card (G)} \Bigg\| \sum_{g \in G}
     \alpha_{g} (f^{2}_{h}) - b \Bigg\|
  < \frac{\ep}{8 \, \card (G)}.
\end{align}
Then, using (\ref{EQQ.1.11.2019.12}) at the last step,
\begin{equation}\label{Eq_1.11.2019.32}
\Biggl\| \sum_{h \in G} p f_h^{2} p  - b p \Biggr\|
\leq \sum_{h \in G}
     \biggl\|  p f_h^{2} p - \frac{1}{\card (G)} b p \biggr\|
< \frac{\ep}{8}.
\end{equation}
Finally, we get, using $p b p = b p$
(by \Lem{Fixedpoint_corner}(\ref{Fixedpoint_corner_c}),
since $b \in A^{\af}$), $N (\tau) = 1$,
and Proposition \ref{BH82_Cor_2_2_5}(\ref{BH82_Cor_2_2_5_c})
at the second step,
and using
(\ref{EQQ.1.11.2019.12}) and  (\ref{Eq_1.11.2019.32}) at the third step,
\begin{align*}
& \Bigg| \tau \Bigg( \sum_{h \in G} p f_h^{2} p \Bigg)
          - \sum_{h \in G} \tau ( p f_h^{2} p) \Bigg|
\\
& \qquad
 \leq \Bigg| \tau \Bigg( \sum_{h \in G} p f_h^{2} p \Bigg)
       - \tau ( b p) \Bigg|
     + \Bigg| \sum_{h \in G} \frac{1}{\card (G)} \tau ( b p)
    - \sum_{h \in G} \tau ( p f_h^{2} p) \Bigg|
\\
& \qquad
 \leq \Bigg\| \sum_{h \in G} p f_h^{2} p - b p \Bigg\|
  + \sum_{h \in G} \Big\| \frac{1}{\card (G)} b p - p f_h^{2} p \Big\|
\\
& \qquad
 < \frac{\ep}{8} + \card (G)
    \left( \frac{\ep}{8 \, \card (G)} \right)
 = \frac{\ep}{4}.
\end{align*}
This completes the proof of the claim.

Now we estimate,
using (\ref{QT_claim_a}), (\ref{QT_claim_c}), (\ref{QT_claim_d}),
and (\ref{QT_claim_b}) at the second step,
\begin{align*}
& \Big| \frac{1}{\card (G)} - \tau (p) \Big|
\\
& \hspace*{3em} {\mbox{}}
  \leq \Big| \Bigl( \frac{1}{\card (G)} \Bigr) \tau (1)
      - \Big( \frac{1}{\card (G)} \Big) \tau (f^2) \Big|
    + \Bigg| \tau \Big( \Big[ \frac{1}{\card (G)} \Big] f^2 \Big)
    - \sum_{h \in G} \tau ( p f_h^2 p ) \Bigg|
\\
& \hspace*{6em} {\mbox{}}
 + \Bigg| \sum_{h \in G} \tau ( p f_h^2 p )
       - \tau \Bigg( \sum_{h \in G} p f^{2}_{h} p \Bigg) \Bigg|
 + \bigl| \tau (p f^2 p) - \tau (p) \bigr|
\\
& \hspace*{3em} {\mbox{}}
  < \frac{\ep}{4 \, \card (G)}
      + \frac{\ep}{4} + \frac{\ep}{4} + \frac{\ep}{4}
  \leq \ep.
\end{align*}
This completes the proof.
\end{proof}

\begin{thm}\label{T_9412_RcCrPrd}
Let $G$ be a finite group,
let $A$ be
an infinite-dimensional stably finite simple unital C*-algebra,
and let $\alpha \colon G \to \Aut (A)$  be an action of
$G$ on $A$ which has the
weak tracial Rokhlin property.
Then
\[
\rc \big( \CGAa \big) = \frac{1}{\card (G)} \cdot \rc (A^{\af})
\quad {\mbox{and}} \quad
\rc \big( \CGAa \big) \leq \frac{1}{\card (G)} \cdot \rc (A).
\]
\end{thm}

\begin{proof}
By \Lem{proj.quasitrace}
and Lemma \ref{Fixedpoint_corner}(\ref{Fixedpoint_corner_d}),
the projection $p \in \CGAa$ of Lemma \ref{Fixedpoint_corner}
satisfies $\ta (p) = \card (G)^{-1}$
for all $\ta \in \QT ( \CGAa )$
and $A^{\alpha} \cong p \CGAa p$.
The algebra $\CGAa$ is simple
by Corollary 3.3 of~\cite{FG17}.
So $p$ is full.
Now $\CGAa$ is stably finite
(being stably isomorphic to $A^{\af} \subseteq A$),
so \Thm{Ourcornertheorem} implies that
$\rc \big( \CGAa \big) = \card (G)^{-1} \rc ( A^{\alpha})$.
This is the first part of the conclusion.
The second part now follows from \Thm{Main.Thm1}.
\end{proof}

We get the same outcome with the Rokhlin property
and for any stably finite unital C*-algebra,
not necessarily simple.

\begin{thm}\label{T_9803_RcCrPrdRokhlin}
Let $G$ be a finite group,
let $A$ be a stably finite unital C*-algebra,
and let $\alpha \colon G \to \Aut (A)$ be an action of
$G$ on $A$ which has the Rokhlin property.
Then
\[
\rc \big( \CGAa \big) = \frac{1}{\card (G)} \cdot \rc (A^{\af})
\quad {\mbox{and}} \quad
\rc \big( \CGAa \big) \leq \frac{1}{\card (G)} \cdot \rc (A).
\]
\end{thm}

\begin{proof}
The proof is the same as that of Theorem~\ref{T_9412_RcCrPrd},
except that we now use
Lemma \ref{Fixedpoint_corner}(\ref{Fixedpoint_corner_Rokhlin})
rather than simplicity of $\CGAa$
to deduce that $p$ is full,
and we use Theorem~\ref{T_9802_RokhlinMain}
instead of Theorem~\ref{Main.Thm1} at the end.
\end{proof}

If $G = \Z / 2 \Z$ and $\af$ is the trivial action,
then the conclusions of Theorem~\ref{Main.Thm1}
and Theorem~\ref{T_9802_RokhlinMain} hold
(because $A^{\af} = A$)
but the conclusions of Theorem~\ref{T_9412_RcCrPrd}
and Theorem~\ref{T_9803_RcCrPrdRokhlin}
generally fail
(because $\CGAa \cong A \oplus A$
and $\rc (A \oplus A) = \rc (A)$).
For pointwise outer actions~$\af$,
in fact the conclusions of all these theorems can fail.
See Example~\ref{R_9412_TrRPNeeded}.

\begin{exa}\label{E_9421_NotOnK0}
We give an example
of a stably finite simple separable \uca~$D$ which is not of type~I
and an action $\alpha \colon \Z / 2 \Z \to \Aut (D)$
such that $\af$ has the weak tracial Rokhlin property
but such that the map $\W (\iota) \colon \W (D^{\alpha}) \to \W (D)$
of Lemma~\ref{WC_plus_injectivity}
is not injective.
This example also shows that
Lemma~\ref{W_T_R_P_inject} fails when $0$ is not a limit point
of $\spec (b)$.
Our algebra $D$ is in fact a UHF algebra,
and $\af$ actually has the tracial Rokhlin property.
This example is therefore a counterexample to
Proposition~6.2 and Corollary~6.3 of~\cite{OskTry2}.
(The mistake in~\cite{OskTry2} is in the use of $g_{\dt} (b)$
in the proof of Proposition~6.2 of~\cite{OskTry2}.
Since $g_{\dt} (0) \neq 0$,
$g_{\dt} (b) \not\in {\overline{b P b}}$.)

Let $D$ and $\af$ be as in Example~2.8 of~\cite{PhT4}.
Let $\bt \in \Aut \bigl( C^* (\Z / 2 \Z, D, \af) \bigr)$
be the automorphism which generates the dual action.
As shown there,
$\af$ has the tracial Rokhlin property
but not the Rokhlin property.
The algebra $C^* (\Z / 2 \Z, D, \af)$ has a unique tracial state,
which we call~$\ta$.
It is clearly $\bt$-invariant.
The algebra $D$ also has a unique tracial state~$\sm$;
necessarily $\sm = \ta |_D$.
Moreover,
there is $\et_0 \in K_0 \bigl( C^* (\Z / 2 \Z, D, \af) \bigr)$
such that $\bt_* (\et_0) \neq \et_0$.

Set $\et = \et_0 - \bt_* (\et_0)$.
Then $\et \neq 0$,
but,
since $\ta \circ \bt = \ta$,
we have $\ta_* (\et) = 0$.
It follows from
Lemma \ref{Fixedpoint_corner}(\ref{Fixedpoint_corner_d})
that $D^{\af}$ is isomorphic
to a full corner of $C^* (\Z / 2 \Z, D, \af)$.
Thus,
except for the $K_0$-class of the identity element,
the Elliott invariants of $D^{\af}$ and $C^* (\Z / 2 \Z, D, \af)$
are isomorphic.
In particular,
$D^{\af}$ has a unique tracial state~$\rh$
(necessarily equal to $\sm |_{D^{\af}}$),
and there is $\mu \in K_0 (D^{\af}) \setminus \{ 0 \}$
such that $\rh_* (\mu) = 0$.

Choose projections $p, q \in K \otimes D^{\af}$
such that $\mu = [p] - [q]$ in $K_0 (D^{\af})$.
Since $[p] \neq [q]$ and $D^{\af}$ is stably finite,
it follows that $\langle p \rangle \neq \langle q \rangle$
in $\Cu (D^{\af})$.
In fact, they are in $W (D^{\af})$.
Let $\io \colon D^{\af} \to D$ be the inclusion map.
Then $\sm (\io (p)) = \sm (\io (q))$.
Since $D$ is a UHF algebra,
this implies that $\io_* ([p]) = \io_* ([q])$ in $K_0 (D)$.
Therefore
$\W (\io) (\langle p \rangle) = \W (\io) (\langle q \rangle)$.
Thus $\W (\io)$ is not injective.
Also, $p \not\precsim_{D^{\af}} q$
but $p \precsim_{D} q$.
\end{exa}


\section{Surjectivity of
  $\Cu_{+} (A^{\af}) \to \Cu_{+} (A)^{\af}$}\label{Sec_Surj}

In this section,
we prove that if $G$ is finite,
$A$ is unital, stably finite, and simple,
and $\alpha \colon G \to \Aut (A)$
has the weak tracial Rokhlin property,
then the inclusion $A^{\af} \to A$
induces an isomorphism of the ordered semigroups
of purely positive elements
$\Cu_{+} (A^{\af}) \cup \{ 0 \} \to \Cu_{+} (A)^{\af} \cup \{ 0 \}$.
If we assume stable rank one,
then the conclusion is valid
for $\W_{+} (A^{\af})$ and $\W_{+} (A)^{\af}$
as well.
We also give the corresponding result for
$\W (A^{\af})$ when $A$ is merely unital
but $\af$ is assumed to have the Rokhlin property.
In this case, we need not discard the classes of the \pj{s},
just like in Theorem 4.1(ii) of~\cite{GdlStg}
for $\Cu (A^{\af})$.

Injectivity was proved in Section~\ref{Sec_rcFix};
the content of this section is the proof of surjectivity.

The next lemma produces the following chain of subequivalences,
for any $g \in G$:
\begin{align*}
(a - \ep)_{+}
& \precsim_A (a' - \dt_6)_{+}
  \precsim_A (a' - \dt_5)_{+}
  \precsim_A (\af_g (a') - \dt_4)_{+}
\\
& \precsim_A (\af_g (a') - \dt_2)_{+}
  \precsim_A (a' - \dt_1)_{+}
  \precsim_A a'
  \precsim_A (a - \dt)_{+}.
\end{align*}

\begin{lem}\label{L_9529_ApproxAndSubeq}
Let $A$ be a \ca,
and let $\alpha \colon G \to \Aut (A)$ be an action
of a finite group $G$ on~$A$.
Let $a \in (K \otimes A)_{+}$
satisfy $a \sim_A \alpha_{g} (a)$ for all $g \in G$
and $\| a \| \leq 1$.
Then for every $\ep > 0$
there are $m \in \N$,
$\dt, \dt_{1}, \dt_2, \ldots, \dt_6 > 0$,
and $a' \in M_m (A)_{+}$ with $\| a' \| \leq \| a \|$,
such that:
\begin{enumerate}
\item\label{L_9529_ApproxAndSubeq_DtOrder}
$0 < \dt < \dt_1
 < \dt_2
 < \dt_3
 < \dt_4
 < \dt_5
 < \dt_6
 < \ep$.
\item\label{L_9529_ApproxAndSubeq_PrDt}
$a' \precsim_A (a - \dt)_{+}$.
\item\label{L_9529_ApproxAndSubeq_21}
$(\af_g (a') - \dt_2)_{+} \precsim_A (a' - \dt_1)_{+}$
for all $g \in G$.
\item\label{L_9529_ApproxAndSubeq_54}
$(a' - \dt_5)_{+} \precsim_A (\af_g (a') - \dt_4)_{+}$
for all $g \in G$.
\item\label{L_9529_ApproxAndSubeq_Ep6}
$(a - \ep)_{+} \precsim_A (a' - \dt_6)_{+}$.
\setcounter{TmpEnumi}{\value{enumi}}
\end{enumerate}
If, in addition, $0$ is a limit point of $\spec (a)$,
then we may also require:
\begin{enumerate}
\setcounter{enumi}{\value{TmpEnumi}}
\item\label{L_9529_ApproxAndSubeq_NZero}
$\spec (a') \cap (0, \dt) \neq \E$.
\end{enumerate}
\end{lem}

\begin{proof}
We may clearly assume that $a \neq 0$.

Let $\ep > 0$.
First use $\af_g (a) \sim_A a$ for $g \in G$
to choose $\bt_1 > 0$ such that $\bt_1 < \frac{\ep}{4}$
and for all $g \in G$ we have
$\bigl( a - \frac{\ep}{4} \bigr)_{+}
  \precsim_A ( \af_g (a) - \bt_1)_{+}$.
Similarly,
choose $\bt_2 > 0$ such that $\bt_2 < \frac{\bt_1}{4}$
and for all $g \in G$ we have
$\bigl( \af_g (a) - \frac{\bt_1}{4} \bigr)_{+}
  \precsim_A (a - \bt_2)_{+}$.
Set
\[
M = \max \bigl( 1, \, 2 \bt_1^{-1/2}, \, 2 \bt_2^{-1/2} \bigr)
\andeqn
\gm = \min \left( \frac{\bt_1}{12 M^2}, \, \frac{\ep}{12 M^2} \right).
\]
Then, for $g \in G$,
by Lemma~\ref{A.G.J.P} there are
$s_g, t_g \in K \otimes A$
such that
\begin{equation}\label{Eq_9519_New2019_03_22}
\Bigl\| s_g \Bigl( \af_g (a) - \frac{3 \bt_1}{4} \Bigr)_{+} s_g^*
    - \Bigl( a - \frac{\ep}{4} \Bigr)_{+} \Bigr\|
< \gm
\andeqn
\| s_g \| \leq 2 \bt_1^{-1/2}
\end{equation}
and
\begin{equation}\label{Eq1_9519_New2019_03_23}
\Bigl\| t_g \Bigl( a - \frac{3 \bt_2}{4} \Bigr)_{+} t_g^*
    - \Bigl( \af_g (a) - \frac{\bt_1}{4} \Bigr)_{+} \Bigr\|
< \gm
\andeqn
\| t_g \| \leq 2 \bt_2^{-1/2}.
\end{equation}
Use Lemma 12.4.5 of~\cite{GKPT18}
to choose $\mu > 0$ so small that whenever $B$ is a \ca{}
and $b, c \in B$
satisfy $0 \leq b, c \leq 1$ and $\| b - c \| < \mu$,
then
\begin{equation}\label{Eq_9607_1stEst}
\Bigl\| \Bigl( b - \frac{\bt_1}{4} \Bigr)_{+}
    - \Bigl( c - \frac{\bt_1}{4} \Bigr)_{+} \Bigr\| < \gm,
\qquad
\Bigl\| \Bigl( b - \frac{3 \bt_2}{4} \Bigr)_{+}
    - \Bigl( c - \frac{3 \bt_2}{4} \Bigr)_{+} \Bigr\| < \gm,
\end{equation}
\begin{equation}\label{Eq_9607_2ndEst}
\Bigl\| \Bigl( b - \frac{3 \bt_1}{4} \Bigr)_{+}
    - \Bigl( c - \frac{3 \bt_1}{4} \Bigr)_{+} \Bigr\| < \gm,
\quad {\mbox{and}} \quad
\Bigl\| \Bigl( b - \frac{\ep}{4} \Bigr)_{+}
    - \Bigl( c - \frac{\ep}{4} \Bigr)_{+} \Bigr\| < \gm.
\end{equation}
Define
\[
\nu = \min \left( \frac{\gm}{3}, \,
     \frac{\mu}{3}, \, \frac{\bt_2}{4}, \, \frac{\ep}{12} \right).
\]

If we do not need Part~(\ref{L_9529_ApproxAndSubeq_NZero})
of the conclusion,
simply take $\dt = \nu$.
Otherwise,
since $0$ is a limit point of $\spec (a)$,
we can choose $\dt > 0$
such that $\dt \leq \nu$
and $\frac{5 \dt}{2} \in \spec (a)$.
Choose $m \in \N$
such that there is $b \in M_m (A)_{+}$
with $\| b \| = \| a \|$ and $\| b - a \| < \frac{\dt}{2}$.
Define $a' = (b - 2 \dt)_{+}$.
Since $\| b - a \| < \dt$, it follows
from Lemma \ref{PhiB.Lem_18_4}(\ref{Item_9420_LgSb_1_6})
that $a' \precsim_A (a - \dt)_{+}$,
which is~(\ref{L_9529_ApproxAndSubeq_PrDt}).
If $0$ is a limit point of $\spec (a)$,
then we arranged to have $\frac{5 \dt}{2} \in \spec (a)$,
so $\| b - a \| < \frac{\dt}{2}$ implies
$\spec (b) \cap (2 \dt, 3 \dt) \neq \E$.
Thus $\spec (a') \cap (0, \dt) \neq \E$.
This is~(\ref{L_9529_ApproxAndSubeq_NZero}).

We have $\| a' - a \| < 3 \dt \leq \mu$.
For $g \in G$,
it therefore follows from (\ref{Eq_9519_New2019_03_22}),
using (\ref{Eq_9607_2ndEst}) and $\| s_g \| \leq M$,
that
\begin{align}\label{Eq_9519_New_aPr}
& \Bigl\| s_g \Bigl( \af_g (a') - \frac{3 \bt_1}{4} \Bigr)_{+} s_g^*
    - \Bigl( a' - \frac{\ep}{4} \Bigr)_{+} \Bigr\|
\\
\notag
& \hspace*{1em} {\mbox{}}
\leq
 \| s_g \|^2
        \Bigl\| \Bigl( a' - \frac{3 \bt_1}{4} \Bigr)_{+}
           - \Bigl( a - \frac{3 \bt_1}{4} \Bigr)_{+} \Bigr\|
\\
\notag
& \hspace*{2em} {\mbox{}}
    + \Bigl\| s_g \Bigl( \af_g (a) - \frac{3 \bt_1}{4} \Bigr)_{+} s_g^*
          - \Bigl( a - \frac{\ep}{4} \Bigr)_{+} \Bigr\|
    + \Bigl\| \Bigl( a - \frac{\ep}{4} \Bigr)_{+}
          - \Bigl( a' - \frac{\ep}{4} \Bigr)_{+} \Bigr\|
\\
\notag
& \hspace*{1em} {\mbox{}}
< M^2 \gm + \gm + \gm
\leq 3 M^2 \gm,
\end{align}
and similarly,
using~(\ref{Eq1_9519_New2019_03_23}),
(\ref{Eq_9607_1stEst}), and $\| t_g \| \leq M$,
\begin{equation}\label{Eq1_9519_New_aPrime}
\Bigl\| t_g \Bigl( a' - \frac{3 \bt_2}{4} \Bigr)_{+} t_g^*
    - \Bigl( \af_g (a') - \frac{\bt_1}{4} \Bigr)_{+} \Bigr\|
\leq 3 M^2 \gm.
\end{equation}

Define
\[
\dt_1 = \frac{3 \bt_2}{4},
\quad
\dt_2 = \frac{\bt_1}{2},
\quad
\dt_3 = \frac{5 \bt_1}{8},
\quad
\dt_4 = \frac{3 \bt_1}{4},
\quad
\dt_5 = \frac{\ep}{2},
\quad {\mbox{and}} \quad
\dt_6 = \frac{3 \ep}{4}.
\]
Then (\ref{L_9529_ApproxAndSubeq_DtOrder}) is clear.
Moreover, we have
\[
\frac{\ep}{4} + 3 M^2 \gm \leq \frac{\ep}{2} = \dt_5
\andeqn
\frac{\bt_1}{4} + 3 M^2 \gm \leq \frac{\bt_1}{2} = \dt_2.
\]
So (\ref{Eq_9519_New_aPr})
and~(\ref{Eq1_9519_New_aPrime})
imply, for $g \in G$,
\[
(a' - \dt_5)_{+}
 \leq \Bigl( a' - \Bigl( \frac{\ep}{4} + 3 M^2 \gm \Bigr) \Bigr)_{+}
 \precsim_A (\af_g (a') - \dt_4)_{+},
\]
which is~(\ref{L_9529_ApproxAndSubeq_54}),
and
\[
(\af_g (a') - \dt_2)_{+}
 \leq \Bigl( \af_g (a')
       - \Bigl( \frac{\bt_1}{4} + 3 M^2 \gm \Bigr) \Bigr)_{+}
 \precsim_A (a' - \dt_1)_{+},
\]
which is~(\ref{L_9529_ApproxAndSubeq_21}).
Finally,
$\| a' - a \| < 3 \dt \leq \frac{\ep}{4}$,
so by Lemma \ref{PhiB.Lem_18_4}(\ref{Item_9420_LgSb_1_6}) we have
$(a - \ep)_{+}
 \precsim_A \bigl( a' - \frac{3 \ep}{4} \bigr)_{+}
 = (a' - \dt_6)_{+}$,
which is~(\ref{L_9529_ApproxAndSubeq_Ep6}).
This completes the proof.
\end{proof}

\begin{lem}\label{Small_fixed_element}
Let $A$ be a simple \ca{} which is not of type I.
Let $\alpha \colon G \to \Aut (A)$
be an action of a finite group $G$ on $A$
and let $x \in A_{+} \setminus \{ 0 \}$.
Then there exists $z \in (A^{\alpha})_{+} \setminus \{ 0 \}$ such that
$z \precsim_A x$.
\end{lem}

\begin{proof}
Set $n = \card (G)$.
By Lemma 2.4 of~\cite{Ph14}, there are
$b_1, b_2, \ldots, b_n \in  A_{+} \setminus \{ 0 \}$
such that,  for $j \neq k$,
\[
b_j b_k = 0,
\quad
b_1 \sim_A b_2 \sim_A \cdots  \sim_A b_n,
\quad
\mbox{and}
\quad
b_1 + b_2 + \cdots + b_n \precsim_A x.
\]
Let $y$ be the direct sum of  $n$ copies of $b_1$.
Using Lemma~2.6 of~\cite{Ph14},
choose $c \in A_{+} \setminus \{ 0 \}$
such that
$c \precsim_A \alpha_{g^{-1}} (b_1)$ for all $g \in G$.
Then $\alpha_g (c) \precsim_A b_1$ for all $g \in G$.
Set $z = \sum_{g\in G} \alpha_g (c)$.
Clearly $z \in ( A^{\alpha} )_{+} \setminus \{ 0 \}$.
Then
\[
z = \sum_g \alpha_g (c)
  \precsim_A \bigoplus_{g \in G} \alpha_g (c)
  \precsim_A y
  \precsim_A b_1 + b_2 + \cdots + b_n
  \precsim_A x,
\]
as desired.
\end{proof}

\begin{lem}\label{Mylemma2019.03.24_wtrp}
Let $A$ be an infinite-dimensional simple unital \ca{} and let
$\alpha \colon G \to \Aut (A)$ be an action
of a finite group $G$ on $A$
which has the weak tracial Rokhlin property.
Let $a \in (K \otimes A)_{+}$
satisfy $a \sim_A \alpha_{g} (a)$ for all $g \in G$
and assume that $0$ is a limit point of $\spec (a)$.
Then for every $\ep > 0$ there exist $\delta > 0$, $m \in \N$,
and $b \in M_m (A^{\alpha})_{+}$ such that
\[
(a - \ep)_{+}
  \precsim_A b \precsim_A (a - \delta)_{+}
\qquad
{\mbox{and}}
\qquad
[0, 1] \subseteq \spec (b).
\]
\end{lem}

\begin{proof}
We may assume $\| a \| = 1$.
Set $n = \card (G)$.

Let $\ep > 0$.
Apply the version of Lemma~\ref{L_9529_ApproxAndSubeq}
which assumes $0$ is a limit point of $\spec (a)$,
and let the notation be as in its conclusion.

By Lemma \ref{L_9529_ApproxAndSubeq}(\ref{L_9529_ApproxAndSubeq_NZero}),
there is $\ld \in \spec (a') \cap (0, \dt)$.
Choose  a \cfn{}
$h \colon [0, \I) \to [0, 1]$
such that $\supp (h) \subseteq (0, \dt)$
and $h (\ld) = 1$.
Use Lemma~\ref{Small_fixed_element}
to choose $d \in M_m (A^{\af})_{+} \setminus \{ 0 \}$
such that $d \precsim_A h (a')$.
Since the action induced by $\af$ on $M_m (A)$
has the weak tracial Rokhlin property
(Corollary 4.6 of~\cite{FG17}),
Lemma~\ref{L_9606_Spec01}
provides $s \in M_m (A^{\af})_{+}$
such that $s \precsim_{A^{\af}} d$ and $\spec (s) = [0, 1]$.

Define
\begin{equation}\label{Eq_wtrp_9519_rh}
\rh = \min \bigl( 1, \, \dt_1 - \dt, \, \dt_3 - \dt_2,
   \, \dt_4 - \dt_3, \, \dt_6 - \dt_5 \bigr),
\end{equation}
\begin{equation}\label{Eq_wtrp_9519_rh_2}
\mu = \frac{\rh^2}{6 n^2},
\andeqn
\ep' = \frac{\rh^3}{36 n^4}.
\end{equation}
For $g \in G$ use Lemma~\ref{A.G.J.P},
Lemma \ref{L_9529_ApproxAndSubeq}(\ref{L_9529_ApproxAndSubeq_21}),
Lemma \ref{L_9529_ApproxAndSubeq}(\ref{L_9529_ApproxAndSubeq_54}),
and~(\ref{Eq_wtrp_9519_rh}) to choose
$v_g, w_g \in M_m (A)$
such that
\begin{equation}\label{Eq_wtrp_2019_03_22}
\bigl\| v_g (a' - \dt)_{+} v_g^*
    - ( \alpha_g (a') - \dt_2)_{+} \bigr\|
< \ep'
\andeqn
\| v_g \| \leq \rh^{-1/2}
\end{equation}
and
\begin{equation}\label{Eq_wtrp1_2019_03_23}
\bigl\| w_g ( \alpha_g (a') - \dt_3)_{+} w_g^*
    - (a' - \dt_5)_{+} \bigr\|
< \ep'
\andeqn
\| w_g \| \leq \rh^{-1/2}.
\end{equation}

Set
\[
F = \bigl\{ v_g, v_g^*, w_{g}, w_{g}^* \colon g \in G \bigr\}
  \cup \bigl\{ (a' - \dt)_{+}, \, (a' - \dt_5)_{+} \bigr\}.
\]
Since the induced action on $M_m (A)$
has the weak tracial Rokhlin property,
we can apply \Lem{invariant_contractions}
with $M_{m} (A)$ in place of $A$, with $F$ as above,
with $\ep'$ in place of~$\ep$, and with $s$ in place of~$x$.
We get positive contractions
$e_g, f_g \in M_m (A)$ for $g \in G$ such that,
with $e = \sum_{g \in G} e_g$ and $f = \sum_{g \in G} f_g$,
the following hold:
\begin{enumerate}
\item\label{WTRP5_e_esff}
$\| e_g e_h \| < \ep'$ and $\| f_g f_h \| < \ep'$
for all $g, h \in G$.
\item\label{WTRP1_b_16}
$\| z e_g - e_g z \| < \ep'$ and $\| z f_g - f_g z \| < \ep'$
for all $g \in G$ and all $z \in F$.
\item\label{WTRP5_e_16}
$(1 - f - \ep' )_{+} \precsim_{A} s$.
\item\label{WTRP2_c_16}
$\alpha_{g} ( e_h ) = e_{g h}$
and $\alpha_{g} ( f_h ) = f_{g h}$ for all $g, h \in G$.
\item\label{WTRP3_d_16}
$e, f \in A^{\alpha}$ and $\| f \| = 1$.
\item\label{WTRP5_e_ef}
$e_g f_g = f_g$ for all $g \in G$.
\end{enumerate}

Define
\[
x = \sum_{g \in G} e_g v_g
\andeqn
y = \sum_{g \in G} f_g w_{g}.
\]
Then
\begin{equation}\label{Eq_wtrp_9520_NormXY}
\| x \|
  \leq \sum_{g \in G} \| v_g \|
  \leq n \rh^{- 1/2}
\andeqn
\| y \|
  \leq \sum_{g \in G} \| w_g \|
  \leq n \rh^{- 1/2}.
\end{equation}
Further define
\[
a_0 = \sum_{g \in G} f_g (a' - \dt_5)_{+} f_g,
\quad
c_0 = \sum_{g \in G} e_g ( \alpha_g (a') - \dt_2)_{+} e_g,
\quad {\mbox{and}} \quad
c = (c_0 - \mu )_{+}.
\]
Then $\| c_0 \| \leq n$ and $\| c \| \leq n$.
Also, $c_0, c \in (A^{\af})_{+}$ by~(\ref{WTRP2_c_16}).

We claim that
\begin{equation}\label{Eq_wtrp_9519_x}
\bigl\| x (a' - \dt)_{+} x^* - c_0 \bigr\| < \mu
\end{equation}
and
\begin{equation}\label{Eq_wtrp_9519_yy}
\bigl\| y c y^* - f (a' - \dt_5)_{+} f \bigr\| < \frac{\rh}{3}.
\end{equation}

We prove~(\ref{Eq_wtrp_9519_x}).
First,
if $g \neq h$
then,
using (\ref{WTRP1_b_16}),
the second part of~(\ref{Eq_wtrp_2019_03_22}),
and (\ref{WTRP5_e_esff}) at the second step,
and~(\ref{Eq_wtrp_9519_rh}) at the last step,
\begin{align*}
\| e_g v_g (a' - \dt)_{+} v_h^* e_h \|
& \leq \| e_g v_g - v_g e_g \| \cdot \| (a' - \dt)_{+} \|
           \cdot \| v_h^* \| \cdot \| e_h \|
\\
& \hspace*{3em} {\mbox{}}
      + \| v_g \| \cdot \| e_g (a' - \dt)_{+} - (a' - \dt)_{+} e_g \|
           \cdot \| v_h^* \| \cdot \| e_h \|
\\
& \hspace*{3em} {\mbox{}}
      + \| v_g \| \cdot \| (a' - \dt)_{+} \|
            \cdot \| e_g v_h^* - v_h^* e_g \| \cdot \| e_h \|
\\
& \hspace*{3em} {\mbox{}}
      + \| v_g \| \cdot \| (a' - \dt)_{+} \|
            \cdot \| v_h^* \| \cdot \| e_g e_h \|
\\
& < \ep' \rh^{- 1/2} + \ep' \rh^{- 1} + \ep' \rh^{- 1/2}
       + \ep' \rh^{- 1}
  \leq 4 \ep' \rh^{- 1}.
\end{align*}
Therefore, using the first part of~(\ref{Eq_wtrp_2019_03_22})
and this estimate at the second step,
\begin{align*}
\bigl\| x (a' - \dt)_{+} x^* - c_0 \bigr\|
& \leq \sum_{g \in G} \| e_g \|
         \cdot \bigl\| v_g (a' - \dt)_{+} v_g^*
             - ( \alpha_g (a') - \dt_2)_{+} \bigr\|
            \cdot \| e_g \|
\\
& \hspace*{3em} {\mbox{}}
          + \sum_{g \neq h} \| e_g v_g (a' - \dt)_{+} v_h^* e_h \|
\\
& < n \ep' + 4 n^2 \ep' \rh^{- 1}
  \leq 5 n^2 \ep' \rh^{- 1}
  < \mu.
\end{align*}
This is~(\ref{Eq_wtrp_9519_x}).

Now we prove~(\ref{Eq_wtrp_9519_yy}).
First,
by~(\ref{Eq_wtrp_9520_NormXY})
and~(\ref{Eq_wtrp_9519_rh_2}),
\begin{equation}\label{Eq_9528_y1}
\| y c y^* - y c_0 y^* \|
  \leq n^2 \rh^{-1} \mu
  = \frac{\rh}{6}.
\end{equation}
Next,
for $g, h, k \in G$ we have,
using (\ref{WTRP1_b_16})
and the second part of~(\ref{Eq_wtrp1_2019_03_23})
at the second step,
\begin{align*}
& \bigl\| f_g w_g e_k ( \alpha_k (a') - \dt_2)_{+} e_k w_h^* f_h
       - f_g e_k w_g ( \alpha_k (a') - \dt_2)_{+} w_h^* e_k f_h \bigr\|
\\
& \hspace*{3em} {\mbox{}}
  \leq \| f_g \| \cdot \| w_g e_k - e_k w_g \|
       \cdot \bigl\| ( \alpha_k (a') - \dt_2)_{+} e_k w_h^* f_h \bigr\|
\\
& \hspace*{6em} {\mbox{}}
     + \bigl\| f_g e_k w_g( \alpha_k (a') - \dt_2)_{+} \bigr\|
          \cdot \| e_k w_h^* - w_h^* e_k \| \cdot \| f_h \|
\\
& \hspace*{3em} {\mbox{}}
  < \ep' \rh^{- 1/2} + \ep' \rh^{- 1/2}
  \leq \frac{\rh}{18}.
\end{align*}
Therefore, by~(\ref{Eq_wtrp_9519_rh_2}),
\begin{equation}\label{Eq_9528_y2_sum}
\Biggl\| y c_0 y^*
    - \sum_{g, h, k \in G}
          f_g e_k w_g ( \alpha_k (a') - \dt_2)_{+} w_h^* e_k f_h
   \Biggr\|
< 2 n^3 \ep' \rh^{- 1/2}
\leq \frac{\rh}{18}.
\end{equation}

Set $S = \{ (g, g, g) \colon g \in G \} \subseteq G^3$.
If $g, k \in G$ are distinct,
then
$\| f_g e_k \| = \| f_g f_k e_k \| < \ep'$
by (\ref{WTRP5_e_ef}) and~(\ref{WTRP5_e_esff}).
Similarly,
if $h, k \in G$ are distinct,
then $\| e_k f_h \| < \ep'$.
In both cases, by~(\ref{Eq_wtrp1_2019_03_23}),
\[
\bigl\| f_g e_k w_g ( \alpha_k (a') - \dt_2)_{+} w_h^* e_k f_h \bigr\|
 < \ep' \rh^{-1}.
\]
So, by~(\ref{Eq_wtrp_9519_rh_2}),
\begin{equation}\label{Eq_9528_y3_CrossTerms}
\Biggl\| \sum_{g, h, k \in G \setminus S}
          f_g e_k w_g ( \alpha_k (a') - \dt_2)_{+} w_h^* e_k f_h
   \Biggr\|
< n^3 \ep' \rh^{- 1}
\leq \frac{\rh}{36}.
\end{equation}
Meanwhile,
by (\ref{WTRP5_e_ef})
and the first part of~(\ref{Eq_wtrp1_2019_03_23}),
\begin{align}\label{Eq_9528_y4_MainTerms}
& \Biggl\| \sum_{g \in G}
          f_g e_g w_g ( \alpha_g (a') - \dt_3)_{+} w_g^* e_g f_g
        - a_0
   \Biggr\|
\\
\notag
& \hspace*{3em} {\mbox{}}
\leq \sum_{g \in G}
   \bigl\| f_g w_g ( \alpha_g (a') - \dt_3)_{+} w_g^* f_g
     - f_g (a' - \dt_5)_{+} f_g \bigr\|
  < n \ep'
  \leq \frac{\rh}{36}.
\end{align}
Finally,
using (\ref{WTRP1_b_16}) and~(\ref{WTRP5_e_esff}),
and using~(\ref{Eq_wtrp_9519_rh_2}) at the last step,
\begin{align}\label{Eq_9528_y5_a0}
& \| a_0 - f (a' - \dt_5)_{+} f \|
\\
\notag
& \hspace*{2em} {\mbox{}}
  \leq \sum_{g \neq h} \| f_g (a' - \dt_5)_{+} f_h \|
\\
\notag
& \hspace*{2em} {\mbox{}}
  \leq \sum_{g \neq h}
        \bigl[ \| f_g (a' - \dt_5)_{+} - (a' - \dt_5)_{+} f_g \|
          \cdot \| f_h \|
        + \| (a' - \dt_5)_{+} \|
           \cdot \| f_g f_h \| \bigr]
\\
\notag
& \hspace*{2em} {\mbox{}}
  < n^2 (\ep' + \ep')
  \leq \frac{\rh}{18}.
\end{align}

Combining (\ref{Eq_9528_y1}),
(\ref{Eq_9528_y2_sum}),
(\ref{Eq_9528_y3_CrossTerms}),
(\ref{Eq_9528_y4_MainTerms}),
and~(\ref{Eq_9528_y5_a0}),
we get
\[
\bigl\| y c y^* - f (a' - \dt_5)_{+} f \bigr\|
 < \frac{\rh}{3},
\]
which is~(\ref{Eq_wtrp_9519_yy}).
The claim is proved.

Define $b = c \oplus s$,
which is in $M_{2 m} (A^{\af})_{+}$.
{}From~(\ref{Eq_wtrp_9519_x})
we get
\[
c = (c_0 - \mu)_{+}
  \precsim_A x (a' - \dt)_{+} x^*
  \precsim_A (a' - \dt)_{+}.
\]
Using $s \precsim_A d \precsim_A h (a')$
and $h (a') \perp (a' - \dt)_{+}$,
as well as
Lemma \ref{L_9529_ApproxAndSubeq}(\ref{L_9529_ApproxAndSubeq_PrDt}),
we have
\[
b \precsim_A (a' - \dt)_{+} \oplus h (a')
  \precsim_A a'
  \precsim_A (a - \dt)_{+}.
\]
Using Lemma \ref{L_9529_ApproxAndSubeq}(\ref{L_9529_ApproxAndSubeq_Ep6})
at the first step,
(\ref{Eq_wtrp_9519_rh}) at the second step,
Lemma~\ref{Lem.ANP.Dec.18} at the third step,
(\ref{Eq_wtrp_9519_yy})
and $\ep' < \frac{\rh}{3}$
(by (\ref{Eq_wtrp_9519_rh_2})) at the fourth step,
and (\ref{WTRP5_e_16}) at the fifth step,
we get
\begin{align*}
(a - \ep)_{+}
& \precsim_A (a' - \dt_6)_{+}
  \leq (a' - \dt_5 - \rh)_{+}
\\
& \precsim_A \Bigl( f (a' - \dt_5)_{+} f - \frac{\rh}{3} \Bigr)_{+}
     \oplus \Bigl( 1 - f - \frac{\rh}{3} \Bigr)_{+}
\\
& \precsim_A y c y^* \oplus (1 - f - \ep' )_{+}
  \precsim_A c \oplus s
  = b.
\end{align*}
The last two relations complete the proof.
\end{proof}

\begin{lem}\label{L_9529_Cu_wtrp_Surjec}
Let $A$ be an infinite-dimensional
stably finite simple unital \ca{} and let
$\alpha \colon G \to \Aut (A)$ be an action
of a finite group $G$ on $A$
which has the weak tracial Rokhlin property.
Recalling the notation of Definition~\ref{D_9421_Pure},
let $a \in (K \otimes A)_{++}$
satisfy $a \sim_A \alpha_{g} (a)$ for all $g \in G$.
Then there exists $b \in (K \otimes A^{\alpha})_{++}$ such that:
\begin{enumerate}
\item\label{Item_9529_Cu_wtrp_Surjec_Same}
$\langle b \rangle_A = \langle a \rangle_A$.
\item\label{Item_9529_Cu_wtrp_Surjec_SupMi}
There are $\et_0, \et_1, \ldots \in \W_{+} (A^{\af})$
such that $\et_0 \leq \et_1 \leq \cdots$ and
$\langle b \rangle_{A^{\af}} = \sup_{n \in \Nz} \et_n$.
\end{enumerate}
\end{lem}

\begin{proof}
By induction on~$n$,
we construct sequences $(\ep_n)_{n \in \Nz}$ in $(0, \infty)$,
$(b_n)_{n \in \Nz}$ in $(K \otimes A^{\af})_{+}$,
and $(m (n) )_{n \in \Nz}$ in~$\N$,
such that $\lim_{n \to \I} \ep_n = 0$,
$b_0 \precsim_A (a - \ep_{0})_{+}$,
and for all $n \in \Nz$ we have
\[
\ep_{n + 1} < \ep_{n},
\qquad
b_n \in \bigl( M_{m (n)} (A)^{\af} \bigr)_{+},
\qquad
(a - \ep_{n})_{+}
  \precsim_A b_{n + 1} \precsim_A (a - \ep_{n + 1})_{+},
\]
and
\[
[0, 1] \subseteq \spec (b_n).
\]

To begin,
set $\ep_0 = 1$.
Given $\ep_n$ with $n \in \Nz$,
apply \Lem{Mylemma2019.03.24_wtrp}
with $\ep_n$ in place of $\ep$,
getting
$\dt > 0$, $m (n + 1) \in \N$,
and $b_{n + 1} \in M_{m (n + 1)} (A^{\alpha})_{+}$
such that
\begin{equation}\label{Eq15_2019_05_14}
(a - \ep_{n})_{+}
  \precsim_A b_{n + 1} \precsim_A (a - \dt)_{+}
\andeqn
[0, 1] \subseteq \spec (b_n).
\end{equation}
Then set $\ep_{n + 1} = \min \bigl( \dt, \frac{\ep_n}{2} \bigr)$.
The induction is complete.

We now have
$b_0
  \precsim_A b_{1}
  \precsim_A b_2
  \precsim_A \cdots$.
Since
$\id_{M_l} \otimes \alpha$ has the
weak tracial Rokhlin property for all $l \in \N$
(by Corollary 4.6 of~\cite{FG17}),
it follows from \Lem{W_T_R_P_inject} that
\begin{equation}\label{Eq21_2019_05_15}
b_0
 \precsim_{A^{\alpha}} b_1
 \precsim_{A^{\alpha}} b_2
 \precsim_{A^{\alpha}} \cdots.
\end{equation}
By Theorem 4.19 of~\cite{APT11},
there exists $b \in (K \otimes A^{\alpha})_{+}$ such that
$\langle b \rangle_{A^{\alpha}}
 = \sup_{n} \, \langle b_n \rangle_{A^{\alpha}}$.
Therefore,
using Theorem 1.16 of~\cite{Ph14}
at the third step,
\begin{align}\label{Eq7_2019_05_29}
\langle b \rangle_{A}
  = \Cu (\iota) \big( \langle b \rangle_{A^{\alpha}}\big)
& = \Cu (\iota) \Big( \sup_{n} \, \langle b_n \rangle_{A^{\alpha}} \Big)
\\
\notag
& = \sup_n \big( \Cu (\iota) \langle b_n \rangle_{A^{\alpha}} \big)
  = \sup_n \, \langle b_n \rangle_{A}.
\end{align}
Moreover,
for all $n \in \Nz$,
we have
$( a - \ep_{n})_{+} \precsim_A b_{n + 1} \precsim_A a$.
Since $\lim_{n \to \I} \ep_n = 0$,
it follows from Lemma 1.25(1) of~\cite{Ph14} that
$\sup_{n} \, \langle b_n \rangle_A = \langle a \rangle_A$.
So $\langle b \rangle_{A} = \langle a \rangle_{A}$,
which is Part~(\ref{Item_9529_Cu_wtrp_Surjec_Same})
of the conclusion.
Part~(\ref{Item_9529_Cu_wtrp_Surjec_SupMi})
follows by taking $\et_n = \langle b_n \rangle_{A^{\af}}$
for $n \in \Nz$.

Finally, we prove that $b \in (K \otimes A^{\alpha})_{++}$.
If not, then (see Definition~\ref{D_9421_Pure})
there is a \pj{} $p \in K \otimes A^{\alpha}$
such that
$\langle b \rangle_{A^{\af}} = \langle p \rangle_{A^{\af}}$.
But then $\langle a \rangle_{A} = \langle p \rangle_{A}$,
contradicting $a \in (K \otimes A)_{++}$.
\end{proof}

Recall the definition of $\Cu_{+} (A)$
(Definition~\ref{D_9421_Pure}).

\begin{thm}\label{T_9529_IsomOnCuPlus}
Let $A$ be an infinite-dimensional
stably finite simple unital \ca{} and let
$\alpha \colon G \to \Aut (A)$ be an action
of a finite group $G$ on $A$
which has the weak tracial Rokhlin property.
Then the inclusion map $\iota \colon A^{\alpha} \to A$
induces an isomorphism of ordered semigroups
$\Cu_{+} (\io) \colon
 \Cu_{+} (A^{\alpha}) \cup \{ 0 \}
  \to \Cu_{+} (A)^{\alpha} \cup \{ 0 \}$.
\end{thm}

By Theorem 4.1(ii) of~\cite{GdlStg},
if $\af$ has the Rokhlin property,
this holds for arbitrary~$A$
and without restricting to the classes of purely positive elements.

\begin{proof}[Proof of Theorem~\ref{T_9529_IsomOnCuPlus}]
It follows from
\Lem{WC_plus_injectivity}(\ref{WC_plus_injectivity_b}),
Lemma \ref{L_9421_PropOfPP}(\ref{Item_L_9421_PropOfPP_ZInSp}),
and simplicity of $A^{\alpha}$
that the map
$\Cu_{+} (\iota) \colon \Cu_{+} (A^{\alpha}) \to \Cu (A)$
is injective
and is an order isomorphism onto its range.
It is trivial that the range is contained in $\Cu (A)^{\alpha}$,
it follows from
Lemma \ref{L_9421_PropOfPP}(\ref{Item_L_9421_PropOfPP_ZInSp})
that the range is contained in $\Cu_{+} (A)$,
and it follows from \Lem{L_9529_Cu_wtrp_Surjec}
that the range contains $\Cu_{+} (A)^{\alpha}$.
So the range is $\Cu_{+} (A)^{\alpha}$.
The extension to
$\Cu_{+} (\io) \colon
 \Cu_{+} (A^{\alpha}) \cup \{ 0 \}
  \to \Cu_{+} (A)^{\alpha} \cup \{ 0 \}$
is immediate.
\end{proof}

\begin{cor}\label{C_9529_WPlusAndSR1}
Let $A$ be an infinite-dimensional simple unital \ca.
Let $\alpha \colon G \to \Aut (A)$ be an action
of a finite group $G$ on $A$
which has the weak tracial Rokhlin property.
Assume that $A^{\af}$ has stable rank one.
Then the inclusion map $\iota \colon A^{\alpha} \to A$
induces an isomorphism of ordered semigroups
$\W_{+} (\io) \colon
 \W_{+} (A^{\alpha}) \cup \{ 0 \}
  \to \W_{+} (A)^{\alpha} \cup \{ 0 \}$.
\end{cor}

It is presumably true that
if $A$ is an infinite-dimensional stably finite simple unital \ca{}
with stable rank one,
$G$ is a finite group,
and $\alpha \colon G \to \Aut (A)$
has the weak tracial Rokhlin property,
then $C^* (G, A, \af)$
and $A^{\af}$ have stable rank one.
However,
this has not been proved,
and a proof presumably requires methods like those in~\cite{ArPh}.
It is known that if
$\alpha$ has the tracial Rokhlin property,
then $C^* (G, A, \af)$ has stable rank one.
This is claimed in Theorem~3.1 of~\cite{FnFg}.
We could not follow the proof there,
but a proof will appear in~\cite{Glsn}.
In this case,
$A^{\af}$ has stable rank one
by Lemma \ref{Fixedpoint_corner}(\ref{Fixedpoint_corner_tsr}).

We need the following fact.
It is part of Theorem 5.15 of~\cite{APT11},
except that we omit the separability hypothesis there.
That hypothesis isn't actually needed for the proof given there.
(The statement in~\cite{APT11} omits ``nondecreasing'',
but, as one sees from the proof, this hypothesis is intended.)

\begin{prp}\label{P_9728_NonsepSupW}
Let $A$ be a unital \ca{} with stable rank one.
Let $(\et_n)_{n \in \Nz}$ be
a bounded nondecreasing sequence in $\W (A)$.
Let $\et = \sup_{n \in \Nz} \et_n$,
evaluated in $\Cu (A)$.
Then $\et \in \W  (A)$.
\end{prp}

\begin{proof}
If  $A$ is separable,
this is contained in Theorem 5.15 of~\cite{APT11}.
The only use of separability
in the proof of that theorem
is in the use of Lemma 5.13 of~\cite{APT11}.
One needs to know that the algebra $A_{\infty}$
in that proof
has a strictly positive element.
It is enough to show that
$A_{\infty}$ has a countable approximate identity,
which follows from the fact that,
using the notation there,
$A_{\infty}$ is the
countable increasing union
of subalgebras $A_{a_n} = {\overline{ a_n A a_n}}$,
each of which clearly has a countable approximate identity.
\end{proof}

\begin{proof}[Proof of Corollary~\ref{C_9529_WPlusAndSR1}]
Since $C^* (G, A, \af)$ is simple,
Theorem~2.8  of~\cite{Brwn}
and Lemma \ref{Fixedpoint_corner}(\ref{Fixedpoint_corner_d})
imply that $A^{\af}$ is stably isomorphic to $C^* (G, A, \af)$.
The algebra $A^{\af}$ is stably finite since it has stable rank one,
so $C^* (G, A, \af)$ is stably finite,
and therefore its subalgebra~$A$ is stably finite.
It now follows from Theorem~\ref{T_9529_IsomOnCuPlus}
that
$\W_{+} (\io) \colon
 \W_{+} (A^{\alpha}) \cup \{ 0 \} \to \W_{+} (A) \cup \{ 0 \}$
is an order isomorphism from $\W_{+} (A^{\alpha}) \cup \{ 0 \}$
to some subsemigroup of $\Cu_{+} (A) \cup \{ 0 \}$
which is contained in
$\bigl( \W_{+} (A) \cap \Cu_{+} (A)^{\alpha} \bigr) \cup \{ 0 \}
 = \W_{+} (A)^{\alpha} \cup \{ 0 \}$.

Now let
$\et \in
 \bigl( \W_{+} (A) \cap \Cu_{+} (A)^{\alpha} \bigr) \cup \{ 0 \}$;
we show that $\et$ is in the range of $\W_{+} (\io)$.
This is trivial if $\et = 0$.
Otherwise, choose $m \in \N$ and $a \in M_{m} (A)_{+}$
such that $\langle a \rangle_A = \et$.
Apply Lemma~\ref{L_9529_Cu_wtrp_Surjec} to~$a$,
getting $b \in (K \otimes A^{\af})_{++}$
and a nondecreasing sequence $(\et_n)_{n \in \Nz}$
in $\W_{+} (A^{\af})$
such that $\langle b \rangle_{A^{\af}} = \sup_{n \in \Nz} \et_n$.
This sequence is bounded
by $\langle 1_{M_m (A^{\af})} \rangle_{A^{\af}}$.
So $\langle b \rangle_{A^{\af}} \in \W (A^{\alpha})$
by Proposition~\ref{P_9728_NonsepSupW},
and $\langle b \rangle_{A^{\af}} \in \Cu_{+} (A^{\af})$
by
Lemma \ref{L_9421_PropOfPP}(\ref{Item_L_9421_PropOfPP_ClosedUnderSups}).
The conclusion follows.
\end{proof}

\begin{cor}\label{C_9720_CuCP_wtrp}
Let $A$ be an infinite-dimensional
stably finite simple unital \ca{} and let
$\alpha \colon G \to \Aut (A)$ be an action
of a finite group $G$ on $A$
which has the weak tracial Rokhlin property.
Then  
\[
\Cu_{+} \bigl( \CGAa \bigr) \cup \{ 0 \}
 \cong \Cu_{+} (A)^{\alpha} \cup \{ 0 \}
\]
as ordered semigroups.
If $A^{\af}$ has stable rank one,
then
\[
\W_{+} \bigl( \CGAa \bigr) \cup \{ 0 \}
 \cong \W_{+} (A)^{\alpha} \cup \{ 0 \}
\]
as ordered semigroups.
\end{cor}

\begin{proof}
It suffices to prove that
$\Cu_{+} \bigl( \CGAa \bigr) \cong \Cu_{+} (A)^{\alpha}$
and, in the stable rank one case,
that
$\W_{+} \bigl( \CGAa \bigr) \cong \W_{+} (A)^{\alpha}$.

Lemma \ref{Fixedpoint_corner}(\ref{Fixedpoint_corner_d})
and simplicity of
$\CGAa$ (Corollary~3.3 of \cite{FG17})
imply that $A^{\alpha}$ is isomorphic to a full corner of $\CGAa$.
Since $A^{\alpha}$ and $\CGAa$
are both unital,
it is easy to check that there is $n \in \N$ such that $\CGAa$
is isomorphic to a full corner of $M_n (A^{\alpha})$.
Therefore
$M_{\I} (\CGAa) \cong M_{\I} (A^{\alpha})$.
In particular,
$K \otimes A^{\alpha} \cong K \otimes \CGAa$.
Using Theorem~\ref{T_9529_IsomOnCuPlus} at the second step,
we get
\[
\Cu_{+} \bigl( \CGAa \bigr)
 \cong \Cu_{+} (A^{\alpha})
 \cong \Cu_{+} (A)^{\alpha}.
\]
When $A^{\af}$ has stable rank one,
the isomorphism $\W_{+} \bigl( \CGAa \bigr) \cong \W_{+} (A)^{\alpha}$
follows similarly,
using Corollary~\ref{C_9529_WPlusAndSR1}
and $M_{\I} (\CGAa) \cong M_{\I} (A^{\alpha})$.
\end{proof}

There is an analog of Corollary~\ref{C_9529_WPlusAndSR1}
for Rokhlin actions on unital \ca{s},
whose proof uses Theorem 4.1(ii) of~\cite{GdlStg}
instead of our Theorem~\ref{T_9529_IsomOnCuPlus}.

\begin{prp}\label{C_9524_WAndSR1}
Let $A$ be a unital \ca{} with stable rank one.
Let $\alpha \colon G \to \Aut (A)$ be an action
of a finite group $G$ on $A$ which has the Rokhlin property.
Then the inclusion map $\iota \colon A^{\alpha} \to A$
induces an isomorphism of ordered semigroups
$\W (\io) \colon \W (A^{\alpha}) \to \W (A)^{\alpha}$.
\end{prp}

We need Proposition 4.1(1) of~\cite{OskPhl3},
but without the separability hypothesis there.
We give an easy proof directly from Theorem~3.2 of~\cite{OskPhl3}.

\begin{prp}\label{P_9728_NonsepTsr1}
Let $A$ be a unital \ca{} with stable rank one.
Let $\alpha \colon G \to \Aut (A)$ be an action
of a finite group $G$ on $A$ which has the Rokhlin property.
Then $C^* (G, A, \af)$ has stable rank one.
\end{prp}

\begin{proof}
Let $a \in C^* (G, A, \af)$ and let $\ep > 0$.
Use Theorem~3.2 of~\cite{OskPhl3}
to choose a \pj{} $f \in A$,
an integer $n \in \N$,
a unital \hm{} $\ph \colon M_n (f A f) \to C^* (G, A, \af)$,
and an element $b \in M_n (f A f)$
such that $\| \ph (b) - a \| < \frac{\ep}{2}$.
Combining Theorem 3.1.8 and
Theorem 3.1.9(1) of~\cite{LnBook},
we see that $M_n (f A f)$ has stable rank one.
Choose $c \in M_n (f A f)$
such that $c$ is invertible and $\| c - b \| < \frac{\ep}{2}$.
Then $\ph (c)$ is an invertible element of $C^* (G, A, \af)$
such that $\| \ph (c) - a \| < \ep$.
\end{proof}

\begin{proof}[Proof of Proposition~\ref{C_9524_WAndSR1}]
The algebra $C^* (G, A, \af)$ has stable rank one
by Proposition~\ref{P_9728_NonsepTsr1}.
It now follows from
Lemma \ref{Fixedpoint_corner}(\ref{Fixedpoint_corner_tsr})
that $A^{\af}$ has stable rank one.

Theorem 4.1(ii) of~\cite{GdlStg}
implies that $\W (\io) \colon \W (A^{\alpha}) \to \W (A)$
is an order isomorphism from $\W (A^{\alpha})$
to some subsemigroup of $\Cu (A)$,
which is necessarily
contained in $\W (A) \cap \Cu (A)^{\alpha} = \W (A)^{\alpha}$.

Let $\et \in \W (A) \cap \Cu (A)^{\alpha}$;
we need to show that $\et$ is in the range of $\W (\io)$.
Choose $m \in \N$ and $a \in M_{m} (A)_{+}$
such that $\langle a \rangle_A = \et$.
Since $\et \in \Cu (A)^{\alpha}$,
by Theorem 4.1(ii) of~\cite{GdlStg}
there is $b \in (K \otimes A^{\af})_{+}$
such that $\langle b \rangle_{A} = \et$.
The case $\et = 0$ is trivial,
so \wolog{} $\| b \| = 1$.
We now construct, by induction on~$n$,
sequences $(\ep_n)_{n \in \Nz}$ in $(0, \infty)$,
$(b_n)_{n \in \Nz}$ in $(K \otimes A^{\af})_{+}$,
and $(m (n) )_{n \in \Nz}$ in~$\N$,
such that $\lim_{n \to \I} \ep_n = 0$,
$b_0 \precsim_A (a - \ep_{0})_{+}$,
and for all $n \in \N$ we have
\[
\ep_{n + 1} < \ep_{n},
\quad
b_n \in \bigl( M_{m (n)} (A)^{\af} \bigr)_{+},
\quad {\mbox{and}} \quad
(b - \ep_{n})_{+}
  \precsim_{A^{\af}} b_{n + 1}
  \precsim_{A^{\af}} (b - \ep_{n + 1})_{+}.
\]
To begin,
set $\ep_0 = 1$ and $b_0 = 0$.
Given $\ep_n$ with $n \in \Nz$,
set $\ep_{n + 1} = \frac{\ep_n}{3}$.
Choose $m (n + 1) \in \N$,
and $c_{n + 1} \in M_{m (n + 1)} (A^{\alpha})_{+}$
such that $\| c_{n + 1} - b \| < \ep_{n + 1}$.
Two applications of Lemma \ref{PhiB.Lem_18_4}(\ref{Item_9420_LgSb_1_6})
give
\[
(b - \ep_{n})_{+}
  = (b - 3 \ep_{n + 1})_{+}
  \precsim_{A^{\af}} (c_{n + 1} - 2 \ep_{n + 1})_{+}
  \precsim_{A^{\af}} (b - \ep_{n + 1})_{+}.
\]
Set $b_{n + 1} = (c_{n + 1} - 2 \ep_{n + 1})_{+}$.
The induction is complete.

For $n \in \N$,
set $\et_n = \langle b_n \rangle_{A^{\af}}$,
which is in $\W (A^{\af})$.
Then $(\et_n)_{n \in \Nz}$
is a nondecreasing sequence in $\W (A^{\af})$
and, by Lemma 1.25(1) of~\cite{Ph14},
we have
$\sup_{n \in \Nz} \et_n = \langle b \rangle_{A^{\af}} = \et$.
This sequence is bounded
by $\langle 1_{M_m (A^{\af})} \rangle_{A^{\af}}$,
so Proposition~\ref{P_9728_NonsepSupW}
now implies $\langle b \rangle_{A^{\af}} \in \W (A^{\alpha})$.
\end{proof}


\section{An example}\label{Sec_Ex}

We give an example of a simple AH~algebra~$A$
with $\rc (A) > 0$ and an
action $\alpha \colon \Z / 2 \Z \to \Aut (A)$ which
has the Rokhlin property.
As discussed in the introduction,
it is not a priori obvious that such examples
should exist,
even with the weak tracial Rokhlin property
in place of the Rokhlin property.
In our example,
we get equality in Theorem~\ref{Main.Thm1}
and Theorem~\ref{T_9412_RcCrPrd}.
See Theorem~\ref{rc.ctn.11} and Theorem~\ref{RC.CP_1}.
The algebras $A$ and $A^{\af}$ have stable rank one
(Lemma~\ref{L_9308_Simple} and Corollary~\ref{C_9605_tsr1}),
and the maps
$\W (A^{\af}) \to \W (A)^{\af}$ and $\Cu (A^{\af}) \to \Cu (A)^{\af}$
are isomorphisms (Corollary~\ref{C_9422_IsomsInExample}).

The construction is motivated by~\cite{HP19},
in which two AH~direct systems with simple direct limits
are ``merged'' into a single larger system
whose direct limit is still simple but which is ``not very far''
from the direct sum of the two original direct limits.
The ``merger'' is accomplished
by writing the systems side by side,
and inserting a very small number of point evaluation maps
which go from one of the original systems to the other.
In~\cite{HP19},
the essential point was that the two systems were very different
but that the base spaces were all contractible.
Here, we use two copies of the same system.
Writing the direct system sideways,
our combined system looks like the following diagram,
in which the solid arrows represent many partial maps
and the dotted arrows represent a small number of point evaluations:
\[
\xymatrix{
C (X_1) \otimes M_{r (1)} \ar@{.>}[dr]
  \ar[r] \ar@<.5ex>[r] \ar@<1ex>[r]
 & C (X_2) \otimes M_{r (2)}\ar@{.>}[dr]
   \ar[r] \ar@<.5ex>[r] \ar@<1ex>[r]
 & C (X_3) \otimes M_{r (3)} \ar@{.>}[dr]
    \ar[r] \ar@<.5ex>[r] \ar@<1ex>[r] & \cdots
     \\
C (X_1) \otimes M_{r (1)} \ar@{.>}[ur]  \ar[r] \ar@<-.5ex>[r]
 \ar@<-1ex>[r] & C (X_2) \otimes M_{r (2)} \ar@{.>}[ur]
   \ar[r] \ar@<-.5ex>[r]
   \ar@<-1ex>[r]& C (X_3) \otimes M_{r (3)} \ar@{.>}[ur]
    \ar[r] \ar@<-.5ex>[r] \ar@<-1ex>[r] & \cdots.}
\]
The order two automorphism
exchanges the two rows.

Since we don't care about contractibility,
we can use products of copies of $S^2$
instead of cones over such spaces as in~\cite{HP19}.
We compute the radius of comparison exactly,
instead of just giving bounds as is done in~\cite{HP19}.

To keep the notation simple,
we carry out only the case of $\Z / 2 \Z$
and radius of comparison less than~$1$.
Modifications of the construction
will presumably work for any finite group
and give any value of the radius of comparison
in $[0, \infty]$.

\begin{ctn}\label{Construction_11}
We define the following objects:
\begin{enumerate}
\item\label{Construction_111}
For $n \in \Nz$, define
\begin{itemize}
\item
$d (n) = 2^{n + 1}-1$.
\item
$l(n) = 2^{n + 1}$.
\item
$r (0) = 1$ and $r (n) = \prod_{k = 1}^n 2^{k + 1}$.
\item
$s (0) = 1$ and $s (n) = \prod_{k = 1}^n (2^{k + 1}-1)$.
\item
$u (n)
 = \frac{s (n)}{r (n)}
 = \prod_{k = 1}^n \big(1 - \frac{1}{2^{k + 1}} \big)$.
\item
$t(0) = 0$ and $t (n + 1) = d (n + 1) t (n) + [r (n) - t (n)]$.
\end{itemize}
\item\label{Construction_112}
Define
$\kp = \lim_{n \to \infty} u (n)$.
\item\label{Construction_113}
For $n \in \Nz$,
define a compact space by
$X_{n} = (S^{2})^{s (n)}$.
Then the covering dimension of $X_n$
is $\dim ( X_{n} ) = 2 s (n)$.
\item\label{Construction_114}
For $n \in \Nz$
and $\nu = 1, 2, \ldots, d (n + 1)$,
let $P^{(n)}_{\nu} \colon X_{n + 1} \to X_n$
be the $\nu$~coordinate projection.
\item\label{Construction_115}
Choose points $x_m \in X_m$
for $m \in \Nz$ such that
for all $n \in \Nz$,
the set
\begin{align*}
&
\Bigl\{ \bigl( P^{(n)}_{\nu_{1}} \circ P^{(n + 1)}_{\nu_{2}}
      \circ \cdots \circ P^{(m - 1)}_{\nu_{m - n}} \bigr) (x_m) \colon
\\
& \hspace*{1em} {\mbox{}}
       {\mbox{$m = n, \, n + 1, \, \ldots$
       and $\nu_j = 1, 2, \ldots, d (n + j)$
       for $j = 1, 2, \ldots, m - n$}} \Bigr\}
\end{align*}
is dense in $X_n$.
(The contribution to this set when $m = n$ is~$x_n$.)
\item\label{Construction_116}
For $n \in \Nz$, define
\[
A_{n}
 = \bigl[ C (X_{n}) \oplus C (X_{n}) \bigr] \otimes M_{r (n)}.
\]
When convenient,
we identify $A_n$ in the obvious ways with
\[
C (X_{n}, M_{r (n)}) \oplus C (X_{n}, M_{r (n)})
\andeqn
C \bigl( X_n \amalg X_n, \, M_{r (n)} \bigr).
\]
\item\label{Construction_117}
For $n \in \Nz$,
define a unital \hm
\[
\lambda_n \colon
C (X_n) \oplus C (X_n)
\to M_{l (n + 1)} \bigl[ C (X_{n + 1}) \oplus C (X_{n + 1}) \bigr]
\]
by
\begin{align}\label{lamdaDefinition}
\lambda_{n} (f, g)
 & = \Big( \diag \big(
      f \circ P^{(n)}_1, \, f \circ P^{(n)}_2, \, \ldots, \,
      f \circ P^{(n)}_{d (n + 1)}, \, g (x_n)
    \big),
\\
& \qquad \qquad {\mbox{}}
  \diag \big( g \circ P^{(n)}_1, \, g \circ P^{(n)}_2, \, \ldots, \,
   g \circ P^{(n)}_{d (n + 1)}, \,  f (x_n) \big) \Big).
 \notag
\end{align}
\item\label{Construction_117'}
For $n \in \Nz$,
define
$\Lambda_{n + 1, \, n} \colon A_n \to A_{n + 1}$
by
$\Lambda_{n + 1, \, n} = \lambda_n \otimes \id_{M_{r (n )}}$.
Thus,
\[
\Lambda_{n + 1, n} \colon
   [C (X_n) \oplus C (X_n)] \otimes M_{ r (n) }
     \to  [ C (X_{n + 1}) \oplus C (X_{n + 1}) ] \otimes M_{r (n + 1)}
\]
is given by
\begin{align*}
& ( f, g ) \otimes c \mapsto
\\
&  \hspace*{2em} {\mbox{}}
\left( \begin{matrix}
   \bigl( f \circ P^{(n)}_{1}, \, g \circ P^{(n)}_{1} \bigr) & & & 0 \\
    & &  \ddots      &      \\
    & & & \bigl( f \circ P^{(n)}_{d (n + 1)}, \,
              g \circ P^{(n)}_{d (n + 1)} \bigr) \\
 0  & & & &  \bigl( g (x_n), f (x_n) \bigr)
 \end{matrix} \right)
\otimes c
\end{align*}
for $f, g \in C (X_{n})$ and $c \in M_{r (n)}$.
Using standard matrix unit notation,
we can also write this definition as
\begin{align}\label{Eq_9616_Formal}
& \Lambda_{n + 1, n} \bigl( ( f, g ) \otimes c \bigr)
\\
\notag
& \hspace*{2em} {\mbox{}}
 = \sum_{j = 1}^{d (n + 1)}
     \bigl( f \circ P^{(n)}_{j}, \, g \circ P^{(n)}_{j} \bigr)
       \otimes e_{j, j} \otimes c
\\
\notag
& \hspace*{4em} {\mbox{}}
   + \bigl( g (x_n) \cdot 1_{C (X_{n + 1})},
          \, f (x_n) \cdot 1_{C (X_{n + 1})} \bigr)
             \otimes e_{d (n + 1) + 1, \, d (n + 1) + 1} \otimes c.
\end{align}

For $m, n \in \Nz$ with $m \leq n$,
now define
\[
\Lambda_{n, m}
= \Lambda_{n, n - 1} \circ \Lambda_{n - 1, \, n - 2} \circ \cdots
          \circ \Lambda_{m + 1, m}
   \colon A_m \to A_n.
\]
\item\label{Construction_118}
Define
\[
A = \dirlim \big( A_n, \, (\Lambda_{m, \, n})_{m \geq n} \big).
\]
For $n \in \Nz$, it is clear that $\Lambda_{n + 1, \, n}$
is an injective unital homomorphism.
Let $\Lambda_{\infty, n} \colon A_n \to A$
be the standard map associated with the direct limit.
\item\label{Construction_119}
Write $\Z / 2 \Z = \{ 0, 1 \}$.
For $n \in \Nz$,
define
$\alpha^{(n)} \colon \Z / 2 \Z \to \Aut (A_n)$  by
\[
\alpha^{(n)}_{1} \big( (f, \, g) \otimes c \big)
 = (g, \, f) \otimes c
\]
for $f,g \in C (X_{n})$ and  $c \in M_{r (n)}$.
We also write $\alpha^{(n)}$ for
the generating automorphism $\alpha^{(n)}_{1}$.
We then have the following diagram:
\begin{equation}\label{Eq_9411_LimDiag}
\begin{CD}
A_0
@>{\Lambda_{1, \, 0}}>>
A_1
 @>{\Lambda_{2, \, 1}}>>
A_2
@>{\Lambda_{3, \, 2}}>>
A_3
@>{}>>
 \cdots   \\
@V{\alpha^{(0)}}VV  @V{\alpha^{(1)}}VV
@V{\alpha^{(2)}}VV     @V{\alpha^{(3)}}VV      \\
A_0
@>{\Lambda_{1, \, 0}}>>
A_1
@>{\Lambda_{2, \, 1}}>>
A_2
@>{\Lambda_{3, \, 2}}>>
A_3
@>{}>>
 \cdots.
\end{CD}
\end{equation}
\end{enumerate}
\end{ctn}

\begin{lem}\label{L_9409_tnrn}
Assume the notation and choices in
Construction~\ref{Construction_11}.
Then $0 \leq t (n) < r (n)$ for all $n \in \Nz$.
\end{lem}

\begin{proof}
The statement is true for $n = 0$ by definition.
Let $n  \in \Nz$ and assume $0 \leq t (n) < r (n)$.
Then
\[
t (n + 1)
 = d (n + 1) t (n) + [r (n) - t (n)]
 = [d (n + 1) - 1] t (n) + r (n),
\]
which implies
(using $d (n + 1) - 1 = 2^{n + 2} - 2 \geq 0$)
\[
0 \leq t (n + 1)
  \leq [d (n + 1) - 1] r (n) + r (n)
  < r (n + 1).
\]
So $0 \leq t (n) < r (n)$ for all $n \in \Nz$ by induction.
\end{proof}

\begin{lem}\label{L_9409_StrDecr}
Assume the notation and choices in
Construction~\ref{Construction_11}.
Then $(u (n))_{n \in \Nz}$ is strictly decreasing
and $0 < \kp < 1$.
\end{lem}

\begin{proof}
The first statement is clear, as is $\kp < 1$.

To prove that $\kp > 0$,
we first observe that if $\bt_1, \bt_2 \in [0, 1]$
then $(1 - \bt_1) (1 - \bt_2) \geq 1 - \bt_1 - \bt_2$.
Induction gives an analogous statement for $n$~factors,
so that, in particular,
$u (n) \geq 1 - \sum_{k = 1}^n \frac{1}{2^{k + 1}}$.
Letting $n \to \I$, we get $\kp \geq \frac{1}{2}$.
\end{proof}

\begin{lem}\label{commute1}
In Construction \ref{Construction_11}(\ref{Construction_119}),
the diagram~(\ref{Eq_9411_LimDiag}) commutes.
Moreover,
there is a unique action
$\af \colon \Z / 2 \Z \to \Aut (A)$
such that $\af = \dirlim \af^{(n)}$,
and this action has the Rokhlin property..
\end{lem}

\begin{proof}
For the first statement,
let $n \in \Nz$.
Using
\ref{Construction_11}(\ref{Construction_117})
in the second step
and \ref{Construction_11}(\ref{Construction_119}) in the third step,
for all $f, g \in C (X_{n})$ and
for all $c \in M_{r (n)}$ we have
\begin{align*}
& \big( \alpha^{(n + 1)} \circ \Lambda_{n + 1, \, n} \big)
    \big( (f, \,g) \otimes c \big)
  = \alpha^{(n + 1)} \big( \lambda_n ( (f, \,g) )
      \otimes c \big)
\\
&
\hspace*{3em} {\mbox{}}
= \Big( \diag \big(
   g \circ P^{(n)}_1, \, g \circ P^{(n)}_2, \, \ldots, \,
   g \circ P^{(n)}_{d (n + 1)}, \,
 f (x_n)
 \big),
\\
& \hspace*{6em} {\mbox{}}
 \diag \big( f \circ P^{(n)}_1, \, f \circ P^{(n)}_2, \, \ldots, \,
   f \circ P^{(n)}_{d (n + 1)}, \,
 g (x_n) \big) \Big) \otimes c
 \\
&
\hspace*{3em} {\mbox{}}
= \big( \Lambda_{n + 1, \, n} \circ \alpha^{(n)} \big)
    \big( (f, \,g) \otimes c \big).
\end{align*}

Existence of $\af$ follows immediately.
It is immediate that $\af^{(n)}$ has the Rokhlin property
for all $n \in \Nz$,
and it follows easily that $\af$ has the Rokhlin property.
\end{proof}

\begin{lem}\label{L_9308_Simple}
Assume the notation and choices in
Construction~\ref{Construction_11}.
Then the \ca~$A$ is stably finite and simple,
and has stable rank one.
\end{lem}

\begin{proof}
Stable finiteness is immediate.
For simplicity, it is easy to check that
the hypotheses of Proposition 2.1(ii) of \cite{DNN92} hold.
For stable rank one,
we observe that the direct system in
Construction \ref{Construction_11}(\ref{Construction_118})
has diagonal maps in the sense of Definition~2.1 of~\cite{ElHoTm}.
Therefore $A$ has stable rank one by Theorem~4.1 of~\cite{ElHoTm}.
\end{proof}

\begin{cor}\label{C_9422_IsomsInExample}
Assume the notation and choices in
Construction~\ref{Construction_11}.
Then the maps
$\Cu (A^{\af}) \to \Cu (A)^{\af}$
and $\W (A^{\af}) \to \W (A)^{\af}$
are isomorphisms.
\end{cor}

\begin{proof}
This follows from
Theorem 4.1(ii) of~\cite{GdlStg}
and Proposition~\ref{C_9524_WAndSR1},
by Lemma~\ref{L_9308_Simple}
and Lemma \ref{commute1}.
\end{proof}

\begin{cor}\label{C_9605_tsr1}
Assume the notation and choices in
Construction~\ref{Construction_11}.
Then $C^* (\Z / 2 \Z, A, \alpha )$ and $A^{\af}$
have stable rank one.
\end{cor}

\begin{proof}
The result for $C^* (\Z / 2 \Z, A, \alpha )$
follows from Lemma~\ref{L_9308_Simple},
Lemma~\ref{commute1},
and Proposition 4.1(1) of~\cite{OskPhl3}.
The result for $A^{\af}$
now follows from
Lemma \ref{Fixedpoint_corner}(\ref{Fixedpoint_corner_tsr}).
\end{proof}

\begin{ntn}\label{N_7806_Bott}
Let $p \in C (S^2, M_2)$ denote the Bott projection, and let
$L$ be the tautological line bundle over
$S^2 \cong \mathbb{C} \mathbb{P}^1$.
(Thus, the range of~$p$ is the section space of~$L$.)
Recall that $X_0 = S^2$.
Assuming the notation and choices in
Construction~\ref{Construction_11},
for $n \in \Nz$
set
\[
p_n = (\id_{M_2} \otimes \Lambda_{n, 0}) (p, 0) \in M_2 (A_n)
\quad {\mbox{and}} \quad
p'_n = (\id_{M_2} \otimes \Lambda_{n, 0}) (p, \, p) \in M_2 (A_n).
\]
In particular, $p_0 = (p, 0)$ and $p_0' = (p, p)$.
\end{ntn}

\begin{lem}[\cite{HP19}]\label{Phillips&illan.1}
The Cartesian product $L^{\times k}$
does not embed in a trivial bundle over $(S^2)^k$
of rank less than $2k$.
\end{lem}

\begin{proof}
This is Lemma~1.9 of~\cite{HP19}.
\end{proof}

\begin{lem}\label{ProjectionRank1}
Assume the notation and choices in
Construction~\ref{Construction_11},
and adopt \Ntn{N_7806_Bott}.
Let $n \in \Nz$.
For $j = 1, 2, \ldots, s (n)$ let
$R_j^{(n)} \colon (S^2)^{s (n)} \to S^2$
be the $j$~coordinate projection.
Then:
\begin{enumerate}
\item\label{ProjectionRank1_a}
There are orthogonal projections
$c^{(0)}_{n}$, $c^{(1)}_{n}, g_{n} \in
M_{2 r (n)} \big( C (X_n) \big)$
such that
\[
p_n = \big( c^{(0)}_{n} + c^{(1)}_{n}, \, g_{n} \big)
\andeqn
\big( \id_{M_2} \otimes \af^{(n)} \big) (p_n)
 = \big( g_{n}, \, c^{(0)}_{n} + c^{(1)}_{n} \big),
\]
$c^{(0)}_{n}$ is the direct sum of the projections
$p \circ R^{(n)}_{j}$ for $j = 1, 2, \ldots, s (n)$,
$c^{(1)}_{n}$ is a constant projection of rank
$r (n) - s (n) - t (n)$,
and
$g_{n}$ is a constant projection of rank $t (n)$.
\item\label{ProjectionRank1_b}
For every $n \in \Nz$ and $\tau \in \T (A_{n})$ we have
$d_{\tau} (p_n) \leq 1$.
\end{enumerate}
\end{lem}

\begin{proof}
We prove the formula in (\ref{ProjectionRank1_a})
for $p_n$ by induction on~$n$.
The formula for $\big( \id_{M_2} \otimes \af^{(n)} \big) (p_n)$
then follows from the definition
of $\af^{(n)}$.

The formula holds for $n = 0$,
since $r (0) = s (0) = 1$,
$t (0) = 0$,
and $r (0) - s (0) - t (0) = 0$.

Now assume that it is known for~$n$.
Recall that
$\Lambda_{n + 1, \, n} = \lambda_n \otimes \id_{M_{r (n )}}$.
(See Construction \ref{Construction_11}(\ref{Construction_117'}).)
We suppress $\id_{M_{2}}$ in the notation.
With this convention,
first take $(f, g)$ in~(\ref{lamdaDefinition})
to be $\big( c_n^{(0)}, 0 \big)$.
The first coordinate
$\Lambda_{n + 1, n} \big( c_n^{(0)}, 0 \big)_1$
is of the form required for $c_{n + 1}^{(0)}$,
while
$\Lambda_{n + 1, n} \big( c_n^{(0)}, 0 \big)_2$
is a constant function of rank $s (n)$.
In the same manner, we see that:
\begin{itemize}
\item
$\Lambda_{n + 1, n} \big( c_n^{(1)}, 0 \big)_1$
is a constant projection of rank
$d (n + 1) [r (n) - s (n) - t (n)]$.
\item
$\Lambda_{n + 1, n} \big( c_n^{(1)}, 0 \big)_2$
is a constant projection of rank
$ r (n) - s (n) - t (n)$.
\item
$\Lambda_{n + 1, n} ( 0, g_n )_1$
is a constant projection of rank
$ t (n)$.
\item
$\Lambda_{n + 1, n} ( 0, g_n )_2$
is a constant projection of rank
$d (n + 1) t (n)$.
\end{itemize}
Putting these together,
we get in the first coordinate of
$\Lambda_{n + 1, n} (p_n)$ the direct sum of
$c_{n + 1}^{(0)}$ as described
and a constant function of rank
\[
d (n + 1) [r (n) - s (n) - t (n)] +  t (n).
\]
A computation shows that this expression is
equal to $r (n + 1) - s (n + 1) - t (n + 1)$.
In the second coordinate we get a
constant projection of rank
\[
s (n) + \big( r (n) - s (n) - t (n) \big) + d (n + 1) t (n)
 = t (n + 1).
\]
This completes the induction.

For Part~(\ref{ProjectionRank1_b}),
we may assume that $\ta$ is extreme in $\T (A_{n})$.
Then there is $x \in X_{n} \amalg X_{n}$  such that
$\tau = \tr_{r (n)} \otimes \ev_{x}$.
Therefore
\[
d_{\tau} (p_n)
  = \tau (p_n)
  = \frac{1}{r (n)} \rank (p_{n} (x))
  = \begin{cases}
        \frac{s (n)}{r (n)}
         + \frac{r (n) - s (n) - t (n)}{r (n)}
              & \qquad x \in X_{n} \amalg \varnothing \\
        \frac{t (n)}{r (n)} & \qquad x \in \varnothing \amalg X_{n}.
\end{cases}
\]
In each case, Lemma~\ref{L_9409_tnrn} implies $d_{\ta} (p_n) \leq 1$.
This completes the proof.
\end{proof}

\begin{lem}\label{ProjectionRank2}
Assume the notation and choices in
Construction~\ref{Construction_11},
and adopt the notation
of \Ntn{N_7806_Bott}.
Let $n \in \Nz$.
For $j = 1, 2, \ldots, s (n)$
let $R_j^{(n)} \colon (S^2)^{s (n)} \to S^2$
be the $j$~coordinate projection.
Then:
\begin{enumerate}
\item\label{ProjectionRank2_a}
There are orthogonal projections $f_{n}, h_{n}$ in
$M_{2 r (n)} \big( C (X_n) \big)$
such that
$p'_n = (f_{n} +h_{n}, \, f_{n} +h_{n})$, $f_{n}$ is
the direct sum of the projections
$p \circ R^{(n)}_{j}$ for $j = 1, 2, \ldots, s (n)$, and
$h_{n}$ is a constant projection of rank
$r (n) - s (n)$.
\item\label{ProjectionRank2_b}
For every $n \in \Nz$ and $\tau \in \T (A_{n})$, we have
$d_{\tau} (p'_n) = 1$.
\end{enumerate}
\end{lem}

\begin{proof}
We prove Part~(\ref{ProjectionRank2_a}).
Using \Lem{ProjectionRank1}(\ref{ProjectionRank1_a}),
Lemma~\ref{commute1}, and the definition of $\af^{(n)}$
in the third step, we get
\begin{align*}
p'_n
& = (\id_{M_2} \otimes \Lambda_{n, 0}) (p, \, p)
  = (\id_{M_2} \otimes \Lambda_{n, 0}) (p, \, 0)
       + (\id_{M_2} \otimes \Lambda_{n, 0}) (0, \, p)
\\
& = \big( c^{(0)}_{n} + c^{(1)}_{n}, \, g_{n} \big)
    + \big( g_{n}, \,c^{(0)}_{n} + c^{(1)}_{n} \big)
  = \big( c^{(0)}_{n} + c^{(1)}_{n} + g_{n},
   \, c^{(0)}_{n} + c^{(1)}_{n} + g_{n} \big).
\end{align*}
Now it is enough to set
$f_{n} = c^{(0)}_{n}$
and $h_{n} = c^{(1)}_{n} + g_{n}$.

For Part~(\ref{ProjectionRank2_b}),
we may assume that $\ta$ is extreme in $\T (A_{n})$.
Then there is $x \in X_{n} \amalg X_{n}$  such that
$\tau = \tr_{r (n)} \otimes \ev_{x}$.
Adding up the ranks given in Part~(\ref{ProjectionRank2_a}),
we see that $\rank (p'_{n} (x)) = r (n)$
for all $x \in X_{n} \amalg X_{n}$.
The conclusion follows.
\end{proof}

\begin{dfn}
Let $A$ be a unital \ca{} and let $p$
be a projection in $M_{\infty} (A)$.
We call $p$ \emph{trivial}
if there is $n \in \Nz$ such that $p$ is Murray-von Neumann
equivalent to $1_{M_{n} (A)}$.
When $n = 0$, this means $p = 0$.
\end{dfn}

\begin{cor}\label{C_7808_BigRank}
Adopt the assumptions and notation of Notation~\ref{N_7806_Bott}.
Let $n \in \Nz$ and
let $e = (e_1, e_2)$ be a projection in
$M_{\infty} (A_n) \cong M_{\infty} (C (X_n) \oplus C (X_n))$
such that $e_1$ is trivial.
If there exists $x \in M_{\infty} (A_n)$
such that $\| x e x^* - p'_{n} \| < \frac{1}{2}$,
then $\rank (e_1) \geq r (n) + s (n)$.
\end{cor}

\begin{proof}
Recall the line bundle $L$
and the projection~$p$ from Notation~\ref{N_7806_Bott}.
Also recall from Definition~\ref{N_9422_MvN}
that we use $\approx$ for Murray-von Neumann equivalence
and $\lessapprox$ for Murray-von Neumann subequivalence.

Let $f_{n}, h_{n} \in M_{2 r (n)} (C (X_n))$
be as in \Lem{ProjectionRank2},
and define $q = f_{n} + h_{n}$.
The range of $f_{n}$ is isomorphic to the section space
of the $s (n)$-dimensional vector bundle
$L^{\times s (n)}$
and
$q (p'_{n} |_{X_{n} \amalg \varnothing} ) q = q$.
Now
$\| x e x^* - p'_{n} \| < \frac{1}{2}$
implies
\[
\bigl\| q (x e x^* |_{X_{n} \amalg \varnothing\ } ) q - q \bigr\|
 < \frac{1}{2}.
\]
Since $e$ and $q$ are projections,
$q \lessapprox e |_{X_{n} \amalg \varnothing} = e_1$.
So there is projection $w \in M_{\infty} (C (X_{n}))$
such that $q + w \approx e_{1}$.
Also, $\| x e x^* - p'_{n} \| < \frac{1}{2}$
implies that $p'_{n}$ is Murray-von Neumann equivalent to
a subprojection of~$e$.
Therefore $h_{n} \lessapprox e_{1}$,
so $\rank (h_n) \leq \rank (e_{1})$.
Take
$e_{0} \in M_{\infty} (C (X_n))$ to be a trivial projection
of rank $\rank(e_{1}) - \rank (h_{n})$ such that $e_{0} \perp h_{n}$.
Since $h_n$ and $e_0$ are trivial,
$e_{0} + h_{n} \approx e_{1}$.
So
\[
f_{n} + h_{n} + w \approx e_{0} + h_{n}.
\]
Define $k = \rank (f_{n} + w)$.
Then $k \geq s (n)$.
Now:
\begin{itemize}
\item
Let $E_{1}$ be a vector bundle whose section space
is isomorphic to the range of $f_n + w$.
\item
Let $E_{2}$ be a trivial vector bundle whose section space
is isomorphic to the range of $e_{0}$.
\item
Set $l = \rank (h_n)$.
\item
Let $H^{l}$ be a trivial vector bundle whose section space
is isomorphic to the range of $h_{n}$.
\end{itemize}
Putting these together and using Theorem~9.1.5 of~\cite{HusB}, we get
$f_{n} + w \approx e_{0}$.
Therefore $f_{n} \lessapprox e_{0}$.
So $\rank ( e_{0}) \geq 2 s (n)$
by Lemma~\ref{Phillips&illan.1}.
Since $e_{0} + h_{n} \approx e_{1}$,
we have
$\rank ( e_{1}) \geq r (n) + s (n)$.
\end{proof}

\begin{rmk}\label{Ni14.remark}
We will use results of Niu from \cite{Ni14}
to obtain an upper bound
on the radius of comparison of our algebra.
Niu introduced a notion of mean dimension
for a diagonal AH-system, \cite[Definition 3.6]{Ni14}.
Suppose we are given
a direct system of homogeneous algebras of the form
\[
A_n = \bigl( C (K_{1,n}) \otimes M_{j_1 (n)} \bigr)
  \oplus \bigl( C (K_{2,n}) \otimes M_{j_2 (n)} \bigr)
   \oplus \cdots \oplus
   \bigl( C (K_{m (n), n}) \otimes M_{j_{m (n) (n)}} \bigr),
\]
in which each of the spaces involved is a connected finite CW complex,
and the connecting maps are unital diagonal maps.
Let $\gamma$ denote the mean dimension of this system,
in the sense of Niu.
It follows trivially from \cite[Definition 3.6]{Ni14} that
\[
\gamma
 \leq \lim_{n \to \infty} \max
  \left(\left\{ \frac{\dim (K_{l,n})}{j_l}
    \colon l = 1, 2, \ldots, m (n) \right\} \right).
\]
Theorem 6.2 of \cite{Ni14} then states that
if $A$ is the direct limit of a system as above,
then $\rc (A) \leq \frac{\gamma}{2}$.
Since the system we are considering here is of this type,
Niu's theorem applies.
\end{rmk}

\begin{thm}\label{rc.ctn.11}
Assume the notation and choices in Construction
\ref{Construction_11} and~\Ntn{N_7806_Bott}.
Then $\rc (A) = \kappa$.
\end{thm}

\begin{proof}
Since
$\limi n \frac{\dim (X_n)}{r (n)} = 2 \kappa $
and the \ca{} $A$ was constructed with
diagonal maps,
we deduce from
\Rmk{Ni14.remark} that
$\rc ( A ) \leq \kappa$.
Now it suffices to prove that
$\rc (A) \geq \kappa$.
Suppose $\rh < \kappa$.
We show that $A$ does not have $\rh$-comparison.
Choose $n \in \N$ such that $1 / {r (n)} < \kappa - \rh$.
Choose $M \in \Nz$ such that
$\rho + 1  < \frac{M}{r (n)} < \kappa + 1$.
Let $e \in M_{\infty} (A_n)$ be a trivial \pj{}
of rank~$M$.
By slight abuse of notation,
we use $\Lambda_{m, n}$ to denote the amplified map
from $M_{\infty} (A_n)$ to $M_{\infty} (A_m)$ as well.
For $m > n$, the rank of $\Lambda_{m, n} (e)$
is $M \cdot \frac{r (m)}{r (n)}$.

We claim that
the rank of $\Lambda_{m, n} (e)$
is strictly less than $r (m) + s (m)$ for $m > n$.
Suppose
$\rank \big( \Lambda_{m, n} (e) \big) \geq r (m) + s (m)$.
Then, by the choice of~$M$,
\[
r (m) + s (m) \leq M \cdot \frac{r (m)}{r (n)} < (\kappa + 1) r (m).
\]
Thus $\frac{s (m)}{r (m)} < \kappa$.
This contradicts Lemma~\ref{L_9409_StrDecr}
and \Ctn{Construction_11}(\ref{Construction_112}).
So the claim follows.

Now, for any tracial state $\tau$ on $A_m$
(and thus for any tracial state on $A$),
we have, using \Lem{ProjectionRank2}(\ref{ProjectionRank1_b})
in the last step,
\begin{equation*}
d_{\tau} (\Lambda_{m, n} (e))
 = \tau (\Lambda_{m, n} (e))
 = \frac{1}{r (m)} \cdot M \cdot \frac{r (m)}{r (n)}
 > 1 + \rh
 = d_{\tau} (p'_{m}) + \rh.
\end{equation*}

On the other hand, if
$\Lambda_{\infty, 0} \big( (p, \, p) \big)
  \lessapprox \Lambda_{\infty, n} (e)$
then, in particular, there exists some $m > n$
and $x \in M_{\infty} (A_m)$
such that $\|x \Lambda_{m, n} (e)x^* - p'_m\| < \frac{1}{2}$.
Using \Cor{C_7808_BigRank}, we have
\[
\rank (\Lambda_{m,n} (e)) \geq r (m) + s (m).
\]
This is a contradiction, and we have proved that
$A$ does not have $\rh$-comparison.
\end{proof}

We now determine the radius of comparison of the crossed product
in our example.
The methods are very similar.

\begin{ctn}\label{Construction_22}
Assume the notation and choices in
Parts (\ref{Construction_111}),
(\ref{Construction_113}),
(\ref{Construction_114}),
and
(\ref{Construction_115}) in Construction
\ref{Construction_11}.

\begin{enumerate}
\item\label{Construction_22_2}
For $n \in \Nz$, we define
$B_{n} = C (X_{n}) \otimes M_{ 2 r (n) }$,
identified with $C (X_{n}, M_2) \otimes M_{r (n) }$.
\item\label{Construction_22_3}
Let $s \in M_2$ be the unitary matrix
$s = \left( \begin{smallmatrix}
  0  &  1 \\
  1  &  0
\end{smallmatrix} \right)$.
Define
$\Lambda'_{n + 1, n} \colon B_{n} \to B_{n + 1}$
by
\[
\Lambda'_{n + 1, n} (f \otimes c)
= \left( \begin{matrix}
   f \circ P^{(n)}_{1} & & & 0 \\
     & f \circ P^{(n)}_{2} &&& \\
    & &  \ddots      &      \\
    & & & f \circ P^{(n)}_{d (n + 1)} \\
  0 & & & & s f (x_{n}) s^*
 \end{matrix} \right)
\otimes c
\]
for $f \in C (X_{n}, M_{2 } )$ and
$c \in M_{r (n)}$.
With abuse of notation
(the expression
$s f (x_n) s^* \cdot 1_{C (X_{n + 1})}$ is the constant
function $X_{n + 1} \to M_2$ with value $s f (x_n) s^*$),
the analog of~(\ref{Eq_9616_Formal}) is
\begin{align}\label{Eq_9616_Formal_CP}
& \Lambda_{n + 1, n}' ( f  \otimes c )
\\
\notag
& \hspace*{3em} {\mbox{}}
 = \sum_{j = 1}^{d (n + 1)}
     f \circ P^{(n)}_{j} \otimes e_{j, j} \otimes c
\\
\notag
& \hspace*{6em} {\mbox{}}
   + s f (x_n) s^* \cdot 1_{C (X_{n + 1})}
             \otimes e_{d (n + 1) + 1, \, d (n + 1) + 1} \otimes c.
\end{align}

It is clear that $\Lambda'_{n + 1, \, n}$
is injective for all $n \in \Nz$.
\item\label{Construction_22_4}
Define
$B = \dirlim (B_n, \, \Lambda'_{n + 1, \, n} )$.
\end{enumerate}
\end{ctn}

\begin{lem}\label{identification.CP1}
Assume the notation and choices in
\Ctn{Construction_11}
and
\Ctn{Construction_22}.
Then
$C^* (\Z / 2 \Z, A, \alpha ) \cong B$.
\end{lem}

\begin{proof}
For $t \in \Z / 2 \Z$,
as in Notation~\ref{N_9408_StdNotation_CP}
let $u_t$ be the standard unitary
in a crossed product by $\Z / 2 \Z$.
(In this proof, no confusion will be caused
by using the same letter in all crossed products.)
For $n \in \Nz$,
there is a \hm{}
\[
\ps_{n + 1, n} \colon
C^* (\Z / 2 \Z, A_n, \, \af^{(n)} \big)
\to
C^* (\Z / 2 \Z, A_{n + 1}, \, \af^{(n + 1)} \big)
\]
such that
\begin{align*}
\ps_{n + 1, n}\big([(f, \, g) \otimes
c ] u_{t} \big)
  = \big[\Lambda_{n + 1, n} ((f, \, g) \otimes c ) \big] u_{t}
  = \big[\lambda_{ n} ( (f, \, g) ) \otimes c \big] u_{t}
\end{align*}
for $f, g \in C (X_{n})$, $t \in \Z / 2 \Z$,
and $c \in M_{r (n)}$.
In view of \Lem{commute1},
we can apply Theorem~9.4.34 of~\cite{GKPT18}
to get an isomorphism
\[
C^* (\Z / 2 \Z, \,A, \, \alpha )
\cong
\dirlim \big( C^* (\Z / 2 \Z, A_n, \, \af^{(n)} \big),
    \, (\ps_{n + 1, n})_{n \in \Nz} \big).
\]
On the other hand,
we have an isomorphism
$\varphi_{n} \colon
   C^* (\Z / 2 \Z, \,A_{n}, \, \alpha^{(n)} ) \to B_{n}$
which is defined
for $f_{m}, g_{m} \in C (X_{n})$
and $c_m \in M_{r (n)}$ for $m = 0, 1$
by
\[
[(f_{0}, \, g_{0}) \otimes c_{0} ] u_{0} +
       [(f_{1}, \, g_{1}) \otimes c_{1} ] u_{1}
\mapsto
\left( \begin{matrix}
f_{0} \otimes c_{0} &  f_{1} \otimes c_{1} \\
g_{1} \otimes c_{1} &  g_{0} \otimes c_{0}
\end{matrix}
\right).
\]
Using matrix unit notation, the right hand side is
\[
f_{0} \otimes e_{1, 1} \otimes c_{0}
 + f_{1} \otimes e_{1, 2} \otimes c_{1}
 + g_{1} \otimes e_{2, 1} \otimes c_{1}
 + g_{0} \otimes e_{2, 2} \otimes c_{0}.
\]
Using (\ref{Eq_9616_Formal}) and~(\ref{Eq_9616_Formal_CP}),
one checks that the diagram
\[
\begin{CD}
C^* (\Z / 2 \Z, \, A_n, \, \alpha^{(n)} )
@>{\psi_{n + 1, \, n}}>>
C^* (\Z / 2 \Z, \, A_{n + 1}, \, \alpha^{(n + 1)} ) \\
@VV{\varphi_{n}}V
    @VV{\varphi_{n + 1}}V \\
B_{n}
@>{\Lambda'_{n + 1, \, n}}>>
B_{n + 1}
\end{CD}
\]
commutes for every $n \in \Nz$.
The result follows.
\end{proof}

\begin{ntn}\label{N_7806_Bott2}
Let $p \in C (X_0, M_2)$ be the Bott projection,
as in Notation~\ref{N_7806_Bott}.
Assuming the notation and choices in
\ref{Construction_22},
for $n \in \Nz$
set $q_n = \Lambda'_{n, 0} (p) \in B_{n}$.
In particular, $q_0 = p$.
\end{ntn}

\begin{lem}\label{ProjectionRank22}
Adopt the assumptions and notation
of Notation~\ref{N_7806_Bott2}.
Let $n \in \Nz$ and for $j = 1, 2, \ldots, s (n)$
let $R_j^{(n)} \colon (S^2)^{s (n)} \to S^2$
be the $j$~coordinate projection.
Then:
\begin{enumerate}
\item
\label{ProjectionRank22_a}
There are orthogonal projections $y_{n}, z_{n}$ in
$M_{2 r (n)} \big( C (X_n) \big)$
such that
$q_n = y_{n} + z_{n}$, $y_{n}$ is
the direct sum of the projections
$p \circ R^{(n)}_{j}$ for $j = 1, 2, \ldots, s (n)$, and
$z_{n}$ is a constant projection of rank
$r (n) - s (n)$.
\item
\label{ProjectionRank22_b}
For every $n \in \Nz$ and $\tau \in \T (B_{n})$, we have
$d_{\tau} (q_n) = \frac{1}{2}$.
\end{enumerate}
\end{lem}

\begin{proof}
The proof of~(\ref{ProjectionRank22_a})
is very similar to that of
Lemma \ref{ProjectionRank2}(\ref{ProjectionRank2_a}),
but simpler because there is only one summand.
The basic facts for the induction step are that
$\Lambda'_{n + 1, n} ( y_{n})$
is the direct sum of the projections
$p \circ R^{(n)}_{j} \circ P^{(n)}_{k}$
for $j = 1, 2, \ldots, s (n)$ and $k = 1, 2, \ldots, d (n + 1)$,
and a constant projection of rank
$s (n)$,
and that $\Lambda'_{n + 1, n} ( z_{n})$
is a constant projection of rank
$l(n + 1) [r (n) - s (n)]$.
We omit the details.

The proof of~(\ref{ProjectionRank22_b})
is essentially the same as that of
Lemma \ref{ProjectionRank2}(\ref{ProjectionRank2_b}),
and is omitted.
\end{proof}

\begin{cor}\label{C_7808_BigRank2}
Adopt the assumptions and notation of Notation~\ref{N_7806_Bott2}.
Let $n \in \Nz$ and let $e$ be a trivial projection in
$M_{\infty} (B_n) \cong M_{\infty} (C (X_n) )$.
If there exists $x \in M_{\infty} (B_n)$
such that $\| x e x^* - q_{n} \| < \frac{1}{2}$ then
$\rank (e) \geq  r (n) + s (n)$.
\end{cor}

\begin{proof}
The proof is essentially the same as that of
Corollary~\ref{C_7808_BigRank}.
We use Lemma \ref{ProjectionRank22}
and the projections $y_n$ and~$z_n$
instead of Lemma~\ref{ProjectionRank2}
and the projections $f_n$ and~$h_n$.
\end{proof}

The next result is the analog of Theorem~\ref{rc.ctn.11}.
It shows that in our example,
we get equality in Theorem~\ref{Main.Thm1}
and Theorem~\ref{T_9412_RcCrPrd}.

\begin{thm}\label{RC.CP_1}
Assume the notation and choices
in Construction~\ref{Construction_11} and
\Ntn{N_7806_Bott}.
Then
\[
\rc \big( C^* (\Z / 2 \Z, \, A, \, \alpha ) \big)
 = \frac{\kappa}{2}
\andeqn
\rc (A^{\af}) = \kp.
\]
\end{thm}

The proof is similar to that of Theorem~\ref{rc.ctn.11}.
We give details to show where the factor $\frac{1}{2}$
comes from,
and for convenient reference in a paper in preparation.

\begin{proof}[Proof of Theorem~\ref{RC.CP_1}]
We prove the first part of the conclusion.
The second part then follows from Theorem~\ref{T_9412_RcCrPrd}.

Because
$C^* (\Z / 2 \Z, \, A, \, \alpha )$
is isomorphic to  the \ca~$B$
by \Lem{identification.CP1}, it suffices to show that
$\rc (B) = \frac{\kappa}{2}$.
Since
$\limi n \frac{\dim (X_n)}{2 r (n)} = \kappa$
and the \ca{} $B$ was constructed with
diagonal maps,
we deduce from \Rmk{Ni14.remark} that $\rc ( B ) \leq \frac{\kappa}{2}$.
Now it suffices to prove that
$\rc (B) \geq \frac{\kappa}{2}$.
Suppose $\rh < \frac{\kappa}{2}$.
We show that $B$ does not have $\rh$-comparison.
Choose $n \in \N$ such that $1 / {r (n)} < \frac{\kappa}{2} - \rh$.
Choose $M \in \Nz$ such that
$\rho + \frac{1}{2}
 < \frac{M}{2 r (n)}
 < \frac{\kappa}{2} + \frac{1}{2}$.
Let $e \in M_{\infty} (B_n)$ be a trivial \pj{}
of rank~$M$.
By slight abuse of notation,
we use $\Lambda'_{m, n}$ to denote the amplified map
from $M_{\infty} (B_n)$ to $M_{\infty} (B_m)$ as well.
For $m > n$, the rank of $\Lambda'_{m, n} (e)$
is $M \cdot \frac{r (m)}{r (n)}$.
We claim that
the rank of $\Lambda'_{m, n} (e)$
is strictly less than $r (m) + s (m)$ for $m > n$.
Suppose
$\rank \big( \Lambda'_{m, n} (e) \big) \geq r (m) + s (m)$.
Then, considering the choice of $M$,
\[
r (m) + s (m) \leq M \cdot \frac{r (m)}{r (n)} < (\kappa + 1) r (m).
\]
Thus $\frac{s (m)}{r (m)} < \kappa$.
This contradicts \Ctn{Construction_11}(\ref{Construction_112}).
So the claim follows.

Now, for any extreme tracial state $\tau$ on $B_m$
(and thus for any trace on $B$),
we have, using \Lem{ProjectionRank22}(\ref{ProjectionRank22_b})
in the last step,
\begin{equation*}
d_{\tau} (\Lambda'_{m, n} (e))
 = \tau (\Lambda'_{m, n} (e))
 = \frac{1}{2 r (m)} \cdot M \cdot \frac{r (m)}{r (n)}
 > \frac{1}{2} + \rh
 = d_{\tau} (q_{m}) + \rh.
\end{equation*}
On the other hand, if
$\Lambda'_{\infty, 0} (p) \lessapprox \Lambda'_{\infty, n} (e)$
then, in particular, there exists some $m > n$
and $x \in M_{\infty} (B_m)$
such that $\|x \Lambda'_{m, n} (e)x^* - q_m\| < \frac{1}{2}$.
Using \Cor{C_7808_BigRank2}, we get
\[
\rank (\Lambda'_{m,n} (e)) \geq r (m) + s (m).
\]
This is a contradiction, and we have proved that
$B$ does not have $\rh$-comparison.
\end{proof}

\begin{exa}\label{R_9412_TrRPNeeded}
We show that,
in Theorem~\ref{Main.Thm1}
and Theorem~\ref{T_9412_RcCrPrd},
the weak tracial Rokhlin property
can't be replaced by pointwise outerness.

Let $A$ and $\af \colon \Z / 2 \Z \to \Aut (A)$
be as in Lemma~\ref{commute1},
set $B = C^* (\Z / 2 \Z, \, A, \, \alpha )$,
and let $\bt = {\widehat{\af}} \colon \Z / 2 \Z \to \Aut (B)$
be the dual action.
It follows from
Theorem~\ref{rc.ctn.11},
Theorem~\ref{RC.CP_1},
and Lemma~\ref{L_9409_StrDecr}
that the inequalities
in Theorem~\ref{Main.Thm1}
and Theorem~\ref{T_9412_RcCrPrd}
fail for the action~$\bt$.

We already know that $B$ is simple,
and $B$ is stably finite because it is an AH~algebra.
It remains to show that $\bt$ is pointwise outer.
Suppose not.
Then in fact $\bt$ is an inner action,
that is, given by conjugation by a unitary of order~$2$.
(See Exercise 8.2.7 of~\cite{GKPT18}.)
So $C^* (\Z / 2 \Z, \, B, \, \bt ) \cong B \oplus B$.
But by Takai duality $C^* (\Z / 2 \Z, \, B, \, \bt ) \cong M_2 (A)$,
which is simple.
Pointwise outerness is proved.
\end{exa}


\section{Open problems}\label{Sec_Q}

The most obvious problem is whether
equality always holds in
Theorem~\ref{Main.Thm1} and Theorem~\ref{T_9412_RcCrPrd}.

\begin{qst}\label{Q_9412_eq}
Let $G$ be a finite group,
let $A$ be
an infinite-dimensional stably finite simple unital C*-algebra,
and let $\alpha \colon G \to \Aut (A)$ be an action of
$G$ on $A$ which has the
weak tracial Rokhlin property.
Does it follow that
\[
\rc (A^{\alpha}) = \rc (A)
\andeqn
\rc \big( \CGAa \big) = \frac{1}{\card (G)} \cdot \rc (A)?
\]
\end{qst}

One might even hope that the reverse inequalities
\begin{equation}\label{Eq_9607_RevIneq}
\rc (A^{\alpha}) \geq \rc (A)
\andeqn
\rc \big( \CGAa \big) \geq \frac{1}{\card (G)} \cdot \rc (A).
\end{equation}
hold without restrictions on the action.
Quite different methods seem to be needed for this question.
Suppose, for example,
that we were able to prove~(\ref{Eq_9607_RevIneq})
for pointwise outer actions.
Suppose $G$ is finite abelian,
$\af \colon G \to \Aut (A)$ is pointwise outer,
and, with $B = \CGAa$,
the dual action
$\bt = {\widehat{\af}} \colon {\widehat{G}} \to \Aut (B)$
is pointwise outer
and $B$ has strict comparison.
We would be able to deduce that
$C^* \bigl( {\widehat{G}}, B, \bt \bigr)$ has strict comparison.
This outcome is at least heuristically related
to the long standing open question
of whether the crossed product of a simple C*-algebra
with stable rank one by a finite group again has stable rank one.
Indeed, if $B$ is classifiable in the sense of the Elliott program,
and the tracial state space has
compact \fd{} extreme boundary,
it would follow
(see Corollary~7.9 of~\cite{KbgRdm4},
Corollary~1.2 of~\cite{Sato3},
or Corollary~4.7 of~\cite{TWW})
that $C^* \bigl( {\widehat{G}}, B, \bt \bigr)$
is ${\mathcal{Z}}$-stable,
and therefore from Theorem 6.7 of~\cite{Rdm7}
that $C^* \bigl( {\widehat{G}}, B, \bt \bigr)$
has stable rank one.
This case of the problem has been solved~\cite{Osk_sr},
but the proof depends on major results
in the classification program.

In the example in Section~\ref{Sec_Ex},
the group action on $\T (A)$ is highly nontrivial.

\begin{qst}\label{Q_9412_TrivOnTA}
Does there exist an action of a nontrivial finite group
with the weak tracial Rokhlin
property on a simple separable unital \ca{} $A$ with $\rc (A) > 0$
and such that every tracial state is invariant?
\end{qst}

One can ask for even more.

\begin{qst}\label{Q_9412_UniqTr}
Does there exist an action of a nontrivial finite group
with the Rokhlin property on a simple
separable unital \ca{} $A$ with $\rc (A) > 0$
and unique tracial state?
\end{qst}

By combining methods of Villadsen~\cite{Vlds}
with those of Section~\ref{Sec_Ex},
one should be able to at least produce an example of
a simple separable unital nuclear \ca~$A$
and an action $\af \colon \Z / 2 \Z \to \Aut (A)$
such that $A$ does not have stable rank one,
$\af$ has the Rokhlin property,
and $A$ has exactly two extreme tracial states,
which are interchanged by the action~$\af$.

\begin{qst}\label{Q_9612_trs_wTRP}
Let $A$ be an infinite-dimensional
simple unital \ca{} with stable rank one,
let $G$ be a finite group,
and let $\af \colon G \to \Aut (A)$
be an action with the weak tracial Rokhlin property.
Does it follow that $\CGAa$ and $A^{\af}$ have stable rank one?
\end{qst}

This is wanted for improvement of Corollary~\ref{C_9529_WPlusAndSR1}.

\begin{qst}\label{Q_9612_Need_tsr1}
Are the stable rank one hypotheses in Proposition~\ref{C_9524_WAndSR1}
and Corollary~\ref{C_9529_WPlusAndSR1}
really necessary?
\end{qst}

That is,
assuming the action
has the Rokhlin property or weak tracial Rokhlin property
as appropriate,
does one get isomorphisms
\[
\W (\io) \colon \W (A^{\alpha}) \to \W (A)^{\alpha}
\qquad {\mbox{or}} \qquad 
\W_{+} (\io) \colon
 \W_{+} (A^{\alpha}) \cup \{ 0 \}
  \to \W_{+} (A)^{\alpha} \cup \{ 0 \},
\]
rather than just
\[
\Cu (\io) \colon \Cu (A^{\alpha}) \to \Cu (A)^{\alpha}
\qquad {\mbox{or}} \qquad 
\Cu_{+} (\io) \colon
 \Cu_{+} (A^{\alpha}) \cup \{ 0 \}
  \to \Cu_{+} (A)^{\alpha} \cup \{ 0 \}?
\]

One possible generalization of the results of this paper
is to the nonunital case.
This will be treated in~\cite{Asd}
(by a different set of authors).
Complications include the additional complexity
of the definition of the weak tracial Rokhlin property
(see Definition~3.1 of~\cite{FG17}),
and what to substitute for the conventional definition
of the radius of comparison.


\section{Acknowledgments}\label{Sec_Ackn}

This research was done while the first author was
a visiting scholar at the University of Oregon
during the period March 2018 to 
September 2019.
He is thankful to that institution
for its hospitality.
He was partially supported by the University of Tehran.
This paper will be part of first author's PhD.\   dissertation.

The research of the third author was partially
supported by the
Simons Foundation Collaboration Grant for Mathematicians
\#587103.

The first author thanks Q.~Wang for pointing out
Corollary II.4.3 in \cite{BH82}.
All three authors would like to thank
M.~Amini, I.~Hirshberg, and  S.~Jamali
for sharing some of their unpublished work with us.

The first author is grateful to M.~B.\  Asadi
for motivating him to study
operator algebras in the first year of his Ph.D.\  program.


\end{document}